\documentclass[11pt]{amsart}

\swapnumbers

\usepackage{amssymb,amsfonts}
\usepackage{euscript}
\usepackage[alphabetic]{amsrefs}

\usepackage[ps,matrix,arrow,curve,cmtip]{xy}

\title{The congruence criterion for power operations in Morava
  $E$-theory}
\author{Charles Rezk}
\date{ \today}
\address{Department of Mathematics \\
University of Illinois at Urbana-Champaign \\ 
Urbana, IL}
\email{rezk@math.uiuc.edu}
\thanks{The author was supported under NSF grant DMS-0505056.}

\numberwithin{equation}{section}

\makeatletter
  \let\c@subsection\c@equation
\makeatother

\theoremstyle{plain}   

\newtheorem{thm}[subsection]{Theorem}
\newtheorem*{thm*}{Theorem}
\newtheorem*{thma}{Theorem A}
\newtheorem*{thmb}{Theorem B}
\newtheorem{prop}[subsection]{Proposition}
\newtheorem{cor}[subsection]{Corollary}
\newtheorem{lemma}[subsection]{Lemma}

\theoremstyle{remark}
\newtheorem{rem}[subsection]{Remark}    
\newtheorem{exam}[subsection]{Example}

\theoremstyle{plain}

\DeclareMathOperator{\id}{id}
\DeclareMathOperator{\colim}{colim}

\DeclareMathOperator{\cok}{cok}
\DeclareMathOperator{\Ker}{Ker}
\DeclareMathOperator{\im}{im}
\newcommand{\op}{{\operatorname{op}}}
\newcommand{\ob}{{\operatorname{ob}}}

\newcommand{\Hom}{{\operatorname{Hom}}}

\newcommand{\ra}{\rightarrow}

\newcommand{\xra}{\xrightarrow}
\newcommand{\la}{\leftarrow}

\newcommand{\xla}{\xleftarrow}

\newcommand{\cat}[1]{{\operatorname{\EuScript{#1}}}}

\DeclareMathOperator{\Ext}{Ext}
\DeclareMathOperator{\Tor}{Tor}

\DeclareMathOperator{\hocolim}{hocolim}
\DeclareMathOperator{\holim}{holim}

\newcommand{\set}[2]{{\{\,#1\mid#2\,\}}}
\newcommand{\tensor}[1]{\underset{#1}{\otimes}}

\newcommand{\powser}[1]{[\![#1]\!]}

\newcommand{\F}{\mathbb{F}}
\newcommand{\Z}{\mathbb{Z}}
\newcommand{\N}{\mathbb{N}}
\newcommand{\R}{\mathbb{R}}
\newcommand{\Q}{\mathbb{Q}}
\newcommand{\C}{\mathbb{C}}

\newcommand{\point}{{\operatorname{pt}}}

\newcommand{\sm}{\wedge} 

\newcommand{\dfn}{\textbf}

\def\defeq{\overset{\mathrm{def}}=}

\begin{document}

\begin{abstract}
  We prove a congruence criterion for the algebraic theory of power
  operations in Morava $E$-theory, analogous to Wilkerson's congruence
  criterion for torsion free $\lambda$-rings.  In addition, we provide
  a geometric description of this congruence criterion, in terms of
  sheaves on the moduli problem of deformations of formal groups
  and Frobenius isogenies.
\end{abstract}

\maketitle

\newcommand{\Spec}{\operatorname{Spec}}
\newcommand{\Alg}[1]{\mathrm{Alg}_{#1}}
\newcommand{\grAlg}[1]{\mathrm{Alg}_{#1}^*}
\newcommand{\Mod}[1]{\mathrm{Mod}_{#1}}
\newcommand{\grMod}[1]{\mathrm{Mod}_{#1}^*}
\newcommand{\zgrMod}[1]{\mathrm{Mod}_{#1}^{\Z*}}
\newcommand{\Modff}[1]{\mathrm{Mod}^{\mathrm{ff}}_{#1}}
\newcommand{\Set}{\mathrm{Set}}
\newcommand{\Cat}{\mathrm{Cat}}
\newcommand{\grCat}{\mathrm{grCat}}
\newcommand{\mor}{\operatorname{mor}}
\newcommand{\Comod}[1]{\mathrm{Comod}_{#1}}

\newcommand{\Sh}{\operatorname{Sh}}
\newcommand{\DefFrob}{\mathrm{Def}}
\newcommand{\Sub}{\mathrm{Sub}}
\newcommand{\ShMod}{\mathrm{Sh}(\DefFrob,\mathrm{Mod})}
\newcommand{\ShAlg}{\mathrm{Sh}(\DefFrob,\mathrm{Alg})}
\newcommand{\grShMod}{\mathrm{Sh}(\DefFrob,\mathrm{Mod})^*}
\newcommand{\grShAlg}{\mathrm{Sh}(\DefFrob,\mathrm{Alg})^*}
\newcommand{\ShAlgCong}{\mathrm{Sh}(\DefFrob,\mathrm{Alg})_{\mathrm{cong}}}
\newcommand{\grShAlgCong}{\mathrm{Sh}(\DefFrob,\mathrm{Alg})^*_{\mathrm{cong}}}
\newcommand{\Pow}{\mathrm{Pow}}
\newcommand{\ShModPow}{\mathrm{Sh}(\Pow,\mathrm{Mod})}

\newcommand{\fulltensorv}[3]{\underset{#1}{\mathstrut^{#2}\otimes\mathstrut^{#3}}}
\newcommand{\fulltensor}[3]{\mathstrut^{#2}\otimes_{#1}\mathstrut^{#3}}

\newcommand{\Frob}{\operatorname{Frob}}
\newcommand{\testrings}{\mathcal{R}}
\newcommand{\hRing}{\widehat{\mathcal{R}}}
\newcommand{\Isogstack}[1]{\mathrm{Isog}^{#1}}
\newcommand{\Isog}[2]{\mathrm{Isog}^{#1}_{#2}}
\newcommand{\univ}{\mathrm{univ}}
\newcommand{\can}{\operatorname{can}}

\newcommand{\E}{\mathcal{E}}
\newcommand{\catC}{\mathcal{C}}

\newcommand{\Ab}{\mathrm{Ab}}

\renewcommand{\O}{\mathcal{O}}
\newcommand{\FuncG}{\mathcal{F}}

\newcommand{\Pfree}{\mathbb{P}}
\newcommand{\cPfree}{\widehat{\mathbb{P}}}
\newcommand{\cE}[1]{E^{\sm}_{#1}}
\newcommand{\rE}{\tilde{E}}
\newcommand{\Fin}{\operatorname{Fin}}

\newcommand{\Prim}{\mathrm{Prim}}
\newcommand{\Ind}{\mathrm{Ind}}

\newcommand{\algapprox}{\mathbb{T}}
\newcommand{\calgapprox}{\widehat{\mathbb{T}}}
\newcommand{\algcat}{\mathrm{Alg}_{\algapprox}}
\newcommand{\gralgcat}{\mathrm{Alg}_{\algapprox}^*}
\newcommand{\gralgcatQ}{\mathrm{Alg}_{\algapprox,\Q}^*}

\newcommand{\Si}{\Sigma^\infty}
\newcommand{\Sip}{\Sigma^\infty_+}
\newcommand{\funcspec}{\underline{\mathrm{hom}}}

\newcommand{\Sym}{\operatorname{Sym}}

\newcommand{\bA}{\overline{A}}

\newcommand{\CP}{\mathbb{CP}}
\newcommand{\bP}{\overline{P}}

\newcommand{\Sqrt}{\operatorname{Sqrt}}


\section{Introduction}

The purpose of this paper is to prove a congruence criterion for the
algebraic theory of power operations acting on the homotopy of a
$K(n)$-local commutative  $E$-algebra spectrum, where $E$ denotes a
spectrum of ``Morava $E$-theory'' at height $n$.   This criterion is
best understood as being a higher chromatic analogue of Wilkerson's
congruence criterion for $\lambda$-rings.  

\subsection{Algebraic theories of power operations}

By an ``algebraic theory of power
operations'' for a commutative ring spectrum $A$, we mean an algebraic
category which models all the algebraic structure which naturally
adheres to $\pi_*A$, the homotopy groups of $A$.  As an examplar of
this notion, consider the description of the natural
operations on the homotopy of a commutative $H\F_p$-algebra spectrum,
as given by McClure (see \cite{bmms-h-infinity-ring-spectra}
especially \S IX.2).  In this work,
it is shown (in modern language) that if $A$ is a commutative
$H\F_p$-algebra, then $\pi_*A$ is an algebra for a certain monad $C$ on graded
$\F_p$-vector spaces.  Furthermore, an algebraic description for
$C$-algebras is provided: a $C$-algebra amounts to a graded
commutative $\F_p$-algebra, together with the structure of a module
over the May-Dyer-Lashof algebra,  which (i) is compatible with
multiplication, in the sense of having a suitable Cartan formula, and
which (ii) satisfies an ``instability'' relation, which includes the
fact that the ``top'' Dyer-Lashof operation is equal to the $p$th power
map.  That this is the right answer is justified by the existence of a
natural isomorphism $C(\pi_*M)\approx \pi_*(\Pfree M)$, where $M$ is an
$H\F_p$-module and $\Pfree$ is the free $H\F_p$-algebra functor.

In this paper, we use the work of Ando, Hopkins and Strickland to
develop an analogous theory for the homotopy of a $K(n)$-local
commutative algebra over a Morava $E$-theory spectrum.  
Let $E$ denote the cohomology theory associated to the universal
deformations of a height $n$ formal group $G_0$ over a perfect field
$k$ of characteristic $p$.  There is a monad $\algapprox$ on the
category $\Mod{E_*}$ of graded $E_*=\pi_*E$-modules such that
\begin{enumerate}
\item $\algapprox E_* \approx \bigoplus_{m\geq0} \cE{*}B\Sigma_m$,
  where $\cE{*}({-})$ denotes $K(n)$-localized homology;
\item $\algapprox (M_*\oplus N_*) \approx \algapprox M_*\otimes_{E_*}
  \algapprox N_*$ for any $E_*$-modules $M_*$ and $N_*$;
\item
There is a natural map $\algapprox (\pi_* M) \ra \pi_* \Pfree M$ for
$E$-module spectra $M$, where $\Pfree M$ denotes the free commutative
$E$-algebra on $M$.  Furthermore, this map induces an isomorphism
$\algapprox (\pi_* M)\approx [\pi_*\Pfree M]_{\mathfrak{m}}^{\sm}$,
where $N^{\sm}_{\mathfrak{m}}$ denotes completion of an $E_*$-module
$N$ with respect to the maximal ideal $\mathfrak{m}\subset \pi_0E$.
\end{enumerate}
The category of algebras for this monad is denoted $\gralgcat$; the
monad $\algapprox$ and its category of algebras are described in
\S\ref{sec:completed-e-theory} and
\S\ref{sec:algebraic-approximation}.  


\subsection{Description of $\gralgcat$}

The next goal is to give a workable description of $\gralgcat$.  Every
object of $\gralgcat$ is (in particular) a graded commutative ring; it
is also a right-module for a certain associative ring $\Gamma$, which
we may call a ``Dyer-Lashof algebra'' (by analogy with the
May-Dyer-Lashof algebra for ordinary mod $p$ homology).

The ring
$\Gamma$ is defined in \S\ref{sec:additive-bialgebra}; effectively, $\Gamma$
is the ring of endomorphisms of the forgetful functor $\gralgcat \ra
\Ab$ which sends an $\algapprox$-algebra to its degree $0$ part,
viewed as an abelian group.  Explicitly, $\Gamma$ is a direct sum of
the $E_0$-linear duals of the rings
$E^0B\Sigma_{p^k}/(\text{transfers})$.

There is a
ring homomorphism $\eta\colon E_0\ra \Gamma$ (the image of which is
not typically central).  The ring $\Gamma$ is very nearly a
Hopf algebra 
(more precisely, it is a ``twisted bialgebra'', see
\S\ref{sec:bialgebras}), so that the  category 
of $\Gamma$-modules admits a
symmetric monoidal structure $\otimes \colon \grMod{\Gamma}\times
\grMod{\Gamma}\ra \grMod{\Gamma}$, in such a way that the underlying
$E_*$-module of $M\otimes N$ is precisely $M\otimes_{E_*}N$.  


Let $\grAlg{\Gamma}$ denote the category of commutative monoid objects
in graded $\Gamma$-modules. 
There is a forgetful functor $U\colon \gralgcat\ra \grAlg{\Gamma}$.
The first main result in this 
paper describes 
the essential image of the restriction of $U$ to torsion free objects.
Before stating the result, we consider the motivating example.

\begin{exam}
Let $E$ be $p$-adic $K$-theory.  In this case,
$E_0\approx \Z_p$, and $\Gamma\approx \Z_p[\psi]$.  The element
$\psi\in \Gamma$ corresponds to the  $p$th
Adams operation in the $K$-theory of a space.   A $\Gamma$-algebra is
precisely a \dfn{$\psi$-ring}, i.e., a graded commutative $\Z_p$-algebra $B$
together with a ring 
homomorphism $\psi\colon B\ra B$.  In this case, the structure of
$\gralgcat$ can also be completely understood, using the work of McClure in
\cite{bmms-h-infinity-ring-spectra}.  In particular, an object in
$\gralgcat$ is essentially  what Bousfield
\cite{bousfield-on-lambda-rings-k-theory-inf-loop} calls a
\dfn{$\Z/2$-graded $\theta$-ring}.  Thus, an object $\gralgcat$ is
strictly commutative 
graded $\Z_p$-algebra $B_*$, together with a function $\theta\colon
B_*\ra B_*$ which satisfies certain axioms.
The operation $\psi$ is recovered from $\theta$ using the identity
\[
\psi(x)=x^p+p\theta(x)
\]
for all $x\in B_0$.
The ``Wilkerson criterion'' states that a torsion free $\psi$-ring
$B$ admits the structure of a $\theta$-ring (necessarily uniquely), if and
only if $\psi(x)\equiv x^p\mod pB$ for all $x\in B$.  (This result is
a ``$p$-typicalization'' of the original theorem of Wilkerson
\cite{wilkerson-lambda-rings}, which characterizes the torsion free
$\lambda$-rings in terms of congruences on the Adams operations at all
primes.)
\end{exam}

Our first result is a generalization of Wilkerson's criterion.  Its
statement involves a representative $\sigma\in \Gamma$ of a certain conjugacy
class in $\Gamma/p\Gamma$, which is described in
\S\ref{subsec:frobenius-class}.  We say that a graded 
$\Gamma$-algebra $B$ 
satisfies the 
\dfn{congruence condition} if for all $x\in B_0$,
\[
x\sigma\equiv x^p\mod pB.
\]
\begin{thma}
An object $B\in \grAlg{\Gamma}$ which is $p$-torsion
free  admits the structure of a
$\algapprox$-algebra (necessarily uniquely), if and only if 
it satisfies the congruence condition.
\end{thma}
The proof of Theorem A is completed in \S\ref{sec:critical-weight}.

\subsection{Interpretation in terms of formal groups}

The second result of this paper gives a reinterpretation of the above
theorem in terms of formal groups, and in doing so explains the
significance of the element $\sigma\in \Gamma$.  

Fix a perfect field $k$ of characteristic $p>0$, and a formal group
$G_0$ of finite height over $k$.  Let $E$ be the Morava $E$-theory
associated to the universal deformation of $G_0$ in the sense of
Lubin-Tate. 
 
Given a complete local ring $R$, there is a category $\DefFrob_R$,
whose objects are deformations of $G_0$ to $R$, and whose morphisms
are isogenies of formal groups which are ``deformations of a power of
Frobenius''; that is, morphism are isogenies which are identified with
some power of the Frobenius isogeny when we base change to to the
residue field of $R$.  The collection of categories $\DefFrob_R$ for
suitable rings $R$ and base change functors $f^*\colon \DefFrob_R\ra
\DefFrob_{R'}$ describe a kind of moduli problem.  We let $\grShAlg$
denote the category of quasi-coherent sheaves of graded commutative
$\O$-algebras 
over $\DefFrob$.  An object $A$ of this category is (modulo some
issues of grading) a pseudonatural
transformation of pseudofunctors $\DefFrob\ra \mathrm{Alg}^*$; more
concretely, an object $A$ consists of data $\{A_R,A_f\}$, where for
each ring $R$ there is a functor $A_R\colon \DefFrob_R\ra \Alg{R}$,
and for each local homomorphism $f\colon R\ra R'$ a natural
isomorphism $A_f\colon A_{R'}f^* \ra f^*A_R$, satisfying a collection
of coherence relations.  The precise definitions of $\DefFrob_R$ and
$\grShAlg$, including correct treatment of the grading,  are given in
\S\ref{sec:deformations-of-frobenius}.  

We say that a quasi-coherent sheaf $A$ satisfies the \dfn{Frobenius
congruence} if (roughly), for every $\F_p$-algebra $R$ and every
object $G\in \DefFrob(R)$, we have $A_R(G\xra{\Frob} \phi^*G) =
(A_R(G) \xra{\Frob} \phi^* A_R(G))$, where ``$\Frob$'' denotes the
relative Frobenius isogeny on formal groups (on the left-hand side)
and on $R$-algebras (on the right-hand side).  (See
\S\ref{subsec:frobenius-cong} for the precise condition.)
\begin{thmb}
There is an equivalence of categories
$\grShAlg \approx \grAlg{\Gamma}$.   Under this
equivalence, sheaves which satisfy the Frobenius congruence exactly
correspond to graded $\Gamma$-algebras which satisfy the congruence condition.
\end{thmb}
As described below, the equivalence of categories of Theorem B is well
known to experts in this area, and amounts to an interpretation of
some theorems of Strickland.  The proof of Theorem B is completed in
\S\ref{sec:bialgebra-and-sheaves}.

\subsection{The work of Ando, Hopkins, and Strickland}

The structure of power operations on Morava $E$-theory is largely
understood, thanks to work of Matt Ando, Mike Hopkins, and Neil
Strickland.   The hard  results that underlie the version of
the  theory I will describe are
theorems of Neil Strickland, and are proved in
\cite{strickland-finite-subgroups-of-formal-groups} and
\cite{strickland-morava-e-theory-of-symmetric} (the latter corrected
in \cite{strickland-morava-e-theory-of-symmetric-corr}); this work in
turn uses crucially some results of Kashiwabara
\cite{kashiwabara-bp-coh-of-oi-si-s2n}. 
  Unfortunately,
there is no complete statement yet in print of the picture of operations
on $E$-algebra spectra.   Strickland provides a very brief sketch in
\cite[\S14]{strickland-finite-subgroups-of-formal-groups}.  The three
authors describe a version of this story in
\cite{ando-hopkins-strickland-h-infinity}; see especially the material
on ``descent for level structures''
\cite[\S11]{ando-hopkins-strickland-h-infinity}, which 
has informed \S\ref{sec:deformations-of-frobenius} in this paper.
Unfortunately, the level structure approach of
\cite{ando-hopkins-strickland-h-infinity} is not convenient for
describing the congruence condition.

Their unpublished work has some overlap with what we discuss in this
paper.  In particular, they  constructed the algebra $\Gamma$ and
perceived the equivalence
of $\Alg{\Gamma}\approx \ShAlg$ of Theorem B (a large part of this
is accomplished in 
Strickland's papers).  They also understood that the difference
between the categories $\algcat$ and $\Alg{\Gamma}$ was precisely an
issue of understanding certain congruences, and 
that these congruences were generated by ones which are detected in
the $E$-homology of the classifying space $B\Sigma_p$.  The precise
economical statement of Theorem A is new, as is the treatment of
gradings.

\subsection{Treatment of gradings}

We should note that we deal with gradings in a somewhat novel way.
Since Morava $E$-theory is an even periodic theory, we can regard the
homotopy groups of a $K(n)$-local commutative $E$-algebra spectrum as
being a $\Z/2$-graded $\pi_0E$-module, rather than a $\Z$-graded
$\pi_*E$-module.  This 
would not be a 
viable procedure, except for the fact that we can modify the tensor
product structure on $\Z/2$-graded modules by ``twisting'' with the
module $\omega=\pi_2E$ (see \S\ref{sec:twisted-z2-cat} for the precise
formulation).  One can think of our twisted $\Z/2$-graded category
as obtained by adjoining an ``odd square-root'' of $\omega$ to a
tensor category.
This point of view turns out to be very convenient
for dealing with power operations, and we believe it is worthy of
attention.  

\subsection{Completion}

There is a piece of structure on the homotopy $\pi_*B$ of a $K(n)$-local
commutative $E$-algebra spectrum $B$ which is not encoded in our algebraic
model $\gralgcat$; namely, the fact that $\pi_*B$ is usually complete with
respect to the maximal ideal of $\pi_0E$ (or more precisely, that
$\pi_*B$ is  always $L_0$-complete, see \S\ref{subsec:completion-functor}).
It is fair to say that a defect of this paper is that we don't handle
completion issues very well; most of the time, we sidestep the question.
We hope to address these matters at some other time.

\subsection{Calculations}

As a companion piece to this paper, I've made available calculations
of the structure of $\Gamma$ and $\gralgcat$ in a particular case
for height $n=2$ at the prime $p=2$ \cite{rezk-dyer-lashof-example}.

\subsection{Acknowledgments}

I'd like to thank Matt Ando and Mike Hopkins for helpful discussions
relating to issues in this paper.  I would especially like to thank
Haynes Miller and MIT, for allowing me to teach a course on power
operations as a visiting instructor there in the spring of 2006; the
results of this paper 
were conjectured and proved while at MIT, and presented in the course I
gave there.

\section{Twisted $\Z/2$-graded categories}
\label{sec:twisted-z2-cat}

In this section, we describe a procedure for constructing
$\Z/2$-graded additive tensor 
categories out of an additive tensor category, by formally adjoining 
an ``odd square-root'' of an object $\omega$ of the original
category.    This procedure will be used in the rest of this paper to
handle issues related to the ``odd'' degree gradings of Morava
$E$-theory.  It gives a very concise way to explain the graded nature
of the objects in question, and we'll use it both for odd degrees in
Morava $E$-theory (see \S\ref{sec:gr-gamma}), and for the graded
version of the category of sheaves on the deformation category (as in
\S\ref{sec:deformations-of-frobenius}). 

\subsection{Symmetric objects}
\label{subsec:symmetric-objects}

Let $(\mathcal{C},\otimes, \Bbbk)$ be an additive tensor category; that
is, an additive category $\mathcal{C}$
equipped with a symmetric monoidal structure $\otimes$ with unit
object $\Bbbk$, such that $\otimes$ distributes over finite sums.

Let $\tau_{M,N}\colon M\otimes N\ra N\otimes M$ 
denote the interchange isomorphism of the symmetric monoidal structure on
$\mathcal{C}$.   
Say that an object
$\omega\in \mathcal{C}$ is \dfn{symmetric} if
$\tau_{\omega,\omega}=\id_{\omega\otimes \omega}$; this is equivalent
to requiring that the symmetric group act trivially on
$\omega^{\otimes m}$ for all $m\geq0$.

\subsection{Twisted tensor product}

Let $\mathcal{C}^*$ be the category of \dfn{$\Z/2$-graded objects} of
$\mathcal{C}$; an object of $\mathcal{C}^*$ is a pair
$M^*=\{M^0,M^1\}$ of objects of $\mathcal{C}$, and a morphism $f\colon
M^*\ra N^*$ is a pair $f^i\colon M^i\ra N^i$, $i=0,1$, of morphisms of
$\mathcal{C}$.  We define a functor $\otimes \colon
\mathcal{C}^*\times \mathcal{C}^*\ra \mathcal{C}^*$ as follows.
If $M^*$ and $N^*$ are objects in
$\mathcal{C}^*$, we define an object $M^*\otimes N^*\in \mathcal{C}^*$ by
\begin{align*}
(M^*\otimes N^*)^0 &\defeq (M^0\otimes N^0) \oplus (M^1 \otimes
N^1\otimes \omega),
\\
(M^*\otimes N^*)^1 &\defeq (M^0\otimes N^1)\oplus (M^1\otimes N^0),
\end{align*}
where the tensor products on the right-hand side are taken in
$\mathcal{C}$.  We refer to this as the \dfn{$\omega$-twisted
tensor product} on $\mathcal{C}^*$.

Let $\Bbbk=\{\Bbbk,0\}$ as an object of $\mathcal{C}^*$; it serves as the unit
object of the monoidal structure via the evident isomorphisms
$\Bbbk\otimes M^*\approx M^*\approx M^*\otimes \Bbbk$.  Define an
interchange map 
$\tau^*\colon M^*\otimes N^* \xra{\sim} 
N^*\otimes M^*$ by 
\begin{align*}
\tau^0 \colon (M^0\otimes N^0) \oplus (M^1\otimes N^1 \otimes \omega)
&\ra (N^0\otimes M^0) \oplus
(N^1\otimes  M^1\otimes \omega) 
\\
(m_0\otimes n_0, m_1\otimes n_1\otimes x) &\mapsto (\tau(m_0\otimes
n_0), -\tau(m_1\otimes n_1)\otimes x)
\\
\tau^1 \colon (M^0\otimes N^1)\oplus (M^1\otimes N^0)&\ra (N^0\otimes
M^1)\oplus (N^1\otimes M^0)
\\
(m_0\otimes n_1, m_1\otimes n_0) &\mapsto (\tau(m_1\otimes n_0),
\tau(m_0\otimes n_1)).
\end{align*}
(I've written these formulas using ``element'' notation, but they are
easily converted into ``arrow theoretic'' formulas, meaningful in any
additive tensor category $\mathcal{C}$.) 

Finally, there is an  associativity isomorphism $\alpha\colon
(M^*\otimes N^*) \otimes P^* \ra M^*\otimes (N^*\otimes P^*)$, which I
won't write out in detail.  It is
defined using the associativity and interchange isomorphisms for
$\mathcal{C}$, with the interchange map used when needed to move the
extra factor 
of $\omega$ into the ``correct'' position.  For instance, $((M^*\otimes
N^*)\otimes P^*)^0$ contains a summand of the form $(M^1\otimes
N^1\otimes \omega )\otimes P^0$, while the corresponding summand of
$(M^*\otimes (N^*\otimes P^*))^0$ is $M^1\otimes (N^1\otimes
P^0\otimes \omega)$; the map $\alpha$ maps one to the other by
switching $\omega$ and $P^0$ using the interchange map $\tau$.

\begin{prop}
If $\omega\in \mathcal{C}$ is a symmetric object, then the above
structure makes $\mathcal{C}^*$ into an additive tensor 
category.  The functor $\mathcal{C}\ra \mathcal{C}^*$ defined
by $M\mapsto  
\{M,0\}$ identifies $\mathcal{C}$ with a full monoidal subcategory of
$\mathcal{C}^*$. 
\end{prop}
\begin{proof}
The only delicate point 
is to check the commutativity of the pentagon which compares the
associativity isomorphisms of four-fold tensor products; that this
commutes makes essential use of the fact that $\omega$ is symmetric. 
\end{proof}

We will typically identify $\mathcal{C}$ with its essential image in
$\mathcal{C}^*$ without comment.

\subsection{The odd square root of $\omega$}

Let $\omega^{1/2}$ denote the object of $\mathcal{C}^*$ defined by
$\omega^{1/2}=\{0,\Bbbk\}$.  Then $\omega^{1/2}\otimes
\omega^{1/2}\approx \{\omega,0\}\approx\omega$.  Furthermore, the
interchange map $\tau^*$ on $\omega^{1/2}\otimes \omega^{1/2}$ is
equal to $-\id$. 
Every object $M^*$ of
$\mathcal{C}^*$ is isomorphic to one of the form $M^0\oplus
(M^1\otimes \omega^{1/2})$, where $M^0,M^1\in \mathcal{C}\subset
\mathcal{C}^*$.  Thus we can think of $\mathcal{C}^*$ as the additive
tensor category obtained from $\mathcal{C}$ by ``adjoining an odd
square-root'' of $\omega$.

\subsection{Functors from a twisted $\Z/2$-graded category}

Given an additive tensor category $\mathcal{C}$, and a symmetric
object $\omega$ of $\mathcal{C}$, we define the groupoid
$\Sqrt(\omega)$ of \dfn{odd square-roots of $\omega$} as
follows.  The 
objects of $\Sqrt(\omega)$ are pairs $(\eta,f)$, where $\eta$ is an object
of $\mathcal{C}$ such that $\tau_{\eta,\eta}=-\id_{\eta\otimes\eta}$,
and $f\colon \eta\otimes \eta\ra \omega$ is an isomorphism, and the
morphisms $(\eta,f)\ra (\eta',f')$ of $\Sqrt(\omega)$ are
isomorphisms $g\colon \eta\ra 
\eta'$ such that $f'(g\otimes g) = f$.  

Now suppose that $\mathcal{C}$ is an additive tensor category with
symmetric object $\omega$, that $\mathcal{D}$ is an additive tensor
category, and that $F\colon \mathcal{C}\ra \mathcal{D}$ is a symmetric
monoidal functor; thus $F(\omega)$ is a symmetric object of
$\mathcal{D}$.  Let $\mathcal{C}^*$ denote the $\Z/2$-graded tensor
category obtained from $\mathcal{C}$ and $\omega$, and identify
$\mathcal{C}$ with its essential image in $\mathcal{C}^*$.  
Let $\mathcal{G}$ denote the groupoid whose objects are additive
symmetric monoidal functors $F^*\colon \mathcal{C}^*\ra \mathcal{D}$
such that $F^*|_\mathcal{C} = F$, and whose morphisms $F^*_1\ra F^*_2$
are monoidal natural isomorphisms which restrict to the identity map
over $\mathcal{C}$.  
\begin{prop}\label{prop:functor-from-twisted-z2-graded}
There is an equivalence of categories $\mathcal{G}\ra
\Sqrt(F(\omega))$,  defined by $F^*\mapsto
(F^*(\omega^{1/2}), g)$, where $g\colon F^*(\omega^{1/2})\otimes
F^*(\omega^{1/2}) \xra{\sim} F^*(\omega^{1/2}\otimes \omega^{1/2})
\xra{F^*(f)} F^*(\omega)$ is the composite of the coherence map of
$\mathcal{D}$ and the map $F^*(f)$, where $f\colon \omega^{1/2}\otimes
\omega^{1/2}\ra \omega$ is the tautological isomorphism in
$\mathcal{C}^*$. 
\end{prop}

\subsection{Examples}

\begin{exam}
If $\omega=\Bbbk$, then $\mathcal{C}^*$ is just the ``usual''
$\Z/2$-graded category of objects of $\mathcal{C}$.
\end{exam}

\begin{exam}
Suppose the symmetric object $\omega$ is also $\otimes$-invertible in
$\mathcal{C}$, 
i.e., there exists an object 
$\omega^{-1}\in \mathcal{C}$ and an isomorphism $\omega\otimes
\omega^{-1}\approx \Bbbk$.  Then we can define a $\Z$-graded commutative
ring object $R_*$ of $\mathcal{C}$ by
$R_{2k}= \omega^{\otimes k}$ and $R_{2k+1}=0$ for all $k\in
\Z$.  It is straightforward to check that $\mathcal{C}^*$ is then
equivalent to the additive tensor category of $\Z$-graded modules over
$R_*$.   This equivalence associates a $\Z$-graded $R_*$ module $M_*$
with the object $M^*=\{M_0,M_{-1}\}\approx M_0\oplus
(M_{-1}\otimes\omega^{1/2})$ in $\mathcal{C}^*$.  
\end{exam}

\subsection{$E_*$-modules}
\label{subsec:graded-estar-modules}

Let $E_*=\pi_*E$, the coefficient ring of an even periodic ring
spectrum, and let $\omega=\pi_2E$ viewed as an $E_0=\pi_0E$-module.
We can, and will, identify the category $\Mod{E_*}$ of $\Z$-graded
$E_*$-modules with the $\Z/2$-graded category $\grMod{E_0}$, where
$\omega=\pi_2 E =E^0S^2=\pi_0 \Sigma^{-2}E$ is used as the symmetric object.
Observe that 
under the equivalence of $\Z$-graded and $\Z/2$-graded $E_*$-modules
described above, $\pi_*\Sigma^q E \approx E_*S^q$ is naturally identified with
$\omega^{-q/2}$.  The K\"unneth isomorphism
$E_*S^i\otimes_{E_*}E_*S^j\ra E_*(S^i\sm S^j)\approx E_*(S^{i+j})$
produces a canonical isomorphism $\kappa\colon \omega^{-i/2}\otimes
\omega^{-j/2} \ra \omega^{-(i+j)/2}$. 

If $M=\{M^0,M^1\}$ is an object of $\grMod{E_0}$, we
can recover the $\Z$-graded $E_*$-module $M_*$ associated to it by
$M_q = \Hom_{\grMod{E_0}}(\omega^{-q/2},M)$. 

In the examples we've given above, the symmetric object $\omega$ was
$\otimes$-invertible.  Later in this paper, in sections
\S\ref{sec:additive-bialgebra} and
\S\ref{sec:deformations-of-frobenius}, we will consider
$\omega$-twisted tensor categories in 
cases where the symmetric object $\omega$ is not
$\otimes$-invertible.  It is in these non-invertible cases that the
formalism of twisted $\Z/2$-graded tensor categories proves especially
convenient.

\section{Morava $E$-theory and extended powers} 
\label{sec:completed-e-theory}

\subsection{Morava $E$-theory}

We fix for the rest of the paper a perfect field $k$ of characteristic
$p$, and a formal group $G_0$ over $k$ of height $n$, with $1\leq
n<\infty$.  Let $E$ denote the Morava $E$-theory associated to the
universal deformation of the formal group $G_0$.

It is a theorem of Goerss, Hopkins, and Miller, that $E$ is a
commutative $S$-algebra in an essentially unique way
\cite{goerss-hopkins-moduli-spaces}.  

\subsection{$E$-modules and $E_*$-modules}

Let
$\Mod{E}$ denote the category of $E$-module spectra as in \cite{ekmm},
and let $h\Mod{E}$ denote its homotopy category.  We write $M\sm_E N$
and $\funcspec_E(M,N)$ for the smash product and function spectrum of
$E$-module, and also for their derived versions on $h\Mod{E}$.

We write $E_0=\pi_0 E$ and $\omega = \pi_{2}E$  as an $E_0$-module.

Taking homotopy groups defines a functor $\pi_*\colon \Mod{E}\ra
\Mod{E_*}$.  In what follows we are going to regard $\Mod{E_*}$ as the
$\Z/2$-graded 
category described in \S\ref{subsec:graded-estar-modules}, so that the
functor $\pi_*$ is explicitly 
described by 
\[
\pi_* M = \{ \pi_0M, \pi_{-1}M\}.
\]
Recall that according to our conventions, there are natural isomorphisms
$\pi_*\Sigma^q M 
\approx \omega^{-q/2}\otimes \pi_*M$ for  $E$-modules $M$, for all
$q\in \Z$. 
(The reader may prefer to regard $\Mod{E_*}$ as the usual
category of modules over a $\Z$-graded ring; doing so should not cause any
confusion until section \S\ref{sec:additive-bialgebra}.)

\subsection{The completion functor}
\label{subsec:completion-functor}

Let $K(n)$ denote the $n$th Morava $K$-theory spectrum.

Let $L\colon \Mod{E}\ra \Mod{E}$ denote the \dfn{Bousfield localization
  functor} with respect to the homology theory on $E$-modules defined
by smashing with the module  $E\sm K(n)$.  It
comes equipped with a natural coaugmentation map $j\colon M\ra LM$,
which is a $K(n)$-homology equivalence.  
The localization functor $L$ 
descends to a functor $h\Mod{E}\ra h\Mod{E}$ on the 
homotopy category of $E$-module spectra, which we also denote $L$.  We
have the following equivalent descriptions of $L$.

\begin{prop}\label{prop:completion-functor-equivalences}
We have the following equivalences of coaugmented functors 
$h\Mod{E} \ra h\Mod{S}$, where $h\Mod{S}$ denotes the homotopy
category of spectra..
\begin{enumerate}
\item $L\approx L_{K(n)}$, where $L_{K(n)}$ denotes Bousfield
  localization of spectra with respect to Morava $K$-theory.
\item $L\approx L_{F(n)}$, where $L_{F(n)}$ denotes Bousfield
  localization of spectra with respect to a type $n$-finite   spectrum $F(n)$.
\item $LM\approx \holim (E\sm M(i_0,\dots,i_{n-1}))\sm_E M$, where
  $\{M(i_0,\dots,i_{n-1})\}$ denotes a certain inverse system of finite
  spectra, constructed so that $\pi_* E\sm M(i_0,\dots,i_{n-1})\approx 
  E_*/(p^{i_0},u_1^{i_1},\dots, u_{n-1}^{i_{n-1}})$.  
\end{enumerate}
\end{prop}
\begin{proof}
Statement (1) is \cite[Prop.\ VIII.1.7]{ekmm}.

Statements (2) and
(3) follow from \cite[Prop.\
7.10]{hovey-strickland-morava-k-theories}; the proof of (2) uses the
fact that $E$ is an $L_n$-local spectrum, whence the underlying
spectrum of every $E$-module is $K(n)$-local.
\end{proof}

Because of the last equivalence in this list, we can think of $L$ as a
``completion'' functor.  

\subsection{Derived functors of $\mathfrak{m}$-adic completion}

Thus, let $L_s\colon \Mod{E_*}\ra
\Mod{E_*}$ denote the \dfn{$s$th left derived functor} of $M\mapsto
M^{\sm}_{\mathfrak{m}}$.  There is a map $L_0 M\ra
M^{\sm}_{\mathfrak{m}}$, which is not generally an isomorphism since
completion is not right exact.  

\begin{prop}\label{prop:der-func-completion-vanish}
The functors $L_s$ vanish identically if $s>n$.  
If $M_*$ is either a flat $E_*$-module, or a finitely generated
$E_*$-module, then $L_0(M_*)\approx 
(M_*)^{\sm}_{\mathfrak{m}}$ and $L_s(M_*)=0$ for $s>0$.  
\end{prop}
\begin{proof}
The first statement is \cite[Thm.\
A.2(d)]{hovey-strickland-morava-k-theories}.  The statement about flat modules
follows from \cite[Thm.\ A.2(b)]{hovey-strickland-morava-k-theories},
and the statement about finitely generated modules follows from
\cite[Thm.\ A.6(e)]{hovey-strickland-morava-k-theories}.
\end{proof}

The functor $L_0\colon \Mod{E_*}\ra \Mod{E_*}$ is equipped with
natural transformations $M\xra{i} L_0M \xra{j} M^{\sm}_{\mathfrak{m}}$.
In general, $i\colon L_0M\ra M^{\sm}_{\mathfrak{m}}$ is a surjection, but not
an isomorphism, while $L_0(i)$ is an isomorphism.
The groups $L_sM$ can be identified with certain local cohomology groups of
the modules $M$, as described in
\cite[Appendix A]{hovey-strickland-morava-k-theories}.  

\begin{prop}\label{prop:completion-spectral-sequence}
\cite[Prop.\ 2.3]{hovey-morava-e-filtered-colim}.
There is a conditionally and strongly convergent spectral sequence  of
$E_*$-modules 
\[E_2^{s,t} = L_s\pi_t M\Longrightarrow \pi_{s+t}LM,\]
which vanishes 
for $s>n$.  
\end{prop}
As a consequence, we have
\begin{cor}\label{cor:completion-iso-flat}
The map $\pi_*M \ra \pi_* LM$ factors through a natural transformation
$ L_0\pi_*M\ra \pi_*LM$ of functors 
$h\Mod{E}\ra \Mod{E_*}$, which is a natural isomorphism whenever
$\pi_*M$ is flat. 
\end{cor}

For a spectrum $X$, the \dfn{completed $E$-homology} of $X$ is defined
by
$\cE{*}(X) \defeq \pi_* L(E\sm X)$. 

\subsection{Complete $E$-modules}

A \dfn{complete $E$-module} is an $E$-module $M$ which is
$K$-local, i.e., one such that $j\colon M\ra LM$ is an equivalence.   Note that
\begin{enumerate}
\item[(a)] If $M$ and $N$ are $E$-modules, and $N$ is complete, then
  $\funcspec_E(M,N)$ is a complete $E$-module.
\item[(b)] If $M$ and $N$ are complete $E$-modules, $M\sm_E N$ need
  not be complete.  However, 
  $L(M\sm_E N)$ is a complete $E$-module, called the \dfn{completed
  smash product}.
\end{enumerate}

\subsection{Finite and finite free modules}

We write $\funcspec_E(M,N)$ for the function spectrum in $\Mod{E}$.

\begin{prop}\label{prop:charac-ffree-module}
Let $M$ be an $E$-module spectrum.  The following are
equivalent. 
\begin{enumerate}
\item[(1)] $\pi_*M$ is a finitely generated (resp.\ finitely generated
  free) $E_*$-module.
\item[(2)] $\pi_*\funcspec_E(M,E)$ is a finitely generated (resp.\
  finitely generated free) $E_*$-module. 
\end{enumerate}
If either of (1) or (2) hold, then $M\approx LM$ and  $M\approx
\funcspec_E(\funcspec_E(M,E),E)$.

In particular, if $X$ is a spectrum, $E^*X$ is finitely generated (resp.\ 
free) if and only if $\cE{*}X$ is so.
\end{prop}
\begin{proof}
See \cite[\S8]{hovey-strickland-morava-k-theories}.
\end{proof}

We say a module $M$ is \dfn{finite} if $\pi_*M$ is a finitely
generated $E_*$-module.  A module is finite if and only if it is
contained in the thick subcategory of $h\Mod{E}$ generated by $E$.
According to the above proposition, finite modules are complete.

Say that a $E$-module $M$ is \dfn{finitely generated and
free}, or \dfn{finite free} for short, if $\pi_*M$ is a finitely
generated free $\pi_*E$-module.  
All such modules are equivalent to ones of the form
$\bigvee_{i=1}^k \Sigma^{d_i}E$.  Note that 
if $M$ and $N$ are finite (resp.\ finite free), then so are $M\sm_E N$ and
$\funcspec_E(M,N)$.  Also, retracts of finite free modules are also finite
free.  

Let $\Modff{E}$ denote the full subcategory of $\Mod{E}$ consisting
finite free modules.  Let $h\Modff{E}$ denote the full subcategory of
$h\Mod{E}$ spanned by the finite free modules. 
\begin{prop}
The functor $\pi_*\colon h\Modff{E}\ra \Modff{E_*}$ which associates
to an $E$-module its homotopy groups is an equivalence of categories.
\end{prop}
\begin{proof}
A straightforward consequence of the observation that
$h\Mod{E}(\Sigma^d E,M) \approx \pi_dM$.
\end{proof}

\subsection{Flat modules}

Say that an $E$-module $M$ is \dfn{flat} if $\pi_*M$ is flat as a
graded $E_*$-module.  We will need the following analogue of
Lazard's theorem on flat modules over a ring, which is a variant of an
observation of Lurie \cite[\S4.6]{lurie-dag2}.

Let $C$ be a category enriched over spaces.  Say that $C$ is
\dfn{filtered} if the following hold.
\begin{enumerate}
\item For every finite set of objects $X_1,\dots, X_k$ in $C$, there
  exists an object $Y$ and morphism $X_j\ra Y$ in $C$ for all $j=1,\dots,k$.
\item For all $k\geq0$, all $X,Y$ objects of $C$, and all maps $S^k\ra
  C(X,Y)$, there exists a map $g\colon Y\ra Z$ and a dotted arrow
  making the diagram 
\[\xymatrix{
{S^k} \ar[r] \ar[d] & {C(X,Y)} \ar[d]^{C(X,g)}
\\
{D^{k+1}} \ar@{.>}[r] & {C(X,Z)}
}\]
commute.  
\end{enumerate}
For a category $C$ enriched over spaces, let $\pi_0C$ denote the
ordinary (enriched over sets) category with the same objects as $C$,
whose morphisms are the sets of path components of the mapping spaces
of $C$. 
If $C$ is filtered in the above sense, then the
ordinary category $\pi_0 C$ is filtered in the usual sense.  

\begin{prop}\label{prop:flat-e-module}
An object $M\in \Mod{E}$ is flat if and only if it is weakly
equivalent to the homotopy
colimit of some continuous functor $F\colon C\ra \Modff{E}$, where $C$
is filtered. 
\end{prop}
\begin{proof}
This is proved in much the same way as \cite[Theorem
4.6.19]{lurie-dag2}, though with changes of detail, since the notions
of ``flat'' and ``finite free'' we use differ than the ones Lurie
uses.  
We briefly sketch the ideas here.

For the if direction, it suffices to note that taking homotopy groups
commutes with taking homotopy colimit over a filtered diagram.

For the only if direction, consider the comma category $\Mod{E}/M$,
which admits the structure of a topological closed model category.
Choose a set $S$ of fibrant-and-cofibrant representatives of weak
equivalence classes of 
objects $(F,f\colon F\ra M)$ for which $F$ is finite free, and let $C$
be the full topological subcategory of $\Mod{E}/M$ spanned by $S$.
Now one shows that if $M$ is flat, then $C$ is a filtered topological
category in the sense described above.  Given this, it is clear that
$M$ is equivalent to the homotopy colimit of the canonical functor
$C\ra \Mod{E}$.
\end{proof}

\subsection{Completed extended powers}

If $M$ is an $E$-module, the \dfn{$m$th symmetric power} is the
quotient $(M^{\sm_E m})_{h\Sigma_m}$ of the $m$th smash power by the
evident symmetric group action.  The free commutative $E$-algebra on
$M$ is the coproduct $\bigvee_{m\geq0} \Pfree_m(M)$.  

In this paper, we will deal mainly with the \dfn{$m$th extended
  powers}; we write $\Pfree_m(M)\defeq 
(M^{\sm_E m})_{h\Sigma_m}$ for this; the extended power also  passes
to a functor on the 
homotopy category $h\Mod{E}$,
also denoted $\Pfree_m$.  We recall that, in the EKMM model for
$S$-modules, if we choose a $q$-cofibrant model for the commutative
$S$-algebra $E$, then the symmetric powers of cell cofibrant
$R$-modules are homotopy equivalent to extended powers
\cite[III.5]{ekmm}.  

We write $\Pfree(M)=\bigvee_{m\geq0}
\Pfree_m(M)$.  If $G\subseteq \Sigma_m$ is a subgroup, we write
$\Pfree_G(M)\defeq (M^{\sm_E m})_{hG}$.    The functor $\Pfree$
defines a monad on the homotopy category of $E$-modules.   
We will assume that the reader is familiar
with properties of these functors, for instance as described in
\cite[Ch.\ 1]{bmms-h-infinity-ring-spectra}.   In particular, we note
that $\Pfree$ defines a monad on the the homotopy category $h\Mod{E}$
of $E$-modules, and any commuative $E$-algebra results in an algebra
for this monad.

\begin{prop}\label{prop:pfree-completed-isos}
The functors $\Pfree_m$ preserve $K$-homology isomorphisms.  In
particular, $\Pfree_m(j)\colon \Pfree_m(M) \ra \Pfree_m(LM)$ is a
$K$-homology isomorphism, and thus there is a natural isomorphism
$L\Pfree_m(j)\colon L\Pfree_m\ra L\Pfree_m L$ of functors on
$h\Mod{E}$.  The functor $L\Pfree \colon h\Mod{E}\ra h\Mod{E}$ admits a
unique monad structure with the property that $j$ is a map of monads.
\end{prop}
\begin{proof}
The functors $\Pfree$ are homology isomorphisms for any homology
theory; the remaining statements are straightforward.
\end{proof}

The goal of this section 
is to prove
\begin{prop}\label{prop:free-ring-preserves-finite-frees}
If $M$ is an $E$-module which is finite free, then
$L\Pfree_m(M)$ is also finite free.
\end{prop}
This is well-known in the case that $\pi_*M$ is concentrated in
even degree (see \cite[Thm.\ D]{hopkins-kuhn-ravenel}).

\begin{prop}\label{prop:transfer}
If $G$ contains a $p$-Sylow subgroup of $\Sigma_m$, then $\Pfree_G(M)
\ra \Pfree_m(M)$ admits a section for any 
$E$-module $M$.
\end{prop}
\begin{proof}
If $G$ is a subgroup of a group $H$, with index prime to $p$, the map
of spectra $\Sip(H/G)_{(p)} \ra \Sip(H/H)_{(p)}$ admits a retraction in the
homotopy category of spectra equipped with a $H$ action.
\end{proof}

Recall that $p$ is the characteristic of the residue field of $E_*$.  
Let $\rho_{C_p}$ denote the real regular representation of the cyclic
group $C_p$, and let $BC_p^{c\rho_{C_p}}$ denote the Thom spectrum of
the virtual representation $c\rho_{C_p}$, where $c\in \Z$.

\begin{lemma}\label{lemma:prime-cyclic-free}
If  $c\in \Z$, then $\cE{*}BC_p^{c\rho_{C_p}}$ is a
finitely generated 
free $\pi_*E$-module.
\end{lemma}
\begin{proof}
First suppose $c=0$.  Then the cofiber
sequence
$$BC_p\approx S(\lambda^{\otimes p}) \ra BS^1\ra
(BS^1)^{\lambda^{\otimes p}}$$
associated to the universal line bundle $\lambda$ over $BS^1$,
together with 
and the Thom isomorphism for $E$-theory, give
$$0\la E^*BC_p \la E\powser{x} \xla{[p](x)} E\powser{x}\la 0,$$
and in particular since $[p](x)\equiv x^{p^n}\mod \mathfrak{m}$, $E^*BC_p$ is
free over $E_*$ on $1,\dots,x^{p^n-1}$.
Thus $\hom(\Sip BC_p, E)$ is a finitely
generated free $E$-module, and therefore so is $L(E\sm \Sip BC_p)$ by
\eqref{prop:charac-ffree-module}. 

The Thom isomorphism for $E$-theory immediately gives the result for
even $c$, by 
identifying $2d\rho_{C_p}$ with the complex bundle $d\rho_{C_p}\otimes \C$.  It
remains to check the case of odd $c$, and the Thom isomorphism allows
us to reduce to the case $c=1$.

If $p$ is odd, there is a splitting $\rho_{C_p}\approx \R\oplus
\bar{\rho}_{C_p}$ of real $C_p$-representations, where the 
real representation $\bar{\rho}_{C_p}$ admits a complex structure.  The
result follows using the Thom isomorphism, since $E$ is complex
orientable. 

If $p=2$, then $\rho_{C_2}\approx \R\oplus \bar{\rho}_{C_2}$, and
$\bar{\rho}_{C_2}$ is the 
sign representation, so that as spaces,
$$BC_2^{\rho_{C_2}} \approx \Sigma BC_2^{\bar{\rho}_{C_2}} \approx \Sigma BC_2.$$
Stably, the latter  is a retract of $\Sigma(\Sip BC_2)$, whose completed
$E$-homology is finite free as noted above.
\end{proof}

\begin{proof}[Proof of \eqref{prop:free-ring-preserves-finite-frees}]
Since $(S^c)^{\sm p}_{hC_p} \approx BC_p^{c\rho_{C_p}}$,
\eqref{lemma:prime-cyclic-free} 
implies that $L\Pfree_{C_p}(\Sigma^c E)$ is finitely generated free.
Thus \eqref{prop:transfer} shows that $L\Pfree_p(M)$ is a retract
of $L\Pfree_{C_p}(M)$, so that $L\Pfree_p(\Sigma^c E)$ is finitely
generated free.  

The ``binomial formula'' for $\Pfree_p$ says  that
$$\Pfree_p(M\vee N)\approx \bigvee_{i+j=p} \Pfree_i(M)\sm_E
\Pfree_j(N),$$
and since $\Pfree_i(M)$ is a retract of $M^{\sm_E i}$ by
\eqref{prop:transfer}, we conclude that $L\Pfree_p(M)$ takes finite
frees to finite frees.

We have that 
$$\Pfree_G\Pfree_H(M)\approx \Pfree_{H\wr G}(M).$$
If $\Sigma_p^{\wr r}$ denotes the $r$-fold wreath power, we have shown
that $L\Pfree_{\Sigma_p^{\wr r}}=L\Pfree_p\cdots\Pfree_p$
preserves finite frees.  

Finally, for $m\geq0$ with $m=\sum a_i p^i$, with
$a_i\in\{0,\dots,p-1\}$, the group $\Sigma_m$ contains a subgroup 
$G=\prod_i (\Sigma_p^{\wr i})^{\times a_i}$, which acts on
$\underline{m}$ in the evident way, and which contains a $p$-Sylow
subgroup of $G$.  By \eqref{prop:transfer}, $\Pfree_m(M)$ is a retract of
the smash product of finitely many $\Pfree_{\Sigma_p^{\wr i}}(M)$,
and therefore we are done.
\end{proof}

\begin{rem}
The above proof shows a little bit more.  Namely, if $M$ is a finite
free module with $\pi_*M$ concentrated in \emph{even} degree, then
$\pi_*L\Pfree_m(M)$ is also concentrated in even degrees.  There is
no corresponding result when $\pi_*M$ is concentrated in odd degree,
although the proof of \eqref{prop:free-ring-preserves-finite-frees}
implies the following result.
\end{rem}

\begin{cor}\label{cor:pth-power-odd-degree}
The graded module $\pi_* L\Pfree_p(\Sigma^cE)$ is concentrated in odd
degree if $c$ is odd.
\end{cor}

We also note the following interesting consequence (a generalization
of an observation of McClure).
\begin{prop}\label{prop:strictly-graded-commutative}
  If $A$ is a $K(n)$-local commutative $E$-algebra, then the multiplication on
  $\pi_*A$ is \emph{strictly} graded commutative, in the sense that if
  $x\in \pi_q A$ with $q$ odd, then $x^2=0$.
\end{prop}
\begin{proof}
It is clear that we only need to prove something in the $2$-local
case.  Let $f\colon \Sigma^q E \ra A$ be the $E$-module map which
represents $x$.  Then $x^2\in \pi_{2q}A$ is the image of an element in
$\pi_{2q} L\Pfree_2(\Sigma^q E)$ under the map $L\Pfree_2(\Sigma^qE)
\ra L\Pfree_2(\Sigma^q A)\ra A$.  But $\pi_*L\Pfree_2(\Sigma^qE)$ is
concentrated in odd degree by \eqref{cor:pth-power-odd-degree}.
\end{proof}

\subsection{Power operations}
\label{subsec:power-ops}

Let $R$ be a commutative $E$-algebra spectrum.  For any space $X$ and
any $m\geq0$, we obtain an operation
\[
P_m\colon R^0X \ra R^0(X\times B\Sigma_m),
\]
defined so that an $E$-module map $x\colon E\sm \Sip X\ra R$ is sent
to the composite 
\[
E\sm \Sip (X\times B\Sigma_m) \ra E\sm \Sip X^m_{h\Sigma_m} \approx
\Pfree_m (E\sm \Sip X) \xra{\Pfree_m(x)} \Pfree_m R \ra R.
\]
The operation $P_m$ is called the \dfn{$m$th power operation}.  It has
the property that $P_m(xy)=P_m(x)P_m(y)$.  Since $E^0B\Sigma_m$ is a
finite free $E_0$-module, there is an isomorphism $R^0(X\times
B\Sigma_m)\approx R^0X\otimes_{E_0}E^0B\Sigma_m$.  

If $R^X$ denotes the spectrum of functions from $\Sip X$ to $R$, which
is a commutative $E$-algebra, then the power operation $P_m\colon
(R^X)^0(*) \ra (R^X)^0(B\Sigma_m)$ coincides with power operation on
$R^0X$.

Let 
\[
J=\sum_{0<i<m} \im[ R^0(X\times B(\Sigma_i\times\Sigma_{m-i}))
\xra{transfer} R^0(X\times 
B\Sigma_m)], 
\]
where these are transfer maps associated to the subgroups
$\Sigma_i\times \Sigma_{m-i}\subset \Sigma_m$.  Thus $J\subseteq
R^0(X\times B\Sigma_m)$ is an ideal.
We write $\bP_m$ for the
composite map
\[
R^0(X)\xra{P_m} R^0(X\times B\Sigma_m) \ra R^0(X\times B\Sigma_m)/J.
\]
\begin{prop}\label{prop:bP-is-ring-hom-spectra}
The map $\bP_m$ is a ring homomorphism.
\end{prop}

Let $i\colon X\ra X\times B\Sigma_m$ denote the map induced by
inclusion of a basepoint in $B\Sigma_m$.  
\begin{prop}\label{prop:lift-of-mth-power-spectra}
The composite map
\[
R^0(X) \xra{P_m} R^0(X\times B\Sigma_m) \xra{i^*} R^0(X)
\]
sends $x\mapsto x^m$.
\end{prop}

\subsection{Relative power operations}

There is a ``relative'' version of the power operation, which we will
need in \S\ref{sec:bialgebra-and-sheaves}.  Let $(X,A)$ be a CW-pair
of spaces,
let $D_m(X,A)\subseteq X^{\times m}$ denote the space which is the
union of the subspaces 
of the form $X^i\times A\times X^{m-i-1}$, and consider the diagram 
\[\xymatrix{
{\tilde{D}_m(X,A)}  \ar[r]^-{g} \ar[d]
& {X} \ar[r] \ar[d]^{\text{diag}}
& {C(g)} \ar[d]
\\
{D_m(X,A)} \ar[r]_-{f}
& {X^{\times m}} \ar[r]
& {C(f)}
}\]
where $\tilde{D}_m(X,A)$ is the homotopy pullback of $f$ along the
diagonal inclusion, and $C(f)$ and $C(g)$ are homotopy cofibers.  The
group $\Sigma_m$ acts on every space in this diagram. We define a
pointed space
\[
B_m(X,A) \defeq C(g)^{\sm m}_{h\Sigma_m};
\]
it comes with a map $B_m(X,A)\ra C(f)^{\sm m}_{h\Sigma_m} \approx
(X/A)^{\sm m}_{h\Sigma_m}$.  Note that $B_m(X,\varnothing)\approx
(X\times B\Sigma_m)_+$.
Given such a pair $(X,A)$, we define
\[
P_m\colon \tilde{R}^0(X/A) \ra \tilde{R}^0B_m(X,A)
\]
so that an $E$-module map $x\colon E\sm \Si X/A \ra R$ is sent to the
composite
\[
E\sm \Si B_m(X,A) \ra E\sm \Si (X/A)^{\sm m}_{h\Sigma_m} \approx
\Pfree_m(E\sm \Si(X/A)) \xra{\Pfree_m(x)} \Pfree_m R\ra R.
\]
The diagram
\[
\xymatrix{
{\tilde{R}^0(X/A)} \ar[r]^-{P_m} \ar[d]
& {\tilde{R}^0B_m(X,A)} \ar[d]
\\
{R^0X} \ar[r]_-{P_m} 
& {R^0X\times B\Sigma_m}
}\]
commutes.

We are mainly interested in pairs of the form $(D(V),S(V))$,
where $V\ra X$ is a real vector bundle.  In this case, we see that
the relative power operation amounts to a map
\[
P_m\colon \tilde{R}^0 X^V \ra \tilde{R}^0 (X\times B\Sigma_m)^{V\boxtimes \rho_m},
\]
where $\rho_m\ra B\Sigma_m$ is the real vector bundle associated to
the real permutation representation, $V\boxtimes \rho_m$ is the external
tensor product bundle, and the spaces are Thom spaces.

\subsection{Power operations in non-zero degree}

The definition of \S\ref{subsec:power-ops} extends to arbitrary degree,
as follows.  Given $q\in \Z$, there is a function
\[
P_m\colon R^q X \ra R^q(X_+\sm B\Sigma_m^{-q\bar\rho_m}),
\]
where $\bar\rho_m$ denotes the real permutation representation; if
$q>0$, then $-q\bar\rho_m$ is a virtual bundle, and thus the target of
$P_m$ is really $R^q (\Sip X\sm B\Sigma_m^{-q\bar\rho_m})$.
The function $P_m$ is defined so that an $E$-module map $x\colon E\sm
\Sip X \sm S^{-q}\ra R$ is sent to the composite
\begin{multline*}
E\sm \Sip X \sm B\Sigma_m^{-q\bar\rho_m}\sm S^{-q}\approx 
E\sm \Sip X \sm B\Sigma_m^{-q\rho_m} \ra 
\\
E\sm (\Sip X\sm
S^{-q})_{h\Sigma_m} \approx \Pfree_m(E\sm \Sip X\sm S^{-q}) \ra \Pfree_m R
\ra R.
\end{multline*}

Let $X$ be a pointed space.  If $m>0$, then  the operator $P_m$
defined 
above restricts to a function
\[
P_m'\colon \tilde{R}^qX \ra \tilde{R}^q(X\sm
B\Sigma_m^{-q\bar\rho_m}).
\]
\begin{prop}\label{prop:susp-of-power-op}
Let $X$ be a pointed space, and $q\in \Z$.
The diagram
\[
\xymatrix{
{\tilde{R}^qX}  \ar[r]^-{P_m'} \ar[d]^{\sim}_{\text{susp.}}
& {\tilde{R}^q(X\sm B\Sigma_m^{-q\bar\rho_m})}
\ar[r]^-{\tilde{R}^q(\id\sm e)}
& {\tilde{R}^q(X\sm B\Sigma_m^{-(q+1)\bar\rho_m})} \ar[d]^{\text{susp.}}_{\sim}
\\
{\tilde{R}^{q+1}(S^1\sm X)}  \ar[rr]_-{P_m'}
&& {\tilde{R}^{q+1}(S^1\sm X\sm B\Sigma_m^{-(q+1)\bar\rho_m})}
}\]
commutes, where the vertical maps are the suspension isomorphisms, and
$e\colon B\Sigma_m^{-(q+1)\bar\rho_m}\ra B\Sigma_m^{-q\bar\rho_m}$ is
the map of Thom spectra induced by the inclusion $0\subset \rho_m$ of
vector bundles.
\end{prop}
\begin{proof}
A straightforward calculation, using the definitions.
\end{proof}

The following corollary will be crucial for our treatment of gradings
in \S\ref{sec:additive-bialgebra}.  It relates the action of power
operations on $R^qS^q$ with the action of power operations on $R^0(*)$.
\begin{cor}\label{cor:power-op-and-s2}
For all $q\geq0$, the diagram
\[
\xymatrix{
{\tilde{R}^0S^0} \ar[r]^-{P_m'} \ar[d]_{\sim}
& {\tilde{R}^0(S^0\sm B\Sigma_m^0)} \ar[r]^{\tilde{R}^0(\id\sm e)}
& {\tilde{R}^0B\Sigma_m^{-q\bar\rho_m}}  \ar[d]^{\sim}
\\
{\tilde{R}^qS^q} \ar[rr]_{P_m'}
&& {\tilde{R}^q(S^q\sm B\Sigma_m^{-q\bar\rho_m})}
}\]
commutes, where $e\colon B\Sigma_m^{-q\bar\rho_m}\ra B\Sigma_m^0$.
\end{cor}

\section{Approximation functors}
\label{sec:algebraic-approximation}

In this section, we are going to produce a monad $\algapprox$ on the
category $\Mod{E_*}$, called the \dfn{algebraic approximation
  functor}, and thus a category $\gralgcat$ of algebras for 
the monad $\algapprox$.  There will be dotted arrow
\[\xymatrix{
& {\gralgcat} \ar[d]
\\
{\Alg{E}} \ar[r]_{\pi_*} \ar@{.>}[ur]^{\pi_*}
& {\Mod{E_*}}
}\]
making the diagram commute up to natural isomorphism.  Furthermore,
the lift 
is be ``nearly'' optimal, in the sense that for a flat $E$-module
spectrum $M$, the object $\pi_*L\Pfree M$ is be ``nearly'' isomorphic to the
free $\algapprox$-algebra on  $\pi_*M$.  In the above
sentence, ``nearly'' indicates that the isomorphism holds only after
a suitable completion.

The category $\gralgcat$ has a number of nice properties.  Most
notably, the forgetful functor to commutative $E_*$-algebras 
\[
U\colon \gralgcat\ra \Alg{E_*}
\]
admits both a left and a right adjoint, so that $U$ preserves both
limits and colimits.

\subsection{Left Kan extension}
\label{subsec:left-kan-def}

Recall that given functors $F\colon \cat{I}\ra \cat{D}$ and $U\colon
\cat{I}\ra \cat{C}$, 
a \dfn{left Kan extension} of $F$ along $U$ is the initial example of
a pair $(E,\delta)$, where $E\colon \cat{C}\ra \cat{D}$ is a functor and
$\delta\colon F\ra EU$ is natural transformation of functors $\cat{I}\ra
\cat{D}$.  We denote the left Kan extension (if it exists) by $\colim^U F$,
and we write $\beta\colon F\ra (\colim^U F)U$ for the canonical
natural transformation.  The universal property of the left Kan
extension is equivalent to the following: if $G\colon \cat{C}\ra
\cat{D}$ is any 
functor, then there is a one-to-one correspondence between 
\[
(\text{natural transformations $F\ra GU$}) \Longleftrightarrow
(\text{natural transformations $\colim^U F\ra G$}). \]
Note that if $U$ is fully faithful, then $\beta\colon F\ra 
(\colim^U F)U$ is a natural isomorphism. 


Given essentially small $\cat{I}$, choose a set of objects $S$ of
$\cat{I}$ which spans its isomorphism classes. 
For $q\in \Z$, define a functor $B_q=B_q^{F,U}\colon \cat{C}\ra
\cat{D}$ by 
\[
B_q(X) \defeq \coprod_{\substack{I_0\ra\cdots\ra I_q \in \cat{I} \\
    UI_q\ra X\in \cat{C}}} F(I_0),
\]
where the direct sum is taken over all diagrams $I_0\ra \cdots\ra I_q$
in $\cat{I}$ where the objects are in $S$, together with all morphisms
$UI_q\ra X$ in $\cat{C}$.  Then the left Kan extension is the
coequalizer of the evident pair of arrows
\begin{equation}\label{eq:left-kan-def}
B_1(X)\rightrightarrows B_0(X) \ra (\colim^U F)(X).
\end{equation}

Now suppose that we are additionally given a functor $V\colon
\cat{C}\ra \cat{C}'$.  Then we can form the left Kan extension of $F$
along $VU\colon \cat{I}\ra \cat{C}'$, and the canonical transformation
$F\ra (\colim^{VU}F)VU$ corresponds (according to the universal
property of $\colim^U$) to a natural transformation $\gamma\colon
\colim^UF \ra (\colim^{VU}F)V$.
\begin{lemma}\label{lemma:finality}
Given a diagram of categories and functors
\[\xymatrix{
{\cat{I}} \ar[rr]^{U}  \ar[d]_{F}
&& {\cat{C}} \ar[rr]^{V} \ar[dll]|{\colim^U F}
&& {\cat{C}'} \ar[dllll]^{\colim^{VU}F}
\\
{\cat{D}}
}\]
with $\cat{I}$ essentially small, if for all objects 
$I$ in $\cat{I}$ and $X$ in $\cat{C}$ the map $\cat{C}(UI,X)\ra \cat{C}'(VUI,VX)$ is a
bijection, then the transformation $\gamma\colon \colim^UF\ra
(\colim^{VU}F)V$ is an isomorphism.
\end{lemma}
\begin{proof}
Using the coequalizer \eqref{eq:left-kan-def}, it suffices to check
that the evident maps $B_q^{F,U}(X)\ra B_q^{F,VU}(VX)$ are
isomorphims, which is immediate from the hypothesis.
\end{proof}


\subsection{Construction of $\algapprox_m$ and $\algapprox$}

In this section, we construct functors $\algapprox_m\colon
\Mod{E_*}\ra \Mod{E_*}$, called \dfn{algebraic approximation} functors,
We will define $\algapprox (M)
\defeq \bigoplus_{m\geq0} \algapprox_m(M)$.

The idea is to define $\algapprox_m$ first on  finitely generated
free $E_*$-modules using the equivalence of categories $\pi_*\colon
h\Modff{E}\ra \Modff{E_*}$.  Thus, for a finite free $E_*$-module
$M_*$, we should set $\algapprox_m(M_*)= \pi_* L\Pfree_m(M)$,
where $M$ is an $E$-module such that $\pi_*M\approx M_*$.  Then we
extend $\algapprox_m$ to all $E_*$-modules via left Kan extension.

We will refer to the following diagram of functors
$$\xymatrix{
{h\Modff{E}} \ar[r]_i \ar[d]_{\pi_*}^{\sim} \ar@/^1pc/[rr]^{\pi_*L\Pfree_mi}
& {h\Mod{E}} \ar[d]_{\pi_*} \ar@{.>}[r]_{\widetilde{\algapprox}_m}
& {\Mod{E_*}}
\\
{\Modff{E_*}} \ar[r]_j 
& {\Mod{E_*}} \ar@{.>}[ur]_{\algapprox_m}
}$$
in which the left-hand square commutes (on the nose), the functors
$i$ and $j$ are  
inclusions of full subcategories, and the vertical arrow on the left
is an equivalence of categories.

We define $\algapprox_m\colon
\Mod{E_*}\ra \Mod{E_*}$ be the left Kan extension of the functor
$\pi_*L\Pfree_m i\colon h\Modff{E}\ra \Mod{E_*}$ along the functor
$\pi_*i=j\pi_*\colon h\Modff{E}\ra \Mod{E_*}$; this exists
because $h\Modff{E}$ is essentially small.

The functor $\pi_* i$ is fully faithful, and so the tautological
transformation $\beta\colon
\pi_*L\Pfree_m i\ra \algapprox_m\pi_* i$ is an isomorphism.  
\begin{lemma}\label{lemma:algapprox-is-self-kan-extension}
The natural map $\kappa\colon \algapprox_m \ra \colim^j \algapprox_m
j$ adjoint to the identity $\id\colon \algapprox_mj\ra \algapprox_m j$
is an isomorphism.
\end{lemma}
\begin{proof}
Since $\beta$ is an isomorphism and $\pi_* i= j\pi_*$, we have natural isomorphisms
\[
\algapprox_m = \colim^{\pi_* i} \pi_* L\Pfree i \approx
\colim^{j\pi_*} \algapprox_m j \pi_*.
\]
Since $\pi_*\colon h\Modff{E}\ra \Modff{E_*}$ is an equivalence of
categories, we see that $\colim^{j\pi_*} \algapprox_m j\pi_*  \approx
\colim^j \algapprox_m j$.
\end{proof}


\subsection{Construction of the approximation map}

Let $\widetilde{\algapprox}_m\colon h\Mod{E}\ra
\Mod{E_*}$ be the left Kan extension of the functor $\pi_* L\Pfree_m
i\colon h\Modff{E}\ra 
\Mod{E_*}$ along the inclusion $i\colon h\Modff{E}\ra h\Mod{E}$.
There are natural isomorphisms
\[
\algapprox_m \pi_* i \xla{\beta} \pi_* L\Pfree_m i \xra{\id}
\pi_*L\Pfree_m i,
\]
where $\beta$ is the tautological natural transformation for
$\algapprox_m$, which are associated (since $\widetilde{\algapprox}_m$
is a left Kan extension along $i$), to natural transformations
\[
\algapprox_m\pi_* \xla{\gamma} \widetilde{\algapprox}_m
\xra{\tilde{\alpha}} \pi_*L\Pfree_m
\]
of functors $h\Mod{E}\ra \Mod{E_*}$.
\begin{lemma}
The map $\gamma$ is a natural isomorphism.
\end{lemma}
\begin{proof}
We are precisely in the setup of \eqref{lemma:finality}, where we use
the fact that $\pi_*\colon h\Mod{E}(F,M)\ra \Mod{E_*}(\pi_* F,
\pi_*M)$ is an isomorphism for all $F$ in $h\Modff{E}$
\end{proof}



The natural transformation 
\[
\alpha_m\colon \algapprox_m(\pi_* M) \ra \pi_*(L\Pfree_m(M)),
\]
is defined by $\alpha_m = \tilde{\alpha}\gamma^{-1}$.  
\begin{prop}\label{prop:algapprox-finite-free}
When $\pi_*M$ is a finite free $E_*$-module, the map $\alpha_m\colon 
\algapprox_m(\pi_*M) \ra \pi_*
L\Pfree_m(M)$ is an isomorphism.
\end{prop}

The natural transformation
\[
\alpha \colon \algapprox(\pi_* M) \ra \pi_*L\Pfree(M)
\]
is defined by 
\[
\bigoplus_m \algapprox_m(\pi_*M) \xra{\alpha_m} \bigoplus_m
\pi_*L\Pfree_m M \ra \pi_*L\left(\bigvee_m L\Pfree_m M\right) \approx
\pi_* L\Pfree M. 
\]
Note that the analogue to \eqref{prop:algapprox-finite-free} does not
hold for $\alpha$.

\subsection{$\algapprox$ is a monad}

\begin{prop}\label{prop:algapprox-is-monad}
The functor $\algapprox\colon \Mod{E_*}\ra \Mod{E_*}$ admits the
structure of a monad, compatibly with the monad structure of $L\Pfree$,
in the sense that the
diagrams
\[
\xymatrix{
{\pi_*M} \ar[r] \ar[rd]
& {\algapprox(\pi_*M)} \ar[d]^{\alpha}
& {\algapprox\algapprox(\pi_*M)} \ar[r] \ar[d]_{\alpha\circ \algapprox\alpha}
& {\algapprox(\pi_*M)} \ar[d] \ar[d]^{\alpha}
\\
& {\pi_*L\Pfree(M)}
& {\pi_*L\Pfree L\Pfree(M)} \ar[r]
& {\pi_*L\Pfree(M)}
}\]
commute, the unlabeled maps being the ones describing the monad
structure.
\end{prop} 
\begin{proof}
The structure maps $I\ra \algapprox$ and $\algapprox\algapprox\ra
\algapprox$ of the monad
are defined on finite free modules using the maps
$$\pi_* X\ra \pi_*L\Pfree X \qquad\text{and}\qquad \pi_*L\Pfree L\Pfree X\ra
\pi_*L\Pfree X$$
for $X\in \Modff{E}$, together with the equivalence $L\Pfree \Pfree\ra
L\Pfree L\Pfree$ as in \eqref{prop:pfree-completed-isos}. 
\end{proof}

\subsection{Colimits}

\begin{prop}\label{prop:algapprox-colimits}
The functors $\algapprox_m$ commute with filtered colimits and
reflexive coequalizers.  
\end{prop}

\begin{proof}
Observe that for $q=0,1$, the functor
\[
M\mapsto B_q(M)=\bigoplus_{\substack{F_0\ra\cdots\ra F_q \in
    h\Modff{E} \\ \pi_* F_q \ra M \in \Mod{E_*}}} \pi_* L\Pfree_mF_0
\]
from $\Mod{E_*}\ra \Mod{E_*}$ preserves filtered colimits, since
the objects $\pi_*F_q$ of $\Mod{E_*}$ are \emph{small}, in the sense that
$\hom_{\Mod{E_*}}(\pi_*F_q,{-})$ preserves filtered colimits.  The
filtered colimit part of the result follows using \eqref{eq:left-kan-def}.  

The functors $M\mapsto B_q(M)$ also  reflexive coequalizers, since the
objects $\pi_*F_q$ of 
$\Mod{E_*}$ are \emph{projective}, in the sense that
$\hom_{\Mod{E_*}}(\pi_*F_q,{-})$ carries epimorphisms to surjections.
Thus the reflexive coequalizer part of the result 
follows using \eqref{eq:left-kan-def}. 
\end{proof}

\subsection{Tensor products}

Let $k\geq0$, and $M_1,\dots,M_k\in \Mod{E_*}$.  We define a natural
map
$$\gamma_k\colon \algapprox(M_1)\otimes \cdots \otimes
\algapprox(M_k)\ra \algapprox(M_1\oplus\cdots \oplus M_k)$$
as follows.  
As in the proof of \eqref{prop:algapprox-colimits}, let
\[
B_q(M) =
\bigoplus_{\substack{F_0\ra\cdots\ra F_q\in \Modff{E_*} \\ F_q\ra M\in
    \Mod{E_*}}} \bigoplus_m \pi_*L\Pfree_m F_0,
\]
so that $\algapprox(M) \approx H_0 B(M)$.  We have maps 
\[
B_{q_1}(M_1)\otimes \cdots \otimes B_{q_k}(M_k) \xra{s}
B_q(M_1)\otimes 
\cdots \otimes B_q(M_k) \xra{t}
B_q(M_1\oplus
\cdots \oplus M_k)
\]
for $q=\sum q_i$, where $s$ is the Eilenberg-Mac\ Lane
shuffle map, and $t$ is the map constructed in the evident way from
the ``exponential 
isomorphism'' maps
\[
\pi_* L\Pfree_{m_1} F_1 \otimes \cdots \otimes \pi_* L\Pfree_{m_k} F_k
\ra \pi_* L\Pfree_{m_1+\cdots+m_k} (F_1\vee \cdots \vee F_k).
\]
These are maps of chain complexes, and taking  the $0$th homology group gives the desired
map $\gamma_k$.

For the following, we will need to make use of comma categories.
Given a category $\mathcal{C}$ and an object $X$ of $\mathcal{C}$, the
comma category $\mathcal{C}/X$ is the category whose objects are pairs
$(Y,f\colon Y\ra X)$ where $Y$ is an objects of $\mathcal{C}$, and
morphisms $(Y,f)\ra (Y',f')$ are maps $g\colon Y\ra Y'$ such that $f'g=f$.

\begin{prop}\label{prop:algapprox-sum-isomorphism}
The map $\gamma_k$ is an isomorphism.
\end{prop}
\begin{proof}
It is standard that $H_0(s)$ is an isomorphism, so it suffices to show
that $H_0(t)$ is an isomorphism.  Consider the  comma categories
\[
\mathcal{C} = (\prod \pi_* \colon (h\Modff{E})^k\ra
(\Mod{E_*}^k))/(M_1,\dots,M_k)
\]
and 
\[
\mathcal{D} = (\pi_*\colon h\Modff{E}\ra
\Mod{E_*})/(M_1\oplus\cdots\oplus M_k),
\]
and let $\rho\colon \mathcal{C}\ra \mathcal{D}$ be the functor sending
a tuple $(F_i,f_i\colon \pi_*F_i\ra M_i)_{i=1,\dots,k}$ to $(\vee
F_i, (f_i)\colon \pi_*(\vee F_i) \ra \oplus M_i)$.  Let $R\colon
\mathcal{C}\ra \Mod{E_*}$ be the functor sending $(F_i,f_i)$ to
$\bigoplus \pi_* L\Pfree_{m_1}F_1\otimes \cdots \pi_*
L\Pfree_{m_k}F_k$, let $S\colon \mathcal{D}\ra \Mod{E_*}$ be the
functor sending $(F,f)$ to $\bigoplus \pi_* L\Pfree_m F$.  Let
$h\colon R\ra S\rho$ be the evident natural isomorphism.
It is clear
that $H_0(t)$ is isomorphic to the map
\[
\colim^{\mathcal{C}} R \approx \colim^{\mathcal{C}} S\rho \xra{\eta}
\colim^{\mathcal{D}} S,
\]
and the result follows from the observation that $\rho$ admits a left
adjoint and therefore $\eta$ is an isomorphism.
\end{proof}
As a consequence, $\algapprox (M)$ has a natural structure of a
commutative ring, with product defined by $\delta_2\colon
\algapprox(M)\otimes 
\algapprox(M) \xra{\gamma_2} \algapprox (M\oplus M)
\xra{\algapprox((\id_M,\id_M))} \algapprox(M)$.

The naturality of the construction of $\gamma_k$ shows the following. 
\begin{cor}
The natural isomorphisms $\gamma_k$ give  $\algapprox$ the structure
of a symmetric
monoidal functor $(\Mod{E_*},0,\oplus)\ra (\Mod{E_*},E_*,\otimes)$.
Furthermore, the monad structure maps $\eta\colon I\ra \algapprox$ and
$\mu\colon \algapprox\algapprox\ra \algapprox$ are maps of monoidal
functors.  
\end{cor}
\begin{proof}
Reduce to the case of free modules.
\end{proof}


\subsection{Completed approximation functor}

We construct \dfn{completed approximation} functors, which are
better approximations to the homotopy of the $K(n)$-localization of a
free $E$-algebra, but which are less convenient to deal with
algebraically.  Thus we define $\calgapprox(M) \defeq
L_0\algapprox(M)$ where $L_0$ is the functor of
\eqref{prop:completion-spectral-sequence}, and we let
$\widehat{\alpha}\colon 
\calgapprox(\pi_*M)\ra \pi_*L\Pfree(M)$ be the unique factorization of
$\alpha$ through $\calgapprox(M)$.  
\begin{prop}\label{prop:algapprox-flat-module}
If $M$ is a flat $E$-module, then $\calgapprox(\pi_*M)\ra
[\algapprox(\pi_*M)]^{\sm}_{\mathfrak{m}}$ and $\widehat{\alpha}\colon
\calgapprox(\pi_*M)\ra \pi_* L\Pfree(M)$ are isomorphisms.
\end{prop}
\begin{proof}
Since $M$ is a flat $E$-module, $\algapprox \pi_*M$ is a flat $E_*$
module (since $\algapprox$ commutes with filtered colimits
\eqref{prop:algapprox-colimits}), and the first isomorphism follows
using \eqref{prop:der-func-completion-vanish}.

Since $M$ is a flat module, then $M\approx \hocolim_J M_j$ for some
filtered topological category $J$, where the $M_\alpha$ are finite
free \eqref{prop:flat-e-module}.  Let $\N$ denote the category whose
objects are natural numbers, and which has no non-identity maps.
Consider the $E$-module
\[
N = \hocolim_{(i,j)\in \N\times J} L\Pfree_i M_j.
\]
Since each $L\Pfree_i M_j$ is finite free, the $E$-module $N$ is flat,
and thus the map $\beta\colon \pi_*N \ra \pi_* LN$ factors through an
isomorphism 
$L_0\pi_* N \ra \pi_* LN$ by \eqref{cor:completion-iso-flat}.  It is
then straightforward to check that 
$\beta$ is isomorphic to the approximation map
\[
\colim \algapprox_i \pi_* M_j   \approx \algapprox \pi_*M
\xra{\alpha} \pi_* L \hocolim L\Pfree_i M_j\approx \pi_* L\Pfree M.
\]

Note that in the above, we have assumed that $L$ is a continuous
functor.  We may in fact do this, for instance using the description
of $L$ given in \eqref{prop:completion-functor-equivalences}(3).
\end{proof}

\subsection{$\algapprox$-algebras}

As we have observed above, the functor $\algapprox\colon \Mod{E_*}\ra
\Mod{E_*}$ is a monad. 
Let $\gralgcat$ denote the category of $\algapprox$-algebras.  
Every object of $\gralgcat$ is a graded commutative $E_*$-algebra, and so there
is a forgetful functor $U\colon \gralgcat\ra \Alg{E_*}$.  (Note that
the image of $U$ is contained inside the \emph{strictly} graded
commutative $E_*$-algebras, by
\eqref{prop:strictly-graded-commutative}.  

\begin{cor}\label{cor:forget-to-modules-commutes-colims}
The forgetful functor $U\colon \gralgcat\ra \Alg{E_*}$ commutes with
colimits.  
\end{cor}
\begin{proof}
It suffices to show that $U$ commutes with filtered colimits,
reflexive coequalizers, and finite coproducts.  That it commutes with
the first two types of colimit is immediate from
\eqref{prop:algapprox-colimits}.
\end{proof}

\subsection{Power operations, revisited}

Let $B$ be a $\algapprox$-algebra.  We obtain functions
\[
P_m\colon B_0\ra \hom_{\Mod{E_*}}(\algapprox_m E_*, B)
\] 
which are defined by sending an $E_*$-module homomorphism $b\colon
E_*\ra B$ to the composite
\[
\algapprox_m E_*\ra \algapprox_m B\ra B.
\]
We call the map $P_m$ a \dfn{power operation}.  If $R$ is a
commutative $E$-algebra spectrum, we see that this operation
is identified with the operation $P_m\colon R^0(*)\ra R^0B\Sigma_m$
defined earlier, via the natural isomorphism $R^0B\Sigma_m\approx
\hom_{\Mod{E_*}}(\cE{0}B\Sigma_m,R_0)$.  

Let $f\colon E_*[x]\ra \algapprox E_*$ be the map from the free
commutative $E_*$-algebra on one generator which sends $x$ to the
tautological generator of $\algapprox E_*$ as a $\algapprox$-algebra.
Let $f_m\colon E_*\ra \algapprox_m E_*$ be the restriction of $f$ to
the degree $m$ part of $E_*[x]$.  In terms of topology, $f_m$
is the map $i_*\colon E_*\approx \cE{*}(*) \ra \cE{*}B\Sigma_m$
induced by basepoint inclusion.
\begin{lemma}\label{lemma:power-op-is-power-mod-aug}
The composite map
\[
B_0 \xra{P_m} \hom_{\Mod{E_*}}(\algapprox_m E_*, B) \ra
\hom_{\Mod{E_*}}(E_*, B_0)\approx B_0
\]
is the map which sends $x\mapsto x^m$
\end{lemma}
\begin{proof}
Use
\eqref{prop:lift-of-mth-power-spectra}. 
\end{proof}

\subsection{Plethories and plethyistic functors}
\label{subsec:plethyistic}

Let $\mathcal{C}$ be an abelian tensor category with tensor product
$\otimes$ and unit object $\Bbbk$, and
let $\mathcal{A}$ 
denote the category of commutative monoid objects in $\mathcal{C}$
with respect to the tensor product.  Say that a functor $U\colon
\mathcal{D}\ra \mathcal{A}$ is \dfn{plethyistic} if
\begin{enumerate}
\item [(1)] $U$ reflects isomorphisms (i.e., $U(f)$ iso implies $f$
  iso), and
\item [(2)] $U$ admits both a left adjoint $F$ and a right adjoint $G$.
\end{enumerate}
It is a consequence 
of this definition (using Beck's theorem \cite{maclane}) that if $M=UF$ and 
$C=UG$ are the monad and comonad associated to these adjoint pairs,
then $\mathcal{D}$ is equivalent to the categories of $M$-algebras and
$C$-coalgebras. 

The basic example of a plethyistic functor occurs when
$\mathcal{C}=\Mod{R}$ for 
some commutative ring $R$.  Then $U\colon \mathcal{D}\ra \Alg{R}$
amounts to what Borger and Wieland call a \dfn{plethory}
\cite{borger-wieland-plethystic-algebra}.  (More precisely, Borger and
Wieland define a plethory to be a commutative ring $P$ equipped
with some additional structure; they extract a plethyistic functor from
this data, in such a way that $P=M(R)$.  Furthermore, they show
\cite[Thm.\ 4.9]{borger-wieland-plethystic-algebra} that a
plethyistic functor $U\colon \mathcal{D}\ra \Alg{R}$ determines a
plethory in their sense.) 

\begin{prop}\label{prop:algcat-to-E-alg-plethyistic}
The functor $U\colon \gralgcat\ra \Alg{E_*}$ is plethyistic.
\end{prop}
\begin{proof}
It is clear from the definitions that $U$ reflects isomorphisms, and
preserves limits.  It clear from \eqref{prop:algapprox-colimits} and
\eqref{prop:algapprox-sum-isomorphism} that 
$U$ preserves colimits.

Next, we construct a left adjoint $F$ to $U$.  Let $\mathcal{C}$ denote the
full subcategory of $\Alg{E_*}$ consisting of $E_*$-algebras $A_*$ for
which there exists an object $F'(A_*)\in \gralgcat$ and a natural
isomorphism $\gralgcat(F'(A_*),B_*)\approx \Alg{E_*}(A_*,U(B_*))$; the
left adjoint $F$ exists if $\mathcal{C}=\Alg{E_*}$.    
Since
$U$ preserves limits, $\mathcal{C}$ is closed under colimits in
$\Alg{E_*}$, and thus it suffices to show that $\Sym^*(E_0)$ and
$\Sym^*(\omega^{1/2})$ are in $\mathcal{C}$.  But it is
straightforward to check that we can take $F'(\Sym^*(E_0))\approx
\algapprox(E_0)$ and $F'(\Sym^*(\omega^{1/2}))\approx
\algapprox(\omega^{1/2})$.   

It follows by Beck's theorem that $U$ is monadic, since $U$ preserves
colimits \eqref{cor:forget-to-modules-commutes-colims}. 

Next, we construct a right adjoint $G\colon \Alg{E_*}\ra \gralgcat$ to
$U$.  For a $\algapprox$-algebra $A$, consider the functor $X\colon
(\gralgcat)^\op\ra \Set$ defined by 
\[
X(B) = \Hom_{\Alg{E_*}}(UB, A).
\]
Since $U$ preserves colimits, $X$ carries colimits to limits.  The
category $\gralgcat$ is locally presentable, and so $X$ must be
representable by an object $GA$ in $\gralgcat$.  
\end{proof}


\subsection{Weight decomposition}
\label{subsec:weight-decomp}

Suppose that $U\colon \mathcal{D}\ra \mathcal{C}$ is a plethyistic
functor, and that the forgetful functor  $\mathcal{A}\ra \mathcal{C}$
is monadic.  Then 
the composite forgetful functor $\mathcal{D}\ra \mathcal{A}\ra
\mathcal{C}$ is also monadic.  We write $T\colon \mathcal{C}\ra
\mathcal{C}$ for this monad.  For future reference, we note the
following structure carried by $T$, namely natural maps of functors
$\mathcal{C}\ra \mathcal{C}$
\begin{align*}
\eta &\colon M\ra T(M), & \mu &\colon TT(M) \ra T(M),
\\
\iota & \colon \Bbbk \ra T(M), & \delta & \colon T(M)\otimes T(M)\ra
T(M),
\end{align*}
which define the monad and commutative ring structures on the functor $T$.

A \dfn{weight
decomposition} of $T$ is a collection of functors $T_k\colon
\mathcal{C}\ra \mathcal{C}$ and a natural isomorphism $T\approx
\bigoplus_{k\geq0}T_k$ such that the dotted arrows exist in each of
the following diagrams.
$$\xymatrix{
{\Bbbk} \ar[dr]_{\iota} \ar@{.>}[r] 
& {T_0(M)} \ar@{>->}[d]
& {T_k(M)\otimes T_\ell(M)} \ar@{>->}[d] \ar@{.>}[r]
& {T_{k+\ell}(M)} \ar@{>->}[d]
\\
& {T(M)}
& {T(M)\otimes T(M)} \ar[r]_-{\delta}
& {T(M)}
}$$
$$\xymatrix{
{M} \ar[dr]_{\eta} \ar@{.>}[r]
& {T_1(M)} \ar@{>->}[d]
& {T_kT_\ell(M)} \ar@{>->}[d] \ar@{.>}[r]
& {T_{k\ell}(M)} \ar@{>->}[d]
\\
& {T(M)}
& {TT(M)} \ar[r]_-{\mu}
& {T(M)}
}$$
In each case, the dotted arrow is unique if it exists.  

In the case of the plethyistic functor $\gralgcat\ra \Alg{E_*}$, $T$
is precisely the algebraic approximation functor $\algapprox$.  It is
clear that the standard splitting $\algapprox \approx
\bigoplus_{k\geq0} \algapprox_k$  is a weight decomposition for $\algapprox$.

\subsection{Suspension map}
\label{subsec:suspension-map}

There are natural ``suspension maps''
$$E_q\colon \Sigma^qL\Pfree_m(X) \ra L\Pfree_m(\Sigma^q X)$$
for $q\geq0$.
This gives rise to a map of $E_*$-modules
$$E_q\colon \omega^{-q/2}\otimes \algapprox_m(M) \ra
\algapprox_m(\omega^{-q/2}\otimes M)$$  
which we will also call a \dfn{suspension map}.  

We record some properties of the suspension map below.  They all
reduce to corresponding properties of the extended power functors.

\begin{prop}\label{prop:susp-composition}
The diagram
\[\xymatrix{
{\omega^{-i/2}\otimes \omega^{-j/2}\otimes \algapprox(M)}
\ar[r]^-{\id\otimes E_j} \ar[d]_{\kappa\otimes \id}
& {\omega^{-i/2}\otimes \algapprox(\omega^{-j/2}\otimes M)} \ar[r]^-{E_i}
& {\algapprox(\omega^{-i/2}\otimes \omega^{-j/2}\otimes M)}
\ar[d]^{\algapprox(\kappa\otimes \id)}
\\
{\omega^{-(i+j)/2} \otimes \algapprox(M)} \ar[rr]_{E_{i+j}}
&& {\algapprox(\omega^{-(i+j)/2}\otimes M)}
}\]
commutes in $\Mod{E_*}$ for all $i,j\geq0$ and all $M\in \Mod{E_*}$.
\end{prop}

\begin{prop}\label{prop:susp-compat-with-algapprox}
The diagrams
\[\xymatrix{
{\omega^{-q/2}\otimes M} \ar[d]_{\id\otimes \eta} \ar[dr]_{\eta}
\\
{\omega^{-q/2}\otimes \algapprox(M)} \ar[r]_{E_q}
& {\algapprox(\omega^{-q/2}\otimes M)}
}\]
and
\[\xymatrix{
{\omega^{-q/2}\otimes \algapprox\algapprox(M)} \ar[r]^-{E_q}
\ar[d]_{\id\otimes \mu}
& {\algapprox(\omega^{-q/2}\otimes \algapprox(M))} \ar[r]^-{\algapprox(E_q)}
& {\algapprox\algapprox(\omega^{-q/2}\otimes M)} \ar[d]^{\mu}
\\
{\omega^{-q/2}\otimes \algapprox(M)} \ar[rr]_{E_q}
&& {\algapprox(\omega^{-q/2}\otimes M)}
}\]
commute in $\Mod{E_*}$ for all $q\geq0$ and all $M\in \Mod{E_*}$.
\end{prop}

Define $\nu\colon \algapprox (M\otimes N) \ra \algapprox (M)\otimes
\algapprox(N)\approx \algapprox(M\oplus N)$ to be the unique
$\algapprox$-algebra map which extends $\eta\otimes \eta\colon
M\otimes N\ra \algapprox(M)\otimes \algapprox(N)$. 

\begin{prop}\label{prop:susp-compat-multiplicativity}
The diagram 
\[\xymatrix{
{\omega^{-i/2}\otimes \omega^{-j/2} \otimes \algapprox(M\otimes N)}
\ar[d]_{\text{``$E_{i+j}$''}} \ar[r]^-{\id\otimes \id\otimes \nu}
& {\omega^{-i/2}\otimes \omega^{-j/2}\otimes \algapprox (M)\otimes
  \algapprox (N)} \ar[d]^{\id\otimes \tau\otimes \id}
\\
{\algapprox(\omega^{-i/2}\otimes \omega^{-j/2}\otimes M\otimes N)}
\ar[d]_{\algapprox(\id\otimes\tau\otimes \id)}
& {\omega^{-i/2}\otimes \algapprox(M)\otimes \omega^{-j/2}\otimes
  \algapprox(N)} \ar[d]^{E_i\otimes E_j}
\\
{\algapprox(\omega^{-i/2}\otimes M\otimes \omega^{-j/2}\otimes N)} 
\ar[r]_-{\nu}
& {\algapprox(\omega^{-i/2}\otimes M)\otimes
  \algapprox(\omega^{-j/2}\otimes N)} 
}\]
commutes in $\Mod{E_*}$ for all $i,j\geq0$ and all $M,N\in\Mod{E_*}$,
where ``$E_{i+j}$'' means the map defined by composition along the top
of the diagram in  \eqref{prop:susp-composition}.   
\end{prop}

\section{Twisted bialgebras}
\label{sec:bialgebras}

Let $R$ be a commutative ring.  An \dfn{algebra under $R$} is an
associative ring $\Gamma$ together with a ring homomorphism
$\eta\colon R\ra \Gamma$.  Note that the image of $\eta$ is not
assumed to be central (in which case $\Gamma$ is an
\dfn{$R$-algebra}).  

An \dfn{twisted cocommutative $R$-bialgebra} (or \dfn{bialgebra}) is an
algebra $\Gamma$  under $R$ which is 
equipped with certain additional structures which make the category of
right $\Gamma$-modules into a symmetric monoidal category, in such a way
that this symmetric monoidal product coincides with the tensor product
over $R$ of the underlying $R$-modules.   In the case when $R$ is
central in $\Gamma$, then we have the conventional notion of a cocommutative
$R$-bialgebra (like a cocommutative Hopf algebra, but without an
antipode).  

The original definition of bialgebra is due to Sweedler
\cite{sweedler-groups-simple-algebras}, as modified by
\cite{takeuchi-groups-algebras}, under the name of
``$\times_R$-bialgebra''.  (Sweedler described a somewhat more general
situation, in which the monoidal structure is not required to be
symmetric.)  This notion as well as various generalizations have been
much studied in the literature on Hopf algebras and allied notions.
It has entered 
topology in various guises.  I first learned about the notion from
\cite{voevodsky-reduced-power-operations}; the motivic Steenrod
algebra is naturally a twisted bialgebra over the motivic homology of
the base scheme.   Bialgebras can be
thought of as a kind of ``dual'' version of the notion of Hopf
algebroids (or more precisely, of ``affine category schemes''), an
observation we will make explicit below
(\S\ref{subsec:gr-bialg-affine-cat-scheme}).  

I've taken the terminology ``twisted bialgebra'' and some notation
from \cite{borger-wieland-plethystic-algebra}.

\subsection{Bimodules, multimorphisms, and functors}

Let $R$ be a commutative ring, and suppose that $P$ and $Q$ be
$R$-bimodules.  We write 
\begin{align*}
  P_R\otimes {}_R Q  & \defeq P\otimes Q/(pr\otimes q
  \sim p\otimes rq),
\\
  P_R\otimes Q_R & \defeq P\otimes Q/(pr\otimes q\sim
  p\otimes qr).
\end{align*}
Because $R$ is commutative, each of these admit \emph{three} possible
$R$-modules structures.  For instance, $P_R\otimes Q_R$ admits the
following three $R$-module structures (two on the left, one on the right):
\[
r\,{\cdot}\strut_1 (p\otimes q) = rp\otimes q;\quad 
r\,{\cdot}\strut_2 (p\otimes q) = p\otimes rq; \quad
(p\otimes q)\cdot r = pr\otimes q= p\otimes qr.
\]

If $P$ and $Q_1,\dots, Q_K$  are $R$-bimodules, then a
\dfn{$k$-multimorphism} is a function $f\colon P\ra (Q_1)_R\otimes \cdots
\otimes (Q_k)_R$ which is a map of right $R$-modules, and which
is a map of left $R$-modules for each of the $k$ left $R$-module
structures on the target.  Note that a $0$-multimorphism is just a map
$P\ra R$ of right $R$-modules.

Given an $R$-bimodule $P$, let
$H_M\colon \Mod{R}\ra \Mod{R}$ denote the functor defined by
$H_P(M)= \hom_R(P,M)$, where this means the set of
homomorphisms with respect to the \emph{right} $R$-module structure on
$P$.  Thus, $f\in H_P(M)$ is a additive map $f\colon
P\ra M$ such that $f(p r)=f(p)r$ for all $p \in
P$ and $r\in R$.   A bimodule map $b \colon P\ra Q$ induces a natural
transformation $\tilde{b}\colon H_Q\ra H_P$, defined by
\[
\tilde{b}\colon H_Q(M)\ra H_P(M),\qquad 
f  \mapsto (p\mapsto f(b(p)).
\]
More generally, a $k$-multimorphism $b\colon P\ra (Q_1)_R\otimes\cdots
\otimes (Q_k)_R$ induces a natural transformations 
\begin{align*}
\tilde{b}\colon  H_{Q_1}(M_1)_R\otimes \cdots \otimes H_{Q_k}(M_k)_R & \ra
H_P(M_1\otimes_R\cdots \otimes_R M_k)
\\
(f_1\otimes \cdots \otimes f_k) & \mapsto 
(p\mapsto \sum_\alpha f(q_1^\alpha)\otimes \cdots \otimes f(q_k^\alpha)),
\end{align*}
where $b(p)=\sum_\alpha q_1^\alpha\otimes\cdots \otimes q_k^\alpha$.
In particular, a $0$-multimorphism $b\colon P\ra R$ induces
\[
\tilde{b}\colon R\ra H_P(R),\qquad r \mapsto (p \mapsto b(rp)).
\]

\begin{rem}
The assignments $b\mapsto \tilde{b}$ described above satisfy a number
of useful properties, relating to composition of functors and to the
tensor product of $R$-modules; these can be safely left to
the reader.  The language of ``multicategories''
\cite[\S2.1]{leinster-higher-operads-higher-cats} can be used to
encapsulate these properties: bimodules and multimorphisms form a
multicategory $\mathcal{B}$, additive endofunctors $\Mod{R}\ra
\Mod{R}$ and multilinear natural transformations form a multicategory
$\mathcal{F}$, and there is a  
morphism
$H\colon \mathcal{B}\ra \mathcal{F}$ of multicategories.
Furthermore, $\mathcal{B}$ and $\mathcal{F}$ admit an additional
``monoidal'' structure (corresponding to bimodule tensor product and
functor composition respectively), and $H$ is compatible with this
monoidal structure (in the weak sense).
\end{rem}

\subsection{Twisted commutative $R$-bialgebra}

Let $R$ be a commutative ring.  An \dfn{$R$-bialgebra} (more
precisely, a \dfn{twisted cocommutative $R$-bialgebra} is 
data $(\Gamma,\epsilon,\Delta,\eta,\mu)$, consisting of
\begin{enumerate}
\item [(a)] an associative ring $\Gamma$;
\item [(b)] a map $\eta\colon R\ra \Gamma$ of rings;
\item [(c)] a map $\epsilon\colon \Gamma\ra R$ of right $R$-modules
  such that $\epsilon(\eta(r))=r$ for all $r\in R$ and
  $\epsilon(xy)=\epsilon(\eta(\epsilon(x))y)$ for all $x,y\in \Gamma$;
\item [(d)]
  a map $\Delta\colon \Gamma\ra \Gamma_R\otimes \Gamma_R$ which is a
  $2$-multimorphism of $R$-bimodules, which is is coassociative and
  cocommutative with counit $\epsilon$, and such that
\[
\Delta(xy) = \sum x_i'y_j'\otimes x_i''y_j'',
\]
where $\Delta(x)=\sum x_i'\otimes x_i''$ and $\Delta(y)=\sum
y_j'\otimes y_j''$.
\end{enumerate}
Note that the multiplication on $\Gamma$ descends to a map $\mu\colon
\Gamma_R \otimes {}_R\Gamma\ra \Gamma$.

The functor $H_\Gamma$ is a comonad, with
comonadic structure $\tilde{\eta}\colon H_\Gamma\ra I$ and
$\tilde{\mu}\colon H_\Gamma\ra H_\Gamma H_\Gamma$ induced by $\eta$
and $\mu$.  The maps $\epsilon$ and $\Delta$ induce natural
transformations 
\[
\tilde{\epsilon}\colon R\ra H_\Gamma(R),\qquad \tilde{\Delta}\colon
H_\Gamma(M)\otimes_R H_\Gamma(N)\ra  
H_\Gamma(M\otimes_R N) 
\]
which give $H_\Gamma$ the structure of a symmetric monoidal functor; this
structure is compatible the monad structure, so that $\tilde{\eta}$
and $\tilde{\mu}$ are transformations of monoidal functors.

By a \dfn{module} over the $R$-bialgebra $\Gamma$, we mean a pair
$(M,\psi_M)$ consisting of an $R$-module $M$, and a map $\psi_M\colon
M\ra H_\Gamma(M)$ of $R$-modules such that $\tilde{\eta}\psi_M=\id_M$
and $H_{\Gamma}(\psi_M)\psi_M=\tilde{\mu}\psi_M$.  That is, a module
is defined to be a coalgebra for the comonad
$(H_\Gamma,\tilde{\eta},\tilde{\mu})$.  
Let $\Mod{\Gamma}$ denote the category of modules over $\Gamma$.  

Note
that what we are  calling a module over $\Gamma$ really does coincide
with the usual notion of a right $\Gamma$-module: the map
$\psi_M\colon M \ra H_\Gamma(M)$ is adjoint to a map
$M_R\otimes {}_R\Gamma\ra M$ defining a right $\Gamma$-module
structure on $M$.

\subsection{Symmetric monoidal structure on $\Mod{\Gamma}$}

There is a canonical $\Gamma$-module structure on $R$, defined by
taking $\psi_R=\tilde{\epsilon}\colon R\ra H_\Gamma(R)$.

Given $M,N\in \Mod{\Gamma}$ we define their tensor product to be
$(M\otimes_R N,\psi_{M\otimes_R N})$, where
$$\psi_{M\otimes_R N}\defeq \tilde{\Delta}(\psi_M\otimes \psi_N)\colon
M\otimes_R N\ra H_\Gamma(M\otimes_R N).$$
That is is indeed a $\Gamma$-module follows from the fact that the
comonad structure maps $\tilde{\eta}$ and $\tilde{\mu}$ are compatible
with monoidal structures.
\begin{prop}
The above defines a symmetric monoidal structure on the category
$\Mod{\Gamma}$, such that the forgetful functor $U\colon
\Mod{\Gamma}\ra \Mod{R}$ is a symmetric monoidal functor.
\end{prop}

\subsection{$\Gamma$-algebras}

Let $\Gamma$ be an $R$-bialgebra.  By a \dfn{$\Gamma$-algebra}, we
mean a commutative monoid object $B$ in the symmetric monoidal
category of $\Gamma$-modules.  That is, $B$ is equipped with maps
$i\colon R\ra B$ and $m\colon B\otimes B\ra B  $ of $\Gamma$-modules which
also provide it with the structure of a commutative ring.  We write
$\Alg{\Gamma}$ for the category of $\Gamma$-algebras.

The \dfn{free $\Gamma$-algebra} on an $R$-module $M$ is seen to have
the form $\Sym_R^*(M\otimes{}_R\Gamma)$, where the symmetric powers
are taken with respect to the right $R$-module structure on $\Gamma$.

The forgetful functor $\Alg{\Gamma}\ra \Alg{R}$ is plethyistic, in the
sense of \S\ref{subsec:plethyistic}.  

\subsection{Graded bialgebras and graded affine category schemes} 
\label{subsec:gr-bialg-affine-cat-scheme}

A \dfn{grading} for an $R$-bialgebra $\Gamma$ is a decomposition
$\Gamma\approx \bigoplus_{k\geq0}\Gamma[k]$, such that
\[
\mu(\Gamma[k]\otimes \Gamma[\ell])\subseteq \Gamma[k+\ell],\qquad 
\eta(R)\subseteq \Gamma[0],\qquad \Delta(\Gamma[k])\subseteq
\Gamma[k]_R\otimes \Gamma[k]_R.
\]
Thus in particular $\Gamma$ is a graded ring, and each graded piece
$\Gamma[k]$ is a cocommutative coalgebra.

Suppose that each $\Gamma[k]$ is projective and finitely generated as
a right $R$-module.  Let $A[k]=H_{\Gamma[k]}(R)$.  The $A[k]$'s admit
two different $R$-module structures, which we call the ``source'' and
``target'' module structures; the ``source'' module
structure is defined by $(r\cdot_s f)(x)=f(xr)$, while the ``target''
module structure is defined by $(r\cdot_t f)(x)=f(rx)$, where $f\in
A_k$, $r\in R$, $x\in \Gamma[k]$.  

We define
\begin{align*}
  \mu &\colon A[k] \otimes_R A[k]\ra A[k], &
  \mu(f\otimes 
  g)(x) &= \sum f(x_i')g(x_i''),
\\
  s^* & \colon R\ra A[k], & s^*(r)(x) &=
  \epsilon(xr)=\epsilon(x)r,
\\
  t^* & \colon R\ra A[k], & t^*(r)(x) &= \epsilon(rx),
\\
  i^* & \colon A[0]\ra R, & i^*(f) &= f(\eta(1)),
\\
  c^* & \colon A[k+\ell]\ra  A[k]_s \otimes_R {}_tA[\ell],
& \sum f_i''(f_i'(x)y) &= f(xy),
\end{align*}
where $\Delta(x)=\sum x_i'\otimes x_i''$ and we write
$c^*(f)=\sum f_j'\otimes f_j''$.   The map $\mu$ makes $A[k]$ into
a commutative ring, and the ``source'' and ``target'' $R$-module
structures on $A[k]$ are are the same as those induced by the ring
homomorphisms $s^*$
and $t^*$:
\[
r\cdot_s f = f(xr)= \mu(s^*(r)\otimes f), \quad
r\cdot_t f= f(rx)= \mu(t^*(r)\otimes f).
\]

By a \dfn{graded category}, we mean a small category $C$
equipped with a degree function $\deg\colon \mor C\ra \N$, such that
(i) $\deg(\alpha\circ \beta)=\deg(\alpha)+\deg(\beta)$, and
$\deg(\id)=0$.
The data $\mathcal{A}=(R,A[k],s^*,t^*,i^*,c^*)$ described above determine a
\dfn{graded affine 
category scheme}; that is, they corepresent a functor from commutative
rings to graded categories.  In particular, $R$ represents the set of
objects of $C$, and $A[r]$ represents the set of morphisms of $C$ of
degree $r$.  The maps $s^*$, $t^*$, $i^*$, $c^*$ correspond to source,
target, identity, and composition in $C$.

If $\mathcal{A}= (R,A[r],s^*,t^*,i^*,c^*)$ is the data representing an
affine graded category scheme over $R$, 
let $V\colon
\Mod{R}\ra \Mod{R}$ be the functor defined by 
\[V(M) \defeq \prod_{r\geq0} A[r]_s\otimes_R M ,\]
 where $V(M)$  obtains 
its $R$-module structure from the ``target'' module structure on
the $A[r]$'s.
The functor $V$ admits the structure of a comonad, in an apparent way;
the counit map $VM\ra M$ is the projection $VM\ra 
A[0]_s\otimes_RM\xra{i^*\otimes \id} R\otimes_R M\approx M$, and  the
comultiplication $VM\ra VVM$ is the map
\[ 
\prod_r A[r]_s \otimes_RM \ra \prod_\ell A[k]_s\otimes_R
{}_t\left(\prod_k A[\ell]_s\otimes  M 
\right) \approx 
\prod_{k,\ell} A[k]{}_s \otimes_R {}_t A[\ell]_s \otimes_R M
\]
induced by the maps $c^*$.  
A coalgebra for the comonad
$V$ is called an \dfn{$\mathcal{A}$-comodule}, and  the category of
$\mathcal{A}$-comodules is denoted $\Comod{\mathcal{A}}$.

Under our hypothesis on $\Gamma$, the 
natural transformation $\chi\colon V(M) \ra
H_\Gamma(M)$ of functors $\Mod{R}\ra \Mod{R}$ defined by
\[
\chi(f\otimes m) \mapsto (x \mapsto f(x)m)
\]
is an isomorphism for all $M$, and the comonad structure on $V$
corresponds exactly to the comonad structure on $H_\Gamma$.  Thus, we
get the following. 
\begin{prop}
If $\Gamma$ is a graded bialgebra over $R$ such that each graded piece
$\Gamma[k]$ is projective and finitely generated as a right
$R$-module, and if $\mathcal{A}$ is the corresponding graded affine
category scheme, then there is an isomorphism of categories
$\Mod{\Gamma}\approx \Comod{\mathcal{A}}$.  
\end{prop}

\section{The bialgebra of power operations}
\label{sec:additive-bialgebra}

In this section, we are going to construct a twisted commutative
$E_0$-bialgebra $\Gamma$, together with a symmetric object
$\omega$ in $\Mod{\Gamma}$; the bialgebra $\Gamma$ will come with a
grading in the sense of \eqref{subsec:gr-bialg-affine-cat-scheme}.  

From the point of view of topology, $\Gamma$ is the ``algebra of power
operations in Morava $E$-theory''; more precisely, it is (up to issues of
completion), the algebra of \emph{additive} operations on $\pi_0$ of a
$K(n)$-local  $E$-algebra spectrum.  The difficult part of the
construction was given in
\cite{strickland-morava-e-theory-of-symmetric}.  From the point of
view of the theory of plethories, $\Gamma$ is the ``additive
bialgebra'' of the plethory associated to $\algapprox$, as described in
\cite[\S10]{borger-wieland-plethystic-algebra}.

\subsection{Additive operations on $\algapprox$-algebras}
\label{subsec:additive-ops-degree-0}

Let $P=\Hom_{\gralgcat}(\algapprox(E_*),\algapprox(E_*))$.  The set $P$
has the structure of a monoid, under composition of morphisms.
Because there is a natural bijection
$\Hom_{\gralgcat}(\algapprox(E_*),A)\approx A_0$, defined by evaluation at
the canonical generator $\iota\in \algapprox(E_0)$, we see that $P$ is
naturally identified with the monoid of endomorphisms of the forgetful
functor $\algcat\ra \Set$ which sends to $A$ to the underlying set
$A_0$ of the degree $0$ part of $A$.
In particular, evaluation at $\iota\in \algapprox(E_0)$ defines a
natural bijection 
\[
\pi\colon P\ra [\algapprox(E_*)]_0
\]
to the degree $0$ part of $\algapprox(E_*)$.  We will use this
bijection implicitly in what follows.


Write $\circ \colon P\times P\ra P$ for the monoid product on $P$,
defined in terms of composition of endomorphisms of the forgetful
functor. 
If $f,g\in P$ correspond to maps
$\tilde{f},\tilde{g}\colon E_0\ra \algapprox E_0$ of $E_0$-modules,
then $f\circ g$ corresponds to
\[
E_0 \xra{\tilde{g}} \algapprox E_0 \xra{\algapprox \tilde{f}}
\algapprox \algapprox E_0 \xra{\mu} \algapprox E_0.
\]
Observe that if $f\in \algapprox_m E_0$ and $g\in \algapprox_n E_0$,
then $f\circ g \in \algapprox_{mn}E_0$; this is a consequence of the
fact that $\mu\colon \algapprox \algapprox\ra \algapprox$ carries
$\algapprox_m \algapprox_n $ into $\algapprox_{mn}$
(\S\ref{subsec:weight-decomp}).  

The set $P$  has the structure of an abelian group, by addition of
natural endomorphisms of the forgetful functor, or equivalently by
addition in $\algapprox(E_*)$.  We observe that
$\circ$ is ``left additive'', in the sense that 
\[
(f+g)\circ h = (f\circ h) + (g\circ h)
\]
for $f,g,h\in P$.

\subsection{The additive bialgebra associated to $\algapprox$}
\label{subsec:additive-bialgebra}

Let $\Gamma\subset P$ denote the subset of elements $f\in P$ which
induce \emph{additive} natural endomorphisms of the forgetful functor;
that is, $f\in \Gamma$ if it induces an endormorphism of the forgetful
functor $\gralgcat\ra
\Ab$.  Thus, $\Gamma$ is naturally an associative ring, isomorphic to
the ring of endomorphisms of the forgetful functor $\gralgcat\ra \Ab$
to abelian groups.

Define 
\[ 
\Delta^{+}\colon \algapprox(E_*)\ra \algapprox(E_*)\otimes \algapprox(E_*)
\]
to be the composite
\[
\algapprox(E_0) \xra{\algapprox(\nabla)} \algapprox(E_0\oplus E_0)
\xra[\sim]{\gamma^{-1}}  \algapprox(E_0)\otimes \algapprox (E_0),
\]
where $\Delta\colon E_*\ra E_*\oplus E_*$ denotes the diagonal.  
Say an element $f\in [\algapprox(E_*)]_0$ is \dfn{primitive} with
respect to $\circ$ if  
$\Delta^+(f)=f\otimes 1+1\otimes f$.
It is a straightforward exercise to check that under the bijection
$P\xra{\sim} [\algapprox(E_*)]_0$, elements of $\Gamma$ correspond
precisely to primitive elements.


Translating this into topology by means of the isomorphism
$\algapprox_m(E_*)\approx \cE{*}B\Sigma_m$, we see that
$\Gamma\approx \bigoplus_{k\geq0} \Gamma[k]$, where $\Gamma[k]\subset
\algapprox_{p^k} (E_0)$ defined by
\[
0\ra \Gamma[k] \ra \cE{0}B\Sigma_{p^k} \ra
\bigoplus_{0<j<p^k}\cE{0} B(\Sigma_j\times \Sigma_{p^k-j}),
\]
where the maps are induced by transfer map associated to the evident
inclusions $\Sigma_j\times \Sigma_{p^k-j}\ra \Sigma_{p^k}$.  (For a
symmetric group of order not a $p$th power, the kernel of these
transfer maps is $0$.)

\begin{prop}\label{prop:gamma-summand-even}
The submodule $\Gamma[k]\subset \algapprox_{p^k}(E_0)$ is a direct 
summand. In particular, $\Gamma[k]$ is a finitely generated free
$E_0$-module. 
\end{prop}
\begin{proof}
By standard considerations involving Sylow subgroups of
$\Sigma_{p^k}$, it is straightforward to show that for $k\geq1$,
$\Gamma[k]$ is the 
kernel of the single transfer map $\cE{0} B\Sigma_{p^k} \ra \cE{0}
B(\Sigma_{p^{k-1}})^p$.    Strickland
  \cite{strickland-morava-e-theory-of-symmetric} proves that cokernel
  of the cohomology transfer is a finite free $E_0$-module, so that
  there is an exact sequence
\[
E^0B(\Sigma_{p^{k-1}})^p \xra{\tau} E^0B\Sigma_{p^k} \xra{\pi} R_{p^k} \ra 0
\]
where $R_{p^k}$ is a finite free $E_0$-module.  Taking $E_0$-linear
duals gives the desired result.
\end{proof}

Observe that
$\Gamma[k]\circ\Gamma[\ell]\subseteq \Gamma[k+\ell]$, whence $\Gamma$
is a graded ring.  The unit $1\in \Gamma$ corresponds to the canonical
generator $\iota\in
\algapprox_1(E_*)$.   From now on we omit the ``$\circ$'' notation
when discussing the product on $\Gamma$.

The map $E_0\ra P$ defined by $r\mapsto r\cdot 1$ factors through a map
$\eta\colon E_0\ra \Gamma$, so that $\eta(r) = 
\pi(\algapprox(\lambda_r))$, where $\lambda_r\colon E_0\ra E_0$
denotes multiplication by $r\in E_0$.  The map $\eta \colon E_0\ra
\Gamma$ is a ring homomorphism.  Observe that 
the image of $\eta$ need not be central in $\Gamma$ (and in fact it
isn't, except when $k=\F_p$ and $n=1$.)

The ring $\Gamma$ becomes a $E_0$-bimodule by $\eta$, and
multiplication descends to a map $\mu\colon \Gamma_{E_0}\otimes
{}_{E_0}\Gamma\ra \Gamma$.  

Let $\epsilon^\times \colon P\ra E_0$ and $\Delta^\times \colon P\ra
P\otimes P$ be the maps isomorphic the morphisms of $\algapprox$-algebras
\[
\algapprox(E_0)\ra E_0, \quad \algapprox (E_0)\ra \algapprox(E_0\oplus
E_0)
\]
which correspond to the maps
\[
E_0\xra{\id} E_0,\quad E_0\xra{\sim} \algapprox_1(E_0)\otimes
\algapprox_1(E_0)\subset \algapprox (E_0)\otimes \algapprox(E_0) 
\]
of $E_0$-modules.
(In terms of topology, $\Delta^\times$ and $\epsilon^\times$ are
induced by the space-level diagonal and projection maps $B\Sigma_m\ra
B\Sigma_m\times B\Sigma_m$ and $B\Sigma_m\ra *$.)

\begin{prop}
The maps $\epsilon^\times$ and $\Delta^\times$ restrict to functions
\[ 
\epsilon \colon \Gamma\ra E_0,\qquad \Delta\colon \Gamma\ra
\Gamma_{E_0}\otimes \Gamma_{E_0},
\]
and the data $(\Gamma,\epsilon,\Delta,\eta,\mu)$ constitutes a graded twisted
commutative $E_0$-bialgebra.
\end{prop}
\begin{proof}
The map $\Gamma_{E_0}^{\otimes 2}\ra P^{\otimes 2}$ is an inclusion,
by \eqref{prop:gamma-summand-even}.  Its image
is the set of ``interlinear elements''; an element $f\in P\otimes P$
is interlinear if the induced operation $\bar{f}\colon A\times A\ra A$
on an $\algapprox$-algebra $A$ is additive in each variable, and if
$\bar{f}(xc,y)=\bar{f}(x,yc)$ for all $c\in E_0$; see \cite[Prop.\
10.2]{borger-wieland-plethystic-algebra} for a proof.  It is clear
that $\Delta^\times$ carries $\Gamma$ into the set of interlinear
elements, and thus we obtain the desired map $\Delta$.  The rest of
the argument is as in \cite{borger-wieland-plethystic-algebra}.
\end{proof}

Note that the diagonal map sends each graded piece to a single
grading; that is, $\Delta(\Gamma[k])\subseteq \Gamma[k]_{E_0}\otimes
\Gamma[k]_{E_0}$.  Thus, $\epsilon$ and $\Delta$ give $\Gamma[k]$ a
cocommutative coalgebra structure, dual to the ring structure on
$E^0B\Sigma_{p^k}/(\text{transfer})$.  Thus, $\Gamma\approx \bigoplus
\Gamma[k]$ is a \emph{graded} bialgebra, in the sense of
\S\ref{subsec:gr-bialg-affine-cat-scheme}. 

\begin{exam}
Let $G_0$ be the multiplicative group over a perfect field $k$ of
characteristic $p$, and
let $E$ be the corresponding Morava $E$-theory.  In this case,
$E$ is the extension of $p$-complete $K$-theory to the Witt ring
$\mathbb{W}_k$.  Let $a\mapsto a^\sigma$ be the lift to
$\mathbb{W}_k$ of the Frobenius automorphism of $k$.
In this case, $\Gamma$ is isomorphic to the ring $\mathbb{W}_k\langle
\psi^p \rangle$, where $\psi^p$  is the $p$th Adams operation (viewed
as a power operation on $E$-algebras), which satisfies the commutation
relation
\[
\psi^p\cdot a = a^\sigma \cdot \psi^p\qquad \text{for $a\in
  \mathbb{W}_k$}.
\]  
Observe that $\mathbb{W}_k$ is not central in $\Gamma$ (unless
$k\approx \F_p$.)  The grading on $\Gamma$ is such that
$\Gamma[0]=\mathbb{W}_k$, and $\psi^p\in \Gamma[1]$.
Since $\psi^p$ is a multiplicative operation, we have
\[
\epsilon(\psi^p) = 1,\qquad \Delta(\psi^p)=\psi^p\otimes \psi^p.
\]
\end{exam}

Let $\Alg{\Gamma}$ denote the category of $\Gamma$-algebras.  
We define a functor $U\colon \algcat\ra \Alg{\Gamma}$ by
\[
U(B) = \Hom_{\algcat}(P,B),
\]
equipped with the $\Gamma$-algebra structure induced by the action of
$\Gamma$ 
on $P$.
The
underlying $E_0$-module of $UB$ is just the underlying $E_0$-module of
$B$.  

\subsection{The $\Gamma$-module $\omega$}

Let $\omega$ be the object in $\Mod{\Gamma}$ defined by the reduced
cohomology 
of $\mathbb{CP}^1\approx S^2$:
\[
\omega =\Ker[ U(E^0(\mathbb{CP}^1))\ra U(E^0) ].
\]
It is free of rank $1$ as an $E_0$-module, and thus is a symmetric
object in $\Mod{\Gamma}$.  This is a slight abuse of notation; we
already use $\omega$ to represent the kernel of $E^*(\mathbb{CP}^1)\ra
E^*$ in the category of $E_*$-modules.

\section{Operation algebras for non-zero gradings}
\label{sec:gr-gamma}

In the previous section we constructed a bialgebra $\Gamma$, which
acts naturally on the even degree part of a $\algapprox$-algebra.  In
this section we describe what happens when we take odd gradings into
account.   
In the end, we will show that underlying an
$\algapprox$-algebra is an object of $\grMod{\Gamma}$, the 
$\Z/2$-graded $\omega$-twisted tensor category of modules over
$\Gamma$, in terms of the formalism of  \S\ref{sec:twisted-z2-cat}. 
We will define  
$\grAlg{\Gamma}$ to be the category of commutative ring objects in
$\grMod{\Gamma}$, and will obtain a ``forgetful'' functor
\[
U\colon \gralgcat\ra \grAlg{\Gamma}.
\]
The functor $U$ is plethyistic
\eqref{prop:algcat-to-gamma-alg-is-plethyistic}, in the sense of
\S\ref{subsec:plethyistic}. 

To obtain this result, we will first set up a $\Z$-graded theory.  In
particular, for each $q\in \Z$, we will
construct $\Gamma^q$ as a set of additive operations sitting inside
$P^q \approx \Hom_{\gralgcat}(\algapprox(\omega^{q/2}),
\algapprox(\omega^{q/2})) \approx \bigoplus \pi_{-q}L\Pfree_m
(\Sigma^{-q}E)$, analogously to the construction of 
$\Gamma$ in \S\ref{sec:additive-bialgebra}.   The ring $\Gamma^q$
naturally acts on 
$\pi_{-q}$ of a $K(n)$-local $E$-algebra.
The
collection of $\{\Gamma^q\}$ forms a kind of ``$\Z$-graded
bialgebra''.  In
particular, there will be coproduct maps 
\[
\Delta_{i,j}\colon \Gamma^{i+j}_{E_0}\ra \Gamma^i_{E_0}\otimes
\Gamma^j_{E_0}
\]
which determine tensor functors
\[
\otimes \colon \Mod{\Gamma^i}\times \Mod{\Gamma^j}\ra
\Mod{\Gamma^{i+j}}.
\]
Furthermore, there are  ``forgetful'' functors $\tilde{U}^i\colon \gralgcat
\ra \Mod{\Gamma^i}$.  These will fit together into a tensor functor
\[
\tilde{U}^*\colon \gralgcat \ra \zgrMod{\tilde{U}^*E_*}
\]
where the target is a category of ``$\Z$-graded
$\tilde{U}^*E_*$-modules'' built from the categories $\Mod{\Gamma^i}$
(see \S\ref{subsec:z-graded-gamma-modules}).  

Two features allow us to reinterpret this structure in terms of the
$\Z/2$-graded formalism of \S\ref{sec:twisted-z2-cat}.  
First, the
periodicity of the theory $E_*$ produces natural equivalences
$\Mod{\Gamma^i}\approx \Mod{\Gamma^{i+2}}$
\eqref{prop:periodicity-equivalence}.  Second, the ``suspension map''
defines an isomorphism $\Gamma^1\ra 
\Gamma^0$ of rings \eqref{prop:suspension-odd-to-even-iso}.
Using these observations, we reformulate the category
$\zgrMod{\tilde{U}^*E_*}$ in terms of the bialgebra $\Gamma$, by
producing (in 
\S\ref{subsec:z2-graded-gamma-modules}) an equivalence of 
tensor categories 
between $\zgrMod{\tilde{U}^*E_*}$ and the $\Z/2$-graded
$\omega$-twisted category $\grMod{\Gamma}$ of
$\Gamma$-modules.

\subsection{The rings $\Gamma^q$}

Let $P^q = \Hom_{\Mod{E_*}}(\omega^{q/2},
\algapprox(\omega^{q/2})) \approx
\Hom_{\gralgcat}(\algapprox(\omega^{q/2}), \algapprox(\omega^{q/2}))$.
As in \S\ref{subsec:additive-ops-degree-0}, there is a natural
identification of $P^q$ with 
the monoid of 
endomorphisms of the functor 
\[
A\mapsto \Hom_{\Mod{E_*}}(\omega^{q/2},A) \colon \gralgcat\ra \Set.
\]
We let $\Gamma^q\subset P^q$ denote the subset of elements which
correspond to additive endomorphisms.  Thus $\Gamma^q$ is isomorphic
to the ring of endomorphisms of $A\mapsto
\Hom_{\Mod{E_*}}(\omega^{q/2},A)\colon \gralgcat\ra \Ab$.

Define $\Delta^+\colon P^q\ra  
P^q\otimes P^q$ by postcomposition with $\gamma^{-1}\circ
\algapprox(\nabla)\colon \algapprox(\omega^{q/2})\ra
\algapprox(\omega^{q/2})\otimes \algapprox(\omega^{q/2})$.  As in
\S\ref{subsec:additive-bialgebra}, we see that $\Gamma^q$ corresponds
to the set of 
primitives in $P^q$ with respect to $\Delta^+$.
Observe that $\Gamma^0$ is just the ring $\Gamma$ of
\S\ref{sec:additive-bialgebra}.  

Furthermore, we have the following.
\begin{prop}
There are non-cannonical isomorphisms of rings $\Gamma^q\approx
\Gamma^{q+2k}$ for all $q,k\in \Z$.
\end{prop}
\begin{proof}
This is immediate from the definitions, and the fact that
$\omega^{q/2}\approx \omega^{(q+2k)/2}$ as $E_*$-modules.
\end{proof}

In terms of topology, $\Gamma^q=\bigoplus_{k\geq0} \Gamma^q[k]$, where
$\Gamma^q[k]\subseteq
\Hom_{\Mod{E_*}}(\omega^{q/2},\algapprox_{p^k}(\omega^{q/2}))$ can be
identified with $\Hom_{\Mod{E_*}}(\cE{0}S^{-q}, K^q[k])$, where
$K^q[k]$ fits in the exact
sequence
\[
0 \ra K^q[k] \ra \cE{*} B\Sigma_{p^k}^{-q\rho_{p^k}} \ra
\bigoplus_{0<j<p^k} \cE{*}B(\Sigma_j\times \Sigma_{p^k-j})^{-q\rho_{p^k}}
\]
defined using transfer, where $\rho_m$ denotes the real
($m$-dimensional) permutation representation of $\Sigma_m$.  In
particular, it is natural to  regard $\Gamma^q[k]$ as a submodule of
$\cE{-q}(B\Sigma_{p^k}^{-q\rho_{p^k}}) \approx
\cE{0}(B\Sigma_{p^k}^{-q\bar{\rho}_{p^k}})$, where $\bar{\rho}_{p^k}
\subset \rho_{p^k}$
denotes the reduced real permutation representation.

\begin{prop}\label{prop:odd-degree-summand}
For all $q\in \Z$, the sub-$E_0$-module $\Gamma^q[k]\subseteq
\Hom_{\Mod{E_*}}(\omega^{q/2},\algapprox_{p^k}(\omega^{q/2}))$ is a
direct summand, and thus in particular is a finite free $E_0$-module.
\end{prop}
\begin{proof}
(Compare \cite[Prop.\ 5.6]{strickland-morava-e-theory-of-symmetric}.)
We have already addressed the case of $q=0$
\eqref{prop:gamma-summand-even}, and thus need only consider $q=1$.
Let  $B_*\defeq\bigoplus_{m\geq0} \pi_* L\Pfree_m \Sigma E$; this is a
commutative 
Hopf algebra, with product given by the ring structure of $\Pfree
\Sigma E$,
and coproduct given by $\Pfree(\nabla)\colon \Pfree \Sigma E \ra
\Pfree(\Sigma E\vee \Sigma E)$. 
The result follows from the observation that $B_*$
 is a primitively
generated exterior algebra, since $B_*\approx
\algapprox(\omega^{1/2})$.  We prove this using the bar spectral 
sequence
\[
E_2=\Tor_*^{E_*}(A_*,A_*) \Longrightarrow B_*
\]
where $A_*\approx \bigoplus_{m\geq0}\pi_*L\Pfree_mE$.  By a
result of Strickland, following Kashiwabara \cite[Prop.\
5.1]{strickland-morava-e-theory-of-symmetric}, $A_*$ is a polynomial 
algebra, finite type with respect to the 
grading determined by $m$.  Thus $\Tor_*^{E_*}(A_*,A_*)$ is an
exterior algebra, with generators represented by the homology
suspensions of the generators of the polynomial ring $A_*$.  These
classes survive, since they detect the image of the suspension maps
$\pi_*L\Pfree_mE \ra \pi_{*+1}L\Pfree_m\Sigma E$, and thus the
spectral sequence collapses at $E_2$.  
\end{proof}

We will need the following result in \S\ref{sec:rational-algebras}.
\begin{prop}\label{prop:rational-prim-gen}
The inclusion of $\Gamma^q\subseteq \Hom_{\Mod{E_*}} (\omega^{q/2},
\algapprox(\omega^{q/2}))$ induces an isomorphism
\[
\Sym_{E_*}(\Gamma^q\otimes \omega^{q/2})\otimes \Q \approx
\algapprox(\omega^{q/2})\otimes \Q
\]
of graded $E_*$-modules.
\end{prop}
\begin{proof}
By \eqref{prop:odd-degree-summand}, $\Gamma^q\otimes \omega^{q/2}$ is
the primitives of the Hopf algebra $\algapprox(\omega^{q/2})$, and is
a direct summand.  Thus, $(\Gamma^q\otimes \omega^{q/2})\otimes \Q$ is
the primitives of $\algapprox(\omega^{q/2})\otimes \Q$.  The result
follows from the structure theory of graded Hopf algebras.
\end{proof}

\begin{rem}\label{rem:prim-ind-isos}
It may be helpful to consider the following picture.  Fix $k\geq1$,
and consider $\Prim_q\rightarrowtail
\cE{q}B\Sigma_{p^k}^{q\rho_{p^k}} \approx \pi_q L\Pfree_{p^k}\Sigma^q E
\twoheadrightarrow \Ind_q$, where $\Prim_q$ and
$\Ind_q$ denote the part of the primitives and indecomposables
of $\bigoplus_m \pi_*L\Pfree_m \Sigma^q E$ associated the $m=p^k$ summand in
dimension $q$; the ring $\Gamma^q$ is the direct sum of the
$E_0$-modules $\Prim_{-q}$ as $k$ varies.
  The suspension map $\pi_q L\Pfree_{p^k}\Sigma^q E \ra
\pi_{q+1}L\Pfree_{p^k} \Sigma^{q+1}E$ factors through $\Prim_q\ra
\Ind_{q+1}$; from these maps we obtain the following diagram, which is
$2$-periodic. 
\[
\xymatrix{
{\qquad}
& {\Prim_{-1}} \ar[d]_{\sim}
& {\Prim_0}  \ar[d]_f
& {\Prim_1} \ar[d]_{\sim}
& {\Prim_2} \ar[d]_f
& {\qquad}
\\
{\qquad} \ar[ur]^{\sim}
& {\Ind_{-1}} \ar[ur]^{\sim}
& {\Ind_0} \ar[ur]^{\sim}
& {\Ind_1} \ar[ur]^{\sim}
& {\Ind_2} \ar[ur]^{\sim}
}\]
The vertical maps $\Prim_q\ra \Ind_q$ are isomorphisms for odd $q$
using \eqref{prop:odd-degree-summand}.  The map $\Ind_{0}\ra
\Prim_{2}$ is an isomorphism by Theorems 8.5 and 8.6 of
\cite{strickland-morava-e-theory-of-symmetric}, where the result is
stated in ``dual'' form; specifically, in that
paper it is proved that the image of $\Prim E^0B\Sigma_{p^k} \ra \Ind
E^0B\Sigma_{p^k}$ is generated by the Euler class of
$\bar{\rho}_{p^k}^\C$.  The map $f$ is not an isomorphism, but is
a monomorphism with torsion cokernel; the present paper is a
essentially a meditation 
on $\cok f$.
\end{rem}

The bijection
$\Hom_{\gralgcat}(\algapprox(\omega^{q/2}),\algapprox(\omega^{q/2}))\ra
P^q$ induces an associative monoid structure on $P^q$, defined by
composition, and thus descends to a multiplication on $\Gamma^q$.
There is a ring homomorphism $\eta_q \colon E_0\ra \Gamma^q$, defined as
for $\Gamma$.

Finally, let $\Delta^\times_{i,j} \colon P^{i+j}\ra P^i\otimes P^j$ be the
map
induced by the ``multiplicativity map''
$\nu\colon \algapprox(\omega^{i/2}\otimes \omega^{j/2}) \ra 
\algapprox(\omega^{i/2}\oplus \omega^{j/2})$.  
As before, we have $\epsilon^\times \colon P^0\ra E_0$.

\begin{prop}\label{prop:gammastar-tensor}
The maps $\Delta^\times_{i,j}$ restrict to maps 
\[
\Delta_{i,j}\colon
\Gamma^{i+j}_{E_0} \ra \Gamma^i_{E_0}\otimes \Gamma^j_{E_0},
\]
which in turn induce functors
\[
\otimes \colon \Mod{\Gamma^i}\times \Mod{\Gamma^j} \ra
\Mod{\Gamma^{i+j}},
\]
which are unital, associative, and commutative in the sense that there
are natural isomorphisms
\[
E_0\otimes M \approx M \approx M\otimes E_0 \quad \text{of
  functors $\Mod{\Gamma^i}\ra \Mod{\Gamma^i}$,}
\]
\[
M_1\otimes (M_2\otimes M_3) \approx (M_1\otimes M_2)\otimes M_3,
\]
and 
\[M_1\otimes M_2\approx M_2\otimes M_1,
\]
where $E_0$ is regarded as a $\Gamma^0$-module.
(Note that the interchange map introduces a sign when applied to
elements of odd degree.)
\end{prop}

Define functors $\tilde{U}^q\colon \gralgcat \ra \Mod{\Gamma^q}$ by sending
$B\in \gralgcat$ to the right $\Gamma^q$-module
$\Hom_{\gralgcat}(\algapprox(\omega^{q/2}), B) \approx
\Hom_{\Mod{E_*}}(\omega^{q/2}, B)$.  

\begin{rem}
If $A$ is a $K(n)$-local commutative $E$-algebra, so that $\pi_*A$ is a
$\algapprox$-algebra, then tracing through the definitions reveals
that the underlying $E_0$-module of $\tilde{U}^q(\pi_*A)$ is canonically
identified with $\pi_{-q}A$.  Thus, we see that $\pi_{-q}A$ naturally
carries the structure of a $\Gamma^q$-module.
\end{rem}

\begin{prop}\label{prop:monoidal-tilde-u}
There are natural transformations $\tilde{U}^i(M)\otimes
\tilde{U}^j(N)\ra \tilde{U}^{i+j}(M\otimes N)$ of functors
$\gralgcat\times \gralgcat\ra 
\Mod{\Gamma^{i+j}}$.  
\end{prop}
\begin{proof}
Straightforward.
\end{proof}

\subsection{$\Z$-graded $\Gamma^*$-modules and
  $\tilde{U}^*E_*$-modules }
\label{subsec:z-graded-gamma-modules}

By a \dfn{$\Z$-graded $\Gamma^*$-module}, we mean a tuple $M=(M^i)_{i\in
  \Z}$, where $M^i$ is a right $\Gamma^i$-module.  We write
$\zgrMod{\Gamma^*}$ for the category of $\Z$-graded
$\Gamma^*$-modules.  This category admits the structure of an additive
symmetric monoidal category, by the tensor product functors of
\eqref{prop:gammastar-tensor}, so that $M\otimes N = (P^i)_{i\in \Z}$,
where $P^i=\bigoplus_q M^q\otimes N^{i-q}$.

Let $\tilde{U}^*\colon \gralgcat\ra \zgrMod{\Gamma^*}$ be the functor
which sends a $\algapprox$-algebra $B$ to the $\Z$-graded
$\Gamma^*$-module $(\tilde{U}^iB)$.  By \eqref{prop:monoidal-tilde-u},
the functor $\tilde{U}^*$ is a lax symmetric monoidal functor; i.e.,
there is a coherent natural transformation $\tilde{U}^*B\otimes
\tilde{U}^*C \ra \tilde{U}^*(B\otimes C)$.

By these remarks, it follows that $\tilde{U}^*E_*$ is a commutative
monoid object in $\zgrMod{\Gamma^*}$, and that $\tilde{U}^*B$ is
tautologically a $\tilde{U}^*E_*$-module object for all $B$ in
$\gralgcat$.  Let 
$\zgrMod{\tilde{U}^*E_*}$ denote the category of
$\tilde{U}^*E_*$-modules; it is a symmetric monoidal category, with
tensor product defined by
\[
M\otimes_{\tilde{U}^*E_*} N \defeq \cok\bigl[ M\otimes \tilde{U}^*E_* \otimes N
\rightrightarrows M\otimes N\bigr].
\]
There is an evident forgetful functor $\zgrMod{\tilde{U}^*E_*} \ra
\zgrMod{E_*}$ to the $\Z$-graded category of $E_*$-modules, and this
forgetful functor is strongly symmetric monoidal: the underlying  $\Z$-graded
$E_*$-module of 
$M\otimes_{\tilde{U}^*E_*} N$ is just the usual $\Z$-graded tensor
product of the underlying $E_*$-modules of $M$ and $N$.

Thus, the functor
$\tilde{U}^*\colon \gralgcat \ra \zgrMod{\Gamma^*}$ lifts
tautologically to a functor
\[
\tilde{U}^*\colon \gralgcat\ra \zgrMod{\tilde{U}^* E_*}.
\]
Furthermore, it is straightforward to check (by looking at what
happens on the underlying $E_*$-modules)
that the lifted
$\tilde{U}^*$ is a strong symmetric monoidal functor; i.e.,
$\tilde{U}^*B\otimes_{\tilde{U}^*E_*} \tilde{U}^*C\ra
\tilde{U}^*(B\otimes C)$ is an isomorphism.

The rest of this section is devoted to giving a more elementary
description of $\zgrMod{\tilde{U}^*E_*}$.  

\subsection{Periodicity}

Observe that the underlying
$E_0$-module of $\tilde{U}^{2i}E_*$ is canonically isomorphic to
$\pi_{-2i}E$.
\begin{prop}\label{prop:periodicity-equivalence}
The functor $\Mod{\Gamma^i}\ra \Mod{\Gamma^{i+2}}$ defined by
$M\mapsto M\otimes \tilde{U}^2E_*$ is an equivalence of categories, and there
are natural isomorphisms $\tilde{U}^i(M)\otimes \tilde{U}^2E_* \approx
\tilde{U}^{i+2}(M)$ of 
functors $\gralgcat\ra \Mod{\Gamma^{i+2}}$.
\end{prop}
\begin{proof}
The isomorphism $\tilde{U}^2(E_0)\otimes \tilde{U}^{-2}(E_0)\approx
\tilde{U}^0(E_0)\approx
E_0$ produces an inverse up to isomorphism for this functor.
\end{proof}


As a consequence, we note the
following.
\begin{lemma}\label{lemma:ue-mod-is-gamma-zero-and-one}
Let
\[
V'\colon \Mod{\Gamma^0}\times \Mod{\Gamma^1}\ra
\zgrMod{\tilde{U}^*E_*}
\]
be the functor which on objects sends $(M^0,M^1)$ to $N^i$, where
$N^{2i}= M^0\otimes \tilde{U}^{2i}E_*$ and $N^{2i+1}= M^1\otimes
\tilde{U}^{2i}E_*$, and $N$ is given a $\tilde{U}^*E_*$-module
structure in the evident way.  Then $V'$ is an equivalence of categories.
\end{lemma}

\subsection{Suspension}

Define functions $e_q\colon \Gamma^{q+i}\ra \Gamma^i$ for all $q\geq0$
and all $i\in \Z$ so that $e_q(f)$ is the composite
\[
\omega^{i/2} \approx \omega^{-q/2}\otimes \omega^{(q+i)/2}
\xra{\id\otimes f} \omega^{-q/2}\otimes \algapprox(\omega^{(q+i)/2})
\xra{E_q} \algapprox(\omega^{-q/2}\otimes \omega^{(q+i)/2})\approx
\algapprox (\omega^{i/2}),
\]
where $E_q$ is was defined in \S\ref{subsec:suspension-map}.

\begin{prop}
The maps $e_q\colon \Gamma^{i+q}\ra
\Gamma^{i}$ are homomorphisms of associative rings under $E_0$.
\end{prop}
\begin{proof}
This is a lengthy but straightforward diagram chase, which depends
essentially on \eqref{prop:susp-compat-with-algapprox}.
\end{proof}

\begin{prop}
The diagram
\[\xymatrix{
{\Gamma^{i+j+a+b}}  \ar[r]^-{\Delta_{i+a,j+b}} \ar[d]_{e_{i+j}}
& {\Gamma^{i+a}\otimes \Gamma^{j+b}} \ar[d]^{e_i\otimes e_j}
\\
{\Gamma^{a+b}} \ar[r]_-{\Delta_{a,b}}
& {\Gamma^a\otimes \Gamma^b}
}\]
commutes.
\end{prop}
\begin{proof}
This is a lengthy but straightforward diagram chase, which depends
essentially on \eqref{prop:susp-compat-multiplicativity}.
\end{proof}

Let $e_q^*\colon \Mod{\Gamma^i}\ra \Mod{\Gamma^{i+q}}$ denote the
functor obtained by restricting along the ring homomorphism $e_q$.

\begin{prop}\label{prop:coherence-e-susp}
There are natural and coherent isomorphisms 
\[
e_a^*(M)\otimes e_b^*(N) \ra e_{a+b}^*(M\otimes N)
\]
of functors $\Mod{\Gamma^{i}}\times \Mod{\Gamma^{j}}\ra
\Mod{\Gamma^{i+j+a+b}}$. 
\end{prop}

\begin{prop}\label{prop:suspension-odd-to-even-iso}
The suspension map $e_1\colon \Gamma^{1+2i}\ra \Gamma^{2i}$ is an
isomorphism for all $i\in \Z$.  Thus, the functor $e_1^*\colon
\Mod{\Gamma^{2i}}\ra \Mod{\Gamma^{1+2i}}$ is an equivalence of
categories. 
\end{prop}
\begin{proof}
By periodicity, this amounts to the observation that $e_1\colon \Gamma^{-1}\ra
\Gamma^{-2}$ is an isomorphism.  We read this off of the fact that the
suspension map $E^1B\Sigma^\rho_{p^k} \ra E^2B\Sigma^{2\rho}_{p^k}$ is
an isomorphism on primitives, as proved in Strickland and discussed in
\eqref{rem:prim-ind-isos}.
\end{proof}

\subsection{$\Z$-graded $\Gamma^*$-modules from spheres}

For $q\geq 0$, define a $\Z$-graded $\tilde{U}^*E_*$-module $\eta(q)$ by
\[
\eta(q) \defeq \Ker\bigl[ \tilde{U}^* (E^*S^{q}) \ra \tilde{U}^*
(E^*(\point)) \bigr].
\]
\begin{lemma}
There are natural and coherent isomorphisms
$\eta(q)\otimes_{\tilde{U}^*E_*} \eta(q')\approx \eta(q+q')$.
\end{lemma}
\begin{proof}
Read this off using the K\"unneth isomorphism $E^*(S^q\times
S^{q'})\approx E^*S^q\otimes E^*S^{q'}$.
\end{proof}
Clearly, $\eta(0)\approx \tilde{U}^*E_*$.  Observe that $\eta(2)$ is a
symmetric object of $\zgrMod{\tilde{U}^*E_*}$ (in the sense of
  \S\ref{subsec:symmetric-objects}), and that $\eta(1)$ is an odd
  square-root of $\eta(2)$. 

\begin{prop}\label{prop:suspension-op-and-eta}
There is an isomorphism $\eta(q)\approx e_q^*(\tilde{U}^*E_*)$ of
$\Z$-graded $\Gamma^*$-modules.
\end{prop}
\begin{proof}
This is a straightforward consequence of \eqref{cor:power-op-and-s2}.
\end{proof}

\begin{prop}\label{prop:suspension-op-is-tensor-with-eta}
There are natural isomorphisms $e_q^*M\approx
\eta(q)\otimes_{\tilde{U}^*E_*} M$ of $\Z$-graded $\tilde{U}^*E_*$-modules.  
\end{prop}
\begin{proof}
Clear from \eqref{prop:suspension-op-and-eta} and
\eqref{prop:suspension-op-is-tensor-with-eta}. 
\end{proof}

\subsection{$\Z/2$-graded $\Gamma$-modules}
\label{subsec:z2-graded-gamma-modules}

Let $\grMod{\Gamma}$ denote the category of $\Z/2$-graded
$\Gamma$-modules.  We give this category a tensor structure using the
procedure of \S\ref{sec:twisted-z2-cat}, where we twist by the
symmetric object $\omega\in \Alg{\Gamma}$ described above.

Let
\[
V\colon \Mod{\Gamma} \ra \zgrMod{\tilde{U}^*E_*}
\]
denote the functor which on objects sends $M$ to $N$, where
$N^{2i}=M\otimes \tilde{U}^{2i}E_*$ and $N^{2i+1}=0$.  It is clear
that $V$ is a strong symmetric monoidal functor.
Furthermore, it is clear from the definitions that $V(\omega)\approx \eta(2)$.

Using the method of \eqref{prop:functor-from-twisted-z2-graded}, we
obtain a symmetric monoidal 
functor
\[
V^*\colon \grMod{\Gamma}\ra \zgrMod{\tilde{U}^*E_*}
\]
which extends $V$ and sends $\omega^{1/2}$ to $\eta(1)$.  
\begin{prop}
The symmetric monoidal functor $V^*\colon \grMod{\Gamma} \ra
\zgrMod{\tilde{U}^*E_*}$ is an equivalence of categories.
\end{prop}
\begin{proof}
Explicitly, the functor $V^*$ sends an object $M=\{M^0,M^1\}$ of
$\grMod{\Gamma}$ to the object $V(M^0)\oplus V(M^1)\otimes \eta(1)$.
By \eqref{prop:suspension-op-is-tensor-with-eta} this functor is
naturally isomorphic to the $M\mapsto V(M^0)\oplus e_1^*V(M_1)$.
Since $e_1^*\colon \Mod{\Gamma^0}\ra \Mod{\Gamma^1}$ is an equivalence
of categories by \eqref{prop:suspension-odd-to-even-iso}, the result
follows using \eqref{lemma:ue-mod-is-gamma-zero-and-one}.
\end{proof}

\subsection{The forgetful functor $U$}

Since $V^*$ is an equivalence of tensor categories, there is a functor
$\tilde{U}\colon \gralgcat\ra \grMod{\Gamma}$ such that there is a
monoidal isomorphism
$V^*\tilde{U}\approx \tilde{U}^*$ of monoidal functors. 

\begin{prop}\label{prop:algcat-to-gamma-alg-is-plethyistic}
The  functor $\tilde{U}\colon \gralgcat\ra \grMod{\Gamma}$ lifts to a
plethyistic functor $U\colon 
\gralgcat\ra \grAlg{\Gamma}$.  
\end{prop}
\begin{proof}
In the sequence of forgetful functors
\[
\gralgcat \xra{U} \grAlg{\Gamma} \xra{U'} \Alg{E_*},
\]
the functor $U'$ is clearly plethyistic, and the composite $U'U$ is plethyistic
by \eqref{prop:algcat-to-E-alg-plethyistic}.  The result follows.
\end{proof}

\section{Rational algebras}
\label{sec:rational-algebras}

We say that an object of $\gralgcat$ (resp.\ an $E_*$-algebra) is
\dfn{rational} if the underlying $E_*$-module is a (graded) rational
vector space.   In this section, we prove that if $B$ is an
$\algapprox$-algebra, then $\Q\otimes B$ is also an
$\algapprox$-algebra in a canonical way; and also that a rational
$\algapprox$-algebra is the same thing as a rational $\Gamma$-algebra.
Our argument follows that of
Knutson \cite{knutson-lambda-rings-and-representation}.   Recall that
$U\colon \gralgcat \ra \Alg{E_*}$ admits a right adjoint $G$
(\S\ref{subsec:plethyistic}).  
\begin{lemma}
If $A\in \Alg{E_*}$ is rational, then so is $UG(A)$.
\end{lemma}
\begin{proof}
Let $i=0$ or $1$, and let $X(A)$ denote the $i$th
graded piece of $UG(A)$, regarded as an abelian group.  We have that
$$X(A) \approx \Mod{E_*}(\omega^{i/2}, UG(A)) \approx
\Alg{E_*}(\algapprox (\omega^{i/2}), A),$$
where the abelian group structure on $X(A)$ is determined by the
coproduct on $\algapprox(\omega^{i/2})$.  We need to show that $X(A)$
is rational, and our proof amounts to observing that $X$ is a
\emph{nilpotent}  abelian group scheme.  More precisely, let
$X_q(A)$ denote the 
quotient of the set $X(A)$ under the following equivalence relation;
we say that $f$ and $g$ in $X(A)$ are equivalent if their
restrictions to $\bigoplus_{k=0}^q \algapprox_k(\omega^{i/2})$ are equal.
Because the coproduct map $\Delta^+\colon \algapprox(\omega^{i/2})\ra
\algapprox(\omega^{i/2})\otimes \algapprox(\omega^{i/2})$ restricts to
a defined on this summand, $X_q(A)$ is in fact
a quotient abelian group of $X(A)$.
We have that $X_0(A)=0$, and 
$$X(A) =\lim_q X_q(A),$$
so that it suffices to show that each $X_q(A)$ is rational.  
Furthermore, we have 
$$\ker[ X_q(A)\ra X_{q-1}(A) ] = \Mod{E_*}(M_q, A),$$
where $M_q= \cok[ \bigoplus_{0<i<q} \algapprox_i(\omega^{i/2})\otimes
\algapprox_{q-i}(\omega^{i/2}) \ra \algapprox_q(\omega^{i/2}) ]$.
Since $A$ is rational, 
this kernel must be rational, and thus inductively we may conclude
that $X_q(A)$ is rational for all $q$.
\end{proof}

\begin{prop}\label{prop:initial-rational-object}
Let $B\in \gralgcat$.  There exists a morphism $B\ra B'\in \gralgcat$ which is
initial among morphisms from $B$ to rational objects of $\gralgcat$.
Furthermore, the evident map $\Q\otimes UB \ra UB'$ is an
isomorphism of rational $E_*$-algebras.
\end{prop}
\begin{proof}
Let $\psi\colon UB\ra (UG)(UB)$
denote the unit of the adjunction.  
We claim that there is a unique map $\psi'\colon \Q\otimes UB \ra
(UG)(\Q\otimes UB)$ making $\Q\otimes UB$ into a coalgebra for the 
comonad $UG$ such that $UB\ra \Q\otimes UB$ is a map of coalgebras.  
Consider the diagram
$$\xymatrix{
{UB} \ar[r]^{\psi} \ar[d]
& {(UG)(UB)} \ar[d]
\\
{\Q\otimes UB} \ar@{.>}[r]
& {(UG)(\Q\otimes UB)}
}$$
Since $(UG)(\Q\otimes UB)$ is a rational $E_*$-algebra by the above,
it is clear that there is a unique dotted arrow making the diagram commute.
\end{proof}

Let $\gralgcatQ\subset \gralgcat$ and $\grAlg{\Gamma,\Q}\subset
\grAlg{\Gamma}$ denote the full subcategories of rational objects.

\begin{prop}
Let $B$ be a rational object of $\grAlg{\Gamma}$.  Then there is a
unique $\algapprox$-algebra structure on $B$ compatible with the given
$\Gamma$-algebra structure.  
\end{prop}
\begin{proof}
It is enough to check this for $S\otimes \Q$, where $S\approx
\Sym_{E_0}(\Gamma\otimes \omega^{q/2})$ is the free $\Gamma$-algebra on one generator in
some degree $q/2$;
observe that $S\otimes \Q$ is clearly an object of $\grAlg{\Gamma}$.  
Let $P\approx \algapprox(\omega^{q/2})$ denote the free
$\algapprox$-algebra on 
one generator, and consider the tautological map $f\colon S\ra P$.  
By \eqref{prop:initial-rational-object}, we see that $P\otimes \Q$
is a $\algapprox$-algebra.  The result follows from the observation
that $f\otimes \Q\colon S\otimes \Q \ra P\otimes \Q$ is an
isomorphism \eqref{prop:rational-prim-gen}.
\end{proof}
\begin{cor}
The forgetful functor $U\colon \gralgcat\ra \grAlg{\Gamma}$ induces an
equivalence $\gralgcatQ\ra \grAlg{\Gamma,\Q}$.
\end{cor}

\section{Critical weights}
\label{sec:critical-weight}

Let $U\colon \mathcal{D}\ra \mathcal{A}$ be a plethyistic functor,
where $\mathcal{A}$ is the category of commutative monoid objects in
an abelian tensor category $\mathcal{C}$.  Let 
$T\colon \mathcal{C}\ra \mathcal{C}$ be the
monad associated to this situation.  Recall from
\S\ref{subsec:weight-decomp} that such a monad comes 
with maps $\eta\colon I\ra T$ and $\mu\colon TT\ra T$ defining the
monad structure, and natural maps $\iota\colon \Bbbk\ra T(M)$ and
$\delta\colon T(M)\otimes T(M)\ra T(M)$ defining the commutative ring
structure on $T(M)$.  We will assume that $T$ is equipped with a
weight decomposition, as in \S\ref{subsec:weight-decomp}.

Recall that a collection of morphisms $\{f_j\colon A_j\ra B\}$ in a
category called an
\dfn{epimorphic family} if for any two morphisms $g,h\colon B\ra C$,
$gf_j=hf_j$ for all $f_j$ implies $g=h$.
We say that $m\geq2$ is a \dfn{regular weight} of $T$ if the family
of maps in $\mathcal{C}$
\begin{equation}\label{eq:critical-map-family}
\delta\colon T_i(M)\otimes T_{m-i}(M)\ra T_m(M),\qquad \mu\colon
T_dT_{m/d}(M) \ra T_m(M)
\end{equation}
is an epimorphic family, where $i$ ranges over integers such that $0<i<m$ and
$d$ ranges over 
divisors of $m$ such that $1<d<m$.  In addition, we say that $0$ is a
regular weight if $\iota\colon \Bbbk\ra T_0(M)$ is an
epimorphism, and that $1$ is a regular weight if $\eta\colon M\ra
T_1(M)$ is an epimorphism.
We say that $m\geq0$ is a \dfn{critical weight} of $T$ if it is not a
regular weight.

The idea of the following proposition is that
$T$ is in some sense ``generated'' by phenomena in critical
weights. 
\begin{prop}\label{prop:t-algebra-structure-critical-wt}
Let $\psi\colon T(M)\ra M$ be a $T$-algebra structure on $M\in
\mathcal{C}$.  Let $N\subseteq M$ be a subobject of $M$ in $\mathcal{A}$.  If
$\psi(T_m(N))\subseteq N$ for each critical weight $m$, then there is
a unique $T$-algebra structure on $N$ such that the inclusion $N\ra M$
is a morphism of $T$-algebras.
\end{prop}
\begin{proof}
It suffices to show that the image of the composite $T(N)\ra
T(M)\xra{\psi} M$ is contained in $N$.  We will show by induction on
$m\geq0$ that $\alpha_m\colon T_m(N) \ra T_m(M)\xra{\psi} M$ factors
through $N$. 
When $m$ is a critical weight this is true by hypothesis.  If $m=0$ is
not critical, then $\iota\colon\Bbbk \ra T_0(M)$ is epimorphic, and thus
$\alpha_0$ factors through $N$ since $N$ is a subobject.  If $m=1$ and
is not critical, then $\eta\colon M\ra T_1(M)$ is epimorphic, and thus
$\alpha_1$ must factor through $N$ since $\psi\circ \eta=\id_M$.  If
$m\geq2$ and is not critical, then there is a unique dotted arrow
making the diagrams
$$\xymatrix{
{T_i(N)\otimes T_{m-i}(N)}
\ar[r]^-{\delta} 
\ar[d]_{\psi\otimes \psi} 
& {T_m(N)} \ar[r] \ar@{.>}[d]
& {T_m(M)} \ar[d]^{\psi}
\\
{N\otimes N} \ar[r]_-{\text{product}}
& {N} \ar@{>->}[r]
& {M}
}$$
and
$$\xymatrix{
{T_dT_{m/d}(M)} \ar[r]^-{\mu} \ar[d]_{T_d(\psi)}
& {T_m(N)} \ar[r] \ar@{.>}[d]
& {T_m(M)} \ar[d]^{\psi}
\\
{T_d(N)} \ar[r]_{\psi}
& {N} \ar@{>->}[r]
& {M}
}$$
commute.
\end{proof}

\section{Proof of the congruence criterion}
\label{sec:congruence-proof}

In this section we prove Theorem A, using the results of
\S\ref{sec:rational-algebras} and \S\ref{sec:critical-weight}. 

\subsection{The only critical weight of $\algapprox$ is $p$}

The proof of the following proposition owes something to McClure's
discussion in \cite[Ch.\ IX]{bmms-h-infinity-ring-spectra}.  
\begin{prop}\label{prop:only-critical-wt-is-p}
With respect to the standard weight decompositions, the only critical
weight for the algebraic approximation functor $\algapprox$ is $p$.
\end{prop}
\begin{proof}
We will see later in this  section exactly how $p$ is a critical weight.
Right now we prove that all other weights are regular.  

It is clear
that $0$ and $1$ are regular weights for $\algapprox$, since
$\iota\colon E_*\ra
\algapprox_0(M_*)$ 
and $\eta\colon M\ra \algapprox_1(M_*)$ are isomorphisms for all
$E_*$-modules.   

To show that $m$ is a regular weight,  it suffices to show that
the collection of maps of \eqref{eq:critical-map-family} is epimorphic
when $M_*$ is a finitely generated and free $E_*$-module, since
$\algapprox$ commutes with filtered colimits and reflexive
coequalizers.  Suppose that $M_*$ is a finite free $E_*$-module.
For $m\geq 2$ and $m\neq p$, there are two cases.

\textit{Case of $m$ relatively prime to $p$.}  In this case, let
$r\geq 0$ be
the largest integer such that $p^r< m$.  I claim that
\[
\delta\colon \algapprox_{p^r}(M_*)\otimes \algapprox_{m-p^r}(M_*) \ra
\algapprox_m(M_*)
\]
is surjective.

\textit{Case of $p|m$ and $m\neq p$.}  In this case, let $d=m/p>1$.  I
claim that 
\[
\mu\colon \algapprox_d \algapprox_p (M_*)\ra \algapprox_m(M_*)
\]
is surjective.

In either case, the claim proves regularity.  To prove these claims,
choose a finite free $E$-module $N$ such that $\pi_*N\approx M_*$.
Then the claims amount to showing that
\[
L\Pfree_{\Sigma_{p^r}\times \Sigma_{m-p^r}}N \ra L\Pfree_m N
\]
and
\[
L\Pfree_{\Sigma_d\wr \Sigma_p} N \ra L\Pfree_m N
\]
are surjective on homotopy groups; in fact, each of these maps admits
a section by \eqref{prop:transfer}, since $\Sigma_{p^r}\times
\Sigma_{m-p^r}\subset \Sigma_m$ (in the first case) and $\Sigma_d\wr
\Sigma_p \subset 
\Sigma_m$ (in the second case) have index prime to $p$. 
\end{proof}

\subsection{The Frobenius class}
\label{subsec:frobenius-class}

Let $I\subset E^0B\Sigma_{p^k}$ denote the kernel of the augmentation
map $i^*\colon E^0B\Sigma_{p^k}\ra E^0$, and let $J\subset E^0B\Sigma_{p^k}$
denote the ideal generated by the images of transfer maps from
subgroups $\Sigma_i\times \Sigma_{p^k-i}\subset \Sigma_{p^k}$, for
$0<i<p^k$. 
\begin{lemma}\label{lemma:augmentation-quotient}
For $k>0$, the image $i^*(J)$ of $J$ under the augmentation map is
$pE_0$, and $E^0B\Sigma_{p^k}/(I+J)\xra{\sim} E_0/(p)$.
\end{lemma}
\begin{proof}
We compute the projection of the transfer ideal along the augmentation
map.  Let $m\geq1$, and consider the fiber square 
\[\xymatrix{
{\Sigma_m/(\Sigma_i\times \Sigma_{m-i})} \ar[r] \ar[d]^j &
{B(\Sigma_i\times \Sigma_{m-i})} \ar[d]
\\
{*} \ar[r]_i & {B\Sigma_m.}
}\]
Applying the double-coset formula in $E$-cohomology shows that the
image of $E^0B(\Sigma_i\times \Sigma_{m-i}) \xra{\text{transfer}} E^0B\Sigma_m
\xra{i^*} E^0$ is precisely the ideal generated by the binomial
coefficient $\binom{m}{i}$.  The result follows from the observation
that $\gcd\set{\binom{p^k}{i}}{0<i<p^k}=p$ for $k>0$.  (Note that for
$m$ prime to $p$, $E^0B\Sigma_m/J\approx 0$ already.)
\end{proof}

\begin{prop}\label{prop:sigma-p-pullback}
The augmentation map $i^*\colon E^0B\Sigma_p\ra E^0$
restricts induces an isomorphism
$J\ra pE^0$, and induces a homomorphism $\sigma^* \colon
E^0B\Sigma_p/J \ra E^0/(p)$ 
on quotients.  Furthermore, 
the commutative square
\[
\xymatrix{
{E^0B\Sigma_p} \ar[r]^-{\pi^*} \ar[d]_{i^*}
& {E^0B\Sigma_p/J} \ar[d]^{\sigma^*}
\\
{E^0} \ar[r]_-{\pi^*}
& {E^0/(p)}
}\]
is a pullback square of $E^0$-modules.  
\end{prop}
\begin{proof}
Since $k=1$, the ideal $J$ is equal to the image in cohomology of the
stable transfer map $t\colon \Sip B\Sigma_p \ra \Sip B\{e\}$
associated to the inclusion of the trivial subgroup $\{e\}\subset
\Sigma_p$.  Thus $J$ is a free
$E_0$-module on the one generator $t^*(1)\in \rE^0B\Sigma_p$, and
therefore the surjective map $J\ra pE_0$ must be an isomorphism.
The existence of $\sigma^*$ is immediate
from \eqref{lemma:augmentation-quotient}, and the square is clearly a
pullback.  
\end{proof}

For any $E_*$-module $M_*$, there are maps 
\[
u\colon M_*\otimes{}_{E_*}\Gamma[1]\ra \algapprox_p(M_*),\qquad v\colon
\Sym^p_{E_*}(M_*)\ra \algapprox_p(M_*).
\]
The map $u$ is induced by the standard $\Gamma$-module structure
$\algapprox M_*\otimes \Gamma\ra \algapprox M_*$, which carries
$\algapprox_1 M_*$ into $\algapprox_p M_*$.  The map $v$ is induced by
the tautological commutative $E_*$-algebra map $\Sym^*_{E_0}(M_*) \ra \algapprox
(M_*)$.  We will be particularly be interested in the case when
$M_*=E_*$, so that all objects are concentrated in even degree and the
maps have the form
\[
u\colon \Gamma[1]\ra \algapprox_p(E_0),\qquad v\colon E_0\ra
\algapprox_p(E_0).
\]  
In this case, observe that $\hom_{E_0}(E^0B\Sigma_p, E_0)\approx
\algapprox_p(E_0)$, 
and that the inclusion $u\colon \Gamma[1]\ra \algapprox_p(E_0)$ is dual to the
projection $E^0B\Sigma_p\ra E^0B\Sigma_p/I$.  
Likewise, observe that $v$ is dual to the augmentation map $i^*\colon
E^0B\Sigma_p \ra 
E^0$.  

Let $\bar{\sigma}\in \Gamma[1]\otimes_{E_0} E_0/pE_0$ denote the element
corresponding via duality to the homomorphism $\sigma^*\colon
E^0B\Sigma_p/I\ra 
E^0/pE^0$ of \eqref{prop:sigma-p-pullback}.  We will call
$\bar{\sigma}$ the \dfn{Frobenius class}. 
Pick a representative $\sigma\in \Gamma[1]$ of this congruence class.

\begin{prop}\label{prop:p-critical-basic}
There is a pushout square in $E_0$-modules of the form
\[\xymatrix{
{E_0} \ar[r]^{\cdot p} \ar[d]_{(\sigma,-\id)}
& {E_0} \ar[d]^{\beta}
\\
{\Gamma[1]\oplus E_0} \ar[r]_{(u,v)}
& {\algapprox_p(E_0)}
}\]
\end{prop}
\begin{proof}
Consider the diagram
\[\xymatrix{
{0} \ar[r]
& {E^0B\Sigma_p} \ar[r]_-{(\pi^*,j^*)} \ar@{.>}[d]_{\beta^*}
& {E^0B\Sigma_p/I \oplus E^0} \ar[r]_-{(\sigma^*,-\pi^*)}
\ar[d]_{(\sigma^*,-\id)} 
& {E^0/pE^0} \ar[r]\ar@{=}[d]
& {0}
\\
{0} \ar[r]
& {E^0} \ar[r]_{\cdot p}
& {E^0} \ar[r]
& {E^0/pE^0} \ar[r]
& {0}
}\]
where the top row is obtained from the pullback square of
\eqref{prop:sigma-p-pullback}, and $\beta^*$ is the unique map making
the diagram commute.  The long exact sequence associated to
$\Ext_{E^0}^*({-},E^0)$ gives exact sequences
\[\xymatrix{
{0} \ar[r]
& {E_0} \ar[r]^{\cdot p} \ar[d]_{(\sigma,-\id)}
& {E_0} \ar[r] \ar[d]_{\beta}
& {\Ext_{E^0}^1(E^0/pE^0,E^0)} \ar[r] \ar@{=}[d]
& {0}
\\
{0} \ar[r]
& {\Gamma[1]\oplus E_0} \ar[r]_{(u,v)}
& {\algapprox_p (E_0)} \ar[r]
& {\Ext_{E^0}^1(E^0/pE^0,E^0)} \ar[r]
& {0}
}\]
\end{proof}

The corresponding odd degree result is simpler.
\begin{prop}\label{prop:p-noncritical-odd}
The inclusion $\Gamma[1]\otimes \omega^{q/2} \ra
\algapprox_p(\omega^{q/2})$ is an isomorphism when $q$ is odd.
\end{prop}
\begin{proof}
By periodicity, we may assume that $q=1$.  In this case, this amounts
to the claim that
\[
E^* B\Sigma_p^{\rho} \ra E^* B\Sigma_p^{\rho}/I
\]
is an isomorphism, where $I$ is the image of the transfer map, which
is immediate from \eqref{cor:pth-power-odd-degree}.
\end{proof}

\subsection{The case of weight $p$}

\begin{prop}\label{prop:p-critical-main}
Let $M_*$ be an $E_*$-module.  The image of the map
\[
\algapprox_p(M_*)\ra 
\algapprox_p(M_*)\otimes \Q
\]
is generated by the images of $u$, $v$, and by elements of the form
$(\sigma x- x^p)/p$ as $x$ ranges over elements of $M_0$.
\end{prop}
\begin{proof}
The functors $\algapprox_p$ and $\Sym^p$ commute with filtered
colimits and reflexive coequalizers, so it suffices to check the
result when $M_*$ is a finite free $E_*$-module.   The ``binomial
formula''
\[
\algapprox_p (M_*\oplus N_*) \approx \bigoplus_{i+j=p}
\algapprox_i(M_*)\otimes \algapprox_j(N_*),
\]
together with the fact that $u\colon \Sym^i(M_*)\ra \algapprox_i(M_*)$
is an isomorphism for $i\leq p$, allows us to reduce to the case that
$M_*$ is free on one generator in some degree, which amounts to
\eqref{prop:p-critical-basic} for even degree and
\eqref{prop:p-noncritical-odd} for odd degree.
\end{proof}

\subsection{Proof of Theorem A}

Say an object of $\gralgcat$ (resp.\ $\Alg{\Gamma}^*$) is \dfn{torsion
free} its  underlying $E_*$-module is torsion free as an abelian
group.  Say that an object $B$ of $\Alg{\Gamma}^*$ satisfies the
\dfn{congruence condition} if 
\[
\sigma x \equiv x^p \mod pB
\]
for all $x\in B^0$.

Let $U\colon \gralgcat\ra \Alg{\Gamma}^*$ denote the forgetful functor.
If $B\in \gralgcat$, then $U(B)$ satisfies the congruence condition;
this is the content of \eqref{prop:p-critical-basic}.  In fact, if
$x\in B^0$, represented by a map $f\colon \algapprox (E_0)\ra B$, then
$\sigma x- x^p\in B^0$ is represented by the composite
\[
E_0 \xra{(\sigma, -\id)} \Gamma[1]\oplus E_0 \xra{(u,v)}
\algapprox_p(E_0) \xra{f} B,
\]
which factors through multiplication by $p$ by
\eqref{prop:p-critical-basic}.

\begin{proof}[Proof of Theorem A]
By the above remarks, it suffices to show that an object $B$ of
$\grAlg{\Gamma}$ which is torsion free and which satisfies the
congruence condition admits a unique structure of
$\algapprox$-algebra.  

Let $B'=B\otimes \Q$; this is  an object of
$\grAlg{\Gamma,\Q}$ by \eqref{prop:initial-rational-object}, and thus
there is a unique $\algapprox$-algebra 
structure on $B'$ compatible with the $\Gamma$-algebra structure.
By \eqref{prop:t-algebra-structure-critical-wt} and
\eqref{prop:only-critical-wt-is-p}, it 
suffices to show that the dotted 
arrow exists in
\[
\xymatrix{
{\algapprox_p(B)} \ar[r] \ar@{.>}[d]
& {\algapprox_p(B')} \ar[d]^{\psi}
\\
{B} \ar@{>->}[r]
& {B'}
}\]
Because $B$ is a $\Gamma$-algebra, we know that
$\psi(u(B\otimes \Gamma[1]))\subseteq B$ and
$\psi(v(\Sym^p(B)))\subseteq B$.  The map $\algapprox_p (B) \ra B'$
factors uniquely through a map $\algapprox_p (B)\otimes \Q\ra B'$, and thus by 
\eqref{prop:p-critical-main} it suffices to check that 
$\psi((\sigma x-x^p)/p)\in B$ for all $x\in B_0$; this is precisely
the congruence condition.
\end{proof}

\section{Sheaves on categories of deformations}
\label{sec:deformations-of-frobenius}

Fix a prime $p$ and an integer $n\geq1$.  Let $k$ be a perfect field
of characteristic $p$, and let $G_0$ be a (one-dimensional,
commutative) formal group of height
$n$ defined over $k$.

\subsection{Deformations of $G_0$}

Let $\hRing$ denote the category whose objects are complete local rings
whose residue 
field has characteristic $p$, and whose morphisms are continuous local
homomorphisms.   If $R\in \hRing$, we write $\mathfrak{m}\subset R$ for
the maximal ideal, and $\pi\colon R\ra R/\mathfrak{m}$ for the
quotient map.

Let $R\in \hRing$.  A \dfn{deformation} of $G_0$ to $R$ is a triple
$(G,i,\alpha)$, consisting of 
\begin{enumerate}
\item [(a)] a formal group $G$ defined over $R$,
\item [(b)] a homomorphism $i\colon k\ra R/\mathfrak{m}$, and
\item [(c)] an isomorphism $\alpha\colon i^*G_0\ra \pi^*G$ of
  formal groups over $R/\mathfrak{m}$.
\end{enumerate}

\subsection{The Frobenius isogeny}

Suppose that $R\in \hRing$ has characteristic $p$.  We write
$\phi\colon R\ra R$ for the $p$th power homomorphism ($\phi(x)=x^p$).
For each formal group $G$ over such a ring $R$, the \dfn{Frobenius
  isogeny} $\Frob\colon G\ra \phi^*G$ is the homomorphism of formal
groups over $R$ induced by the relative Frobenius map on rings.  We
write $\Frob^r\colon G\ra (\phi^r)^*G$ for the 
homomorphism inductively defined by $\Frob^r= \phi^*(\Frob^{r-1})\circ
\Frob$.  

\subsection{The ``deformations of Frobenius'' category}

Let $(G,i,\alpha)$ and $(G',i',\alpha')$ be two deformations of
$G_0$ to $R$.  We say that a homomorphism $f\colon G\ra G'$ of
formal groups over $R$ is a \dfn{deformation of $\Frob^r$} if 
\begin{enumerate}
\item [(i)] $i\circ \phi^r=i'$, so that $(\phi^r)^*i^*G_0\approx
  (i')^*G_0$, and
\item [(ii)] the square
\[\xymatrix{
{i^*G_0} \ar[r]^{\Frob^r} \ar[d]_{\alpha}
& {(\phi^r)^*i^*G_0} \ar[d]^{\alpha'}
\\
{\pi^*G} \ar[r]^f
& {\pi^*G'}
}\]
of homomorphisms of formal groups over $R/\mathfrak{m}$ commutes.
\end{enumerate}

We let $\DefFrob_R$ denote the category whose objects are
deformations of 
$G_0$ to $R$, and whose morphisms are homomorphisms which are a
deformation of $\Frob^r$ for some (unique) $r\geq0$.  

\begin{rem}
Say that a morphism in $\DefFrob_R$ has a  \dfn{height} $r$ (where
$r\geq0$), if it is a deformation of $\Frob^r$.  The height of a
composition of morphisms is the sum of the heights.  For $G$, $G'$
objects of $\DefFrob_R$, let $\DefFrob^r_R(G,G')\subset
\DefFrob_R(G,G')$ denote the subset of the hom-set consisting of
morphisms of height $r$.  Note that $\DefFrob^0_R(G,G')$ are
precisely the isomorphisms from $G$ to $G'$ in $\DefFrob$, i.e., the
isomorphisms 
$G\ra G'$ of formal groups which are deformations of the identity map
of $G_0$.
\end{rem}

Observe that if $R$ is an $\F_p$-algebra, then for every object
$(G,i,\alpha)$ in $\DefFrob_R$ the Frobenius isogeny defines a
morphism $\Frob \colon (G,i,\alpha)\ra (\phi^*G, \phi i,
\phi^*\alpha)$ in this category.

Given $f\colon R\ra R'\in \hRing$, we define a functor $f^*\colon
\DefFrob_{R'}\ra \DefFrob_R$ by base change.  The base change functors
are coherent, in the sense that if $R\xra{f} R'\xra{g} R''$, then
there are natural isomorphisms $g^*f^*\approx (gf)^*$ satisfying the
evident coherence property.

\begin{rem}
We can summarize the above using $2$-categorical language: $\DefFrob$ is a
pseudofunctor from $\hRing$ to the $2$-category $\grCat$ of graded categories. 
\end{rem}

\subsection{Lubin-Tate theory}

According to the deformation theory of Lubin-Tate, there exists a
ring $L\in \hRing$ , and a deformation $(G_\univ,\id,\alpha_\univ)$
which is a deformation of 
$G_0$ to $L$, such that for each $G\in \DefFrob_R$ there exists a
unique map $f\colon L\ra R \in \hRing$ for which there is an
isomorphism $u\colon G\ra f^*G_\univ$ in $\DefFrob_R$, and that
furthermore the 
isomorphism $u$ is itself unique.  The ring $L$ has the form $L\approx
\mathbb{W}k\powser{u_1,\dots,u_{n-1}}$.  As we hope the reader is
already aware, $L$ is canonically identified with $\pi_0E$ for the
Morava $E$-theory 
spectrum associated to $G_0$.

Say that a groupoid $D$ is \dfn{thin} if between any two objects of
$D$ there is at most one isomorphism.  As a consequence of the Lubin-Tate
theorem, the groupoids $\DefFrob^0_R$ are thin.

Given a category $C$ and a subcategory $D\subseteq C$ which is a thin
groupoid and which contains all the objects of $C$, we can define a
quotient category $C/D$ as follows.  The object set $\ob (C/D)$ is
defined to be $\ob C/\sim$, where we say that $X\sim Y$ if there
exists $f\colon X\ra Y\in D$; morphisms in $C/D$ are defined by
$(C/D)([X],[Y])=C(X,Y)$.  That this becomes a category uses the fact
that $D$ is thin.  
There is an evident quotient functor $C\ra C/D$, which is an
equivalence of categories.

\subsection{Subgroups of formal groups}

Let $\Sub_R\defeq \DefFrob_R/\DefFrob^0_R$.  Thus the quotient
functors $\DefFrob_R\ra \Sub_R$ are equivalences of categories for all
$R$.  The assignment $R\mapsto \Sub_R$ is a functor $\hRing\ra \Cat$,
(i.e., not merely a pseudofunctor.)

\begin{prop}
There is a one-to-one
correspondence $\ob \Sub_R \approx \hRing(L,R)$.  
\end{prop}
\begin{proof}
This is Lubin-Tate deformation theory.
\end{proof}

For objects $G,G'\in
\Sub_R$, let $\Sub^r_R(G,G')\subset \Sub_R(G,G')$ denote the image of
$\DefFrob^r_R(G,G')$.  

\begin{prop}
The assignment $f\mapsto \ker f$ is a one-to-one correspondence between
$\coprod_{[G']\in \ob\Sub_R}\Sub^r_R(G,G')$ and the set of finite
subgroup schemes of $G$ 
which have rank $p^r$.
\end{prop}
Thus any pair $([G],K)$ consisting of an object $[G]$ of $\Sub_R$ and
a subgroup $K$ of $G$ of rank $p^r$ determines an object $[G']$ of
$\Sub_R$, which is represented by any deformation of Frobenius
$f\colon G\ra G'$ with $\ker f=K$; we'll write $[G/K]$ for the object $[G']$.

\begin{prop}
For each $r\geq0$ there is a complete local ring $L[r]$ such that
$\hRing(L[r],R)$ is in one to one correspondence with the set of pairs
$([G],K)$, where $[G]$ is an object of $\Sub_R$ and $K$ is a subgroup of
$G$ of degree $p^r$.  Under the continuous homomorphism $L\ra L[r]$
which classifies $([G],K)\mapsto [G]$, the ring $L[r]$ is finite and
free as an $L$-module.
\end{prop}
\begin{proof}
See \cite{strickland-finite-subgroups-of-formal-groups}.
\end{proof}
\begin{cor}
Let $\Sub^r_R=\coprod_{[G],[G']\in \ob \Sub_R}\Sub^r_R(G,G')$.  Then
$\Sub^r_R\approx 
\hRing(L[r], R)$.  
\end{cor}

\subsection{The formal graded category scheme $\Sub$}
\label{subsec:fml-graded-cat-scheme-sub}

Let $\can^*\colon L[r]\ra E_0/(p)$ be the map which classifies the kernel of
$\Frob^r \colon \pi^*G_\univ\ra (\phi^r)^*\pi^*G_\univ$, where
$\pi\colon E_0\ra E_0/(p)$.  Observe that if a deformation $g$ of
$\Frob^r$ is represented by some ring map $f\colon L[r]\ra R$, then $f$
factors through $\can^*$ if and only if $g$ is \emph{equal} to $\Frob^r$.

Let $\can_0^*\colon L[r]\ra k$ denote the
composite of $\can^*$ with $E_0/(p)\ra k$; it classifies the kernel
of $\Frob^r\colon G_0\ra (\phi^r)^*G_0$.  
The maps $\can^*$ and  $\can^*_0$ are
surjective, and the kernel of $\can^*_0$ is  precisely the maximal ideal
of $L[r]$.  Let $s^*,t^*\colon L[0]\ra L[r]$ be the maps in $\hRing$
which classify $([G],K) \mapsto [G]$ and $([G],K)\mapsto [G/K]$
respectively.  It is clear that the diagram
\[\xymatrix{
{L[0]} \ar[r]^{s^*} \ar[d]_{\can^*_0}
& {L[r]} \ar[d]_{\can^*_0}
& {L[0]} \ar[l]_{t^*} \ar[d]^{\can^*_0}
\\
{k} \ar[r]_\id & {k} & {k} \ar[l]^{\phi^r}
}\]
commutes.  Since $k$ is perfect, we see that the tensor product
$L[r]_s\otimes_{E_0}{}_t L[r']$ is also a complete local ring with
residue field $k$.  There are maps $c^*\colon L[r+r']\ra
L[r]_s\otimes_{E_0}{}_t L[r']$ and $i^*\colon L[0]\ra E_0$ in
$\hRing$, corresponding to composition and identity maps in $\Sub$;
from what we have already shown, it follows that $c^*$ and $i^*$ are
local homomorphisms, inducing isomorphisms on the residue fields.  In
fact, $i^*$ is an isomorphism.

Thus, $\Sub$ is what we might call a \dfn{formal affine category
  scheme}; it is represented by data
$\mathcal{L}=(E_0,L[r],s^*,t^*,i^*,c^*)$, which are objects and morphisms
in $\hRing$.

\subsection{Sheaves on the deformations category}

For $R\in \hRing$, let $\Mod{R}$ and $\Alg{R}$ denote the category of
$R$-modules and commutative $R$-algebras respectively.  If $f\colon
R\ra R'\in \hRing$, let $f^*\colon \Mod{R}\ra \Mod{R'}$ (or $f^*\colon
\Alg{R}\ra \Alg{R'}$) denote the evident basechange functors, defined
in either case by $f^*(M)\approx R'\otimes_R M$.  These functors are
coherent, in the sense that if $R\xra{f} R'\xra{g}R''$ then there are
natural isomorphisms $g^*f^*\approx (gf)^*$ satisfying the evident
coherence property.

A \dfn{quasi-coherent sheaf of $\O$-modules} on $\DefFrob$ consists of data
$M=(\{M_R\}, \{M_f\})$, where for each $R\in \hRing$ we have a functor
$M_R\colon \DefFrob_R^\op \ra \Mod{R}$, and for each $f\colon R\ra
R'\in \hRing$ we are given natural isomorphisms $M_f\colon
f^*M_R\xra{\sim} M_{R'}f^*$, such that for all $R\xra{f} R'\xra{g}
R''$, both ways of constructing a natural isomorphism $g^*f^*M_R\ra M_{R'}g^*f^*$ coincide, and for $\id_R\colon R\ra R$ the
natural map $M_{\id_R}\colon \id_R^*M_R\ra M_R\id_R^*$ is the identity
transformation.

A morphism $\gamma\colon M\ra N$ of quasi-coherent sheaves modules
consists of data $\{\gamma_R\}$, where $\gamma_R\colon M_R\ra N_R$ for
each $R\in \hRing$  is a natural transformation compatible with base
change, in the sense that if $f\colon R\ra R'$ in $\hRing$, then
$N_f(f^*\gamma_R) = (\gamma_Rf^*)M_f$.

Let $\ShMod$ denote the category of quasi-coherent sheaves of
$\O$-modules.  It is a tensor category, with tensor product defined
``objectwise'', so that $(M\otimes N)_R(G) = M_R(G)\otimes_R N_R(G)$.
The unit object is $\O$, defined so that $\O_R\colon \DefFrob_R\ra
\Mod{R}$ is the constant functor sending every object to $R$.  

In a similar way, we define a category $\ShAlg$ of \dfn{quasi-coherent
sheaves of $\O$-algebras}, by replacing ``$\Mod{R}$'' in the above
definitions with ``$\Alg{R}$''.  The category $\ShAlg$ is plainly
isomorphic to the category of commutative monoid objects in $\ShMod$.

It is clear from the definition that $\ShAlg$ and $\ShMod$ have all
small colimits and finite products.

\begin{rem}\label{rem:pseudofunc}
The above is easily summarized using $2$-categorical language:
$\DefFrob$, $\mathrm{Mod}$, and $\mathrm{Alg}$ are pseudofunctors
$\hRing\ra \Cat$, and $\ShMod$ and $\ShAlg$ are defined to 
be the categories of pseudonatural transformations and modifications between
pseudofunctors 
$\DefFrob\ra \mathrm{Mod}$ and $\DefFrob\ra \mathrm{Alg}$ respectively.
\end{rem}

\subsection{Sheaves are comodules}

Recall that for a graded affine category scheme such as defined by the
data $\mathcal{L}$,
there is an associated category of comodules $\Comod{\mathcal{L}}$
\eqref{subsec:gr-bialg-affine-cat-scheme}.  

\begin{prop}
There is an
equivalence of tensor categories between $\ShMod$ and
$\Comod{\mathcal{L}}$. 
\end{prop}
\begin{proof}
This is straightforward.  Given an $\mathcal{L}$-comodule
$(M,\psi\colon M\ra VM)$, we produce
a sheaf by the following procedure.
Let 
\[
\tilde{\psi}_r \colon L[r]{}_t\otimes_{E_0}M\ra
L[r]{}_s\otimes_{E_0}M
\]
denote the extension of $\psi_r \colon M\ra VM\ra
L[r]{}_s\otimes_{E_0}M$ to a map of $L[r]$-modules. 

On objects $h\colon
E_0\ra R$ of $\DefFrob_R$, the functor $M_R\colon \DefFrob_R\ra \Mod{R}$ is
given on objects of $\DefFrob_R$ by 
\[
M_R(h) = R{}_h\otimes_{E_0}M.
\]
On morphisms $g\colon L[r]\ra R$ of $\Pow_R$, the functor $M_R$ is
given by
\[
M_R(g\colon L[r]\ra R) = (R{}_{t^*g}\tensor{E_0}M \approx
R{}_g\tensor{L[r]} L[r]{}_t\tensor{E_0}M\xra{\tilde{\id\otimes\psi_r}}
R{}_g\tensor{L[r]} L[r]{}_s\tensor{E_0}M\approx
R{}_{s^*g}\tensor{E_0}M).
\]
Likewise, given an object $M$ of $\ShMod$, we obtain an
$\mathcal{L}$-comodule by evaluating $M$ at the the universal
deformation of $G_0$, and at the universal examples of deformations of
$\Frob^r$.  
\end{proof}
(Note that although $\mathcal{L}$ is really a \emph{formal} graded
category scheme, formality does not play a role here.)

\subsection{The sheaves $\FuncG$, $\omega$, and $\Z/2$-graded sheaves} 

If $G$ is a formal group over a local ring $R$, let $F_G$
denote the ring 
of functions on $G$; it is an $R$-algebra non-canonically isomorphic
to $R\powser{x}$.  We define $\FuncG\in \ShAlg$ by 
\[
\FuncG_R((G,i,\alpha)) \defeq F_G.
\]
The evident algebra maps $\O\ra \FuncG \xra{e} \O$ induce a direct sum
splitting $\FuncG\approx \O\oplus \mathcal{I}$ as modules.  Let 
\[
\omega \defeq \mathcal{I}/\mathcal{I}^2
\]
as an object of $\ShMod$; observe that the $R$-modules
$\omega_R((G,i,\alpha))$ are non-canonically isomorphic to $R$, so
that $\omega$ is a symmetric object in the sense of
\S\ref{subsec:symmetric-objects}. 

Let $\grShMod$ denote the $\Z/2$-graded $\omega$-twisted tensor
category (\S\ref{sec:twisted-z2-cat}) associated to $\ShMod$.  Let
$\grShAlg$ denote the category 
of commutative monoid objects in $\grShMod$.

\subsection{Frobenius congruence}
\label{subsec:frobenius-cong}

Suppose that $R\in \hRing $ has characteristic $p$.  The $p$th power
map $\phi$ defines a functor $\phi^*\colon \Alg{R}\ra \Alg{R}$ by base
change, together with a natural transformation $\Frob\colon \phi^*\ra
\id$, where $\Frob\colon \phi^*A\ra A$ is the relative Frobenius
defined  as the 
unique map making the diagram
\[\xymatrix{
{R} \ar[r]^\phi \ar[d] 
& {R} \ar@{=}[r] \ar[d]
& {R} \ar[d]
\\
{A} \ar[r]  \ar@/_1pc/[rr]_{\phi}
& {\phi^*A} \ar[r]^{\Frob}
& {A}
}\]
commute.

Given an object $B\in \ShAlg$, we say it satisfies the \dfn{Frobenius
  congruence} if for all $R\in \hRing$ of characteristic $p$, and for
all $G\in \DefFrob_R$, the diagram
\[\xymatrix{
{\phi^* B_R(G)} \ar[r]^{B_\phi}_{\sim} \ar[dr]_{\Frob}
& {B_R(\phi^*G)} \ar[d]^{B_R(\Frob)}
\\
& {B_R(G)}
}\]
commutes.  Roughly speaking, the Frobenius congruence condition says that $B$
carries the relative Frobenius on formal groups to the relative
Frobenius on algebras.  

\begin{exam}
Both $\O$ and $\FuncG$ satisfy the Frobenius congruence.
\end{exam}

Let $\ShAlgCong\subset \ShAlg$ denote the full subcategory consisting
of sheaves which satisfy the Frobenius congruence.  Let
$\grShAlgCong\subset \grShAlg$ denote the full subcategory consisting
of sheaves whose even degree part is in $\ShAlgCong$.

\begin{prop}\label{prop:frobenius-congruence-comodule}
Let $A$ be an object of $\ShAlg$, and let $(B,\psi\colon B\ra CB)$ be
the corresponding commutative monoid object in $\Comod{\mathcal{A}}$.
Then $A$ satisfies the Frobenius congruence if and only if the
composite map 
\[
B \xra{\psi_1} L[1]{}_s\otimes_{E_0} B \xra{\can^*\otimes\id}
(E_0/(p)){}_s\otimes_{E_0} B\approx B/pB
\]
is equal to $x\mapsto x^p$.
\end{prop}
\begin{proof}
This amounts to checking the congruence condition on the universal
example. 
\end{proof}

\section{Bialgebra rings as sheaves on the deformation category} 
\label{sec:bialgebra-and-sheaves}

We retain the notation of the previous section.

\subsection{The formal graded category scheme $\Pow$}

The discussion of \S\ref{subsec:gr-bialg-affine-cat-scheme} applies to
our twisted $L$-bialgebra $\Gamma$ of power operations defined in
\S\ref{sec:additive-bialgebra}.  Thus, let 
$\mathcal{A}= (E_0, A[r],s^*,t^*,i^*,c^*)$ be the data obtained by
``dualizing'' 
$\Gamma$, so that $A[r]\approx H_{\Gamma[r]}(E_0)\approx
E^0B\Sigma_{p^r}/(\text{transfer})$, and in particular $A[0]\approx E_0$.
Our first order of business is to show that, like $\mathcal{L}$, the 
data $\mathcal{A}$ consists of objects and morphisms of $\hRing$, and
thus defines a 
formal graded category scheme
\eqref{subsec:fml-graded-cat-scheme-sub}, which we will call $\Pow$. 

For $r>0$, let $\bA[r]$ denote the quotient ring of $A[r]$ defined by 
\[
\bA[r] \defeq E^0B\Sigma_{p^r}/((\text{transfer})+(\text{aug})),
\]
where 
$(\text{aug})$ denotes the kernel of the augmentation map
$E^0B\Sigma_{p^r} \ra E^0$ induced by inclusion of the base point
$*\ra B\Sigma^{p^r}$.  Additionally, we set $\bA[0]\defeq E_0/pE_0$.  
By \eqref{lemma:augmentation-quotient}, the augmentation quotient
induces an isomorphism $\bA[r]\ra E_0/(p)$.
\begin{lemma}\label{lemma:local-rings}
The rings $A[r]$ are local rings, with maximal ideal given by the
kernel of the composite $A[r]\ra \bA[r]\approx E_0/pE_0 \ra k$.
\end{lemma}
\begin{proof}
The ring $A[r]$ is finite as an $E_0$-module (using $s^*$), and thus every
maximal ideal of $A[r]$ must contract to the maximal ideal of $E_0$ by
the ``going up theorem''.  Thus it suffices to show that
$A[r]\otimes_{E_0}k$ is a local ring.  Since $A[r]$ is a quotient of
$E^0B\Sigma_{p^r}$, it suffices to show that 
$E^0B\Sigma_{p^r}\otimes_{E_0}k$ is a local ring.  Let $K$ denote the
generalized cohomology theory obtained by killing the maximal ideal of
$E$, so that $E^0B\Sigma_{p^r}\otimes_{E_0}k \approx
K^0B\Sigma_{p^r}$; the result follows from the observation that the
augmentation ideal of $K^0B\Sigma_{p^r}$ is nilpotent, since
$B\Sigma_{p^r}$ is connected of finite type.
\end{proof}

\begin{lemma}\label{lemma:modulo-abar}
Let $B$ be an object in $\algcat$.  Then the composite
\[
B_0 \ra \Hom_{E_*}(\Gamma[r], B) \approx B\otimes_{E_0}{}_s A[r] \ra
B\otimes_{E_0}{}_s \bA[r] \approx B/pB
\]
sends $x\mapsto x^{p^r}$.
\end{lemma}
\begin{proof}
Immediate using \eqref{lemma:power-op-is-power-mod-aug} and
\eqref{lemma:augmentation-quotient}. 
\end{proof}

Let $\can^*\colon A[r]\ra \bA[r]\approx k$ denote the evident
projection.  
The diagram
\[\xymatrix{
{A[0]} \ar[r]^{s^*} \ar[d]_{\can^*}
& {A[r]} \ar[d]_{\can^*}
& {A[0]} \ar[l]_{t^*} \ar[d]^{\can^*}
\\
{k} \ar[r]_\id & {k} & {k} \ar[l]^{\phi^r}
}\]
commutes, because $s^*$ is induced by the ``standard'' inclusion
$E_0\ra E^0B\Sigma_{p^k}$, while $t^*$ is induced by the structure map
$E_0\ra \Hom_{E_0}(\Gamma[r],E_0)$ which defines the canonical
$\Gamma$-module structure on the ground ring $E_0$, and thus is the
$p^r$th power map on quotients by \eqref{lemma:modulo-abar}.  Thus the
rings 
$A[r]{}_s\otimes_{A[0]}A[r']$ are complete 
local with residue field $k$, and the homomorphisms $s^*$, $t^*$,
$c^*$, $i^*$ are all local homomorphisms.  

Thus, $\Pow$ is a formal
graded category scheme.  Observe that $\Pow^0_R$ is a thin groupoid
for all $R\in \hRing$, and that $\Pow\colon \hRing\ra \grCat$ is a
functor (not just a pseudofunctor).

\subsection{Isomorphism between $\Sub$ and $\Pow$}

We are going to construct a pseudonatural transformation $F\colon
\Pow\ra \DefFrob$ of pseudofunctors $\hRing\ra \grCat$, 
and we will show, by applying a theorem of Strickland, that
$F_R\colon \Pow_R\ra \DefFrob_R$ is an equivalence of graded categories
for each $R\in \hRing$.  This immediately implies the following.
\begin{prop}
The natural transformation 
\[
F'\colon \Pow\xra{F} \DefFrob\ra \Sub
\]
of functors $\hRing\ra \grCat$ induces an isomorphism $\Pow_R\ra
\Sub_R$ of graded categories for all $R$ in $\hRing$.
\end{prop}

Observe that an object $M\in \Mod{\Gamma}$ determines functors
$M_R\colon \Pow_R\ra \Mod{R}$, as follows.  Let $\psi\colon M\ra
H_{\Gamma[r]}(M) \approx A[r]{}_s\otimes_{E_0} M$ denote the structure
map of the $\Gamma$-module $M$; it is a map of $E_0$-modules, using
the module structure on $A[r]$ given by the ring homomorphism
$t^*\colon E_0\ra A[r]$.  
If $X$ is a space, then $E^0X$ is an object of $\Alg{\Gamma}$, and
thus we get a functor $(E^0X)_R\colon \Pow_R\ra \Alg(R)$.

The next observation is that if we take $X=\CP^\infty$, then we get a
functor from $\Pow_R$ to $\DefFrob_R$.  More precisely, for
each $f\colon E_0\ra R$, the
projective system $\{R\otimes_{E_0}E^0\CP^m\}$  determines a
formal scheme 
over $R$, and the maps $\{R\otimes_{E_0} E^0\CP^{i+j}\ra
R\otimes_{E_0}E^0(\CP^i\times \CP^j)\}$ 
give it the structure of a formal group over $R$, which we will call
$f^*G_\univ$.  Let $f_0\colon k\ra R/\mathfrak{m}$ be the
homomorphism obtained from $f$ by passing to residue fields.  Then
the triple $F(f)=(f^*G_\univ, f_0, \id\colon f_0^*G_0\ra
\pi^*f^*G_\univ)$ is a deformation of $G_0$ to $R$.    

A map $g\colon A[r]\ra R$ determines a homomorphism of formal groups
$(s^*g)^*G_\univ\ra (t^*g)^*G_\univ$ over $R$.
\begin{lemma}
This homomorphism is a deformation of $\Frob^r$; i.e., it determines a
morphism $F(s^*g)\ra F(t^*g)$ in $\DefFrob^r_R$.
\end{lemma}
\begin{proof}
It suffices to check the universal example, corresponding to the
identity map of $A[r]$.  In this case, it is enough to note that the
composite
\[
E^0\CP^\infty \xra{\psi} H_{\Gamma[k]}(E^0\CP^\infty) \ra
E^0\CP^\infty\otimes_{E_0}E_0/(p) \ra E^0\CP^\infty\otimes_{E_0}k
\]
is the map $x\mapsto x^{p^r}$, and that this composite is the same as 
\[
E^0\CP^\infty \ra A[r]{}_s\otimes_{E_0}E^0\CP^\infty \ra
\bA[r]{}_s\otimes_{E_0} E^0\CP^\infty.
\]
\end{proof}

Thus we have obtained functors $F\colon \Pow_R\ra \DefFrob_R$, which
are clearly natural in $R$, and thus a functor $F'\colon \Pow_R\ra
\Sub_R$.  By the Yoneda lemma, the functor $F'$ determines and is
determined by a 
collection of homomorphisms $F^*\colon L[r]\ra A[r]$ in $\hRing$.

It remains to show that $F'$ is an isomorphism of graded
categories.  It is clear that $F'$ is a bijection on objects.  To show
that $F$ is a bijection on morphism, it suffices to check that
$F^*\colon L[r]\ra A[r]$ is an isomorphism for each $r\geq0$.  By the
definition of 
$L[r]$, this map classifies a pair $([G],K)$, where $G=F(s^*\colon
E_0\ra A[r])$ and $K$ is the kernel of the isogeny
$F(\id_{A[r]})\colon F(s^*\colon E_0\ra A[r])\ra F(t^*\colon E_0\ra
A[r])$.   The kernel $K$ is a closed subscheme of $G$; its function
ring $\mathcal{O}_K$ is the pushout of $E_0$-algebras
\[\xymatrix{
{E^0\CP^\infty} \ar[r]^-{\bP_{p^r}}  \ar[d]_{i^*}
& {E^0(\CP^\infty\times B\Sigma_{p^r})/(\text{tr})} \ar[d]
\\
{E^0} \ar[r]
& {\mathcal{O}_K}
}\]
\begin{lemma}
Let $\rho_{p^r}^\C$ denote the complex vector bundle over
$B\Sigma_{p^r}$ associated to the permutation representation, and let
$\pi\colon \mathbb{P}(\rho_{p^r}\otimes \C)\ra B\Sigma_{p^r}$ denote the
projective bundle associated to this vector bundle, with tautological
line bundle $L$ classified by $\ell\colon \mathbb{P}(\rho_{p^r}^\C)\ra
\CP^\infty$.  Then there is a pushout square of $E_0$-algebras 
of the form
\[\xymatrix{
{E^0\CP^\infty} \ar[r]^-{\bP_{p^r}}  \ar[d]_{i^*}
& {E^0(\CP^\infty\times B\Sigma_{p^r})/(\text{tr})} \ar[d]^{(\ell,\pi)^*}
\\
{E^0} \ar[r]
& {E^0\mathbb{P}(\rho_{p^r}^\C)/(\text{tr})}
}\]
\end{lemma}
\begin{proof}
Examine the diagram
\begin{equation}\label{eq:kernel-power-op-calc}
\vcenter{
\xymatrix{
{\rE^0(\CP^\infty)^L} \ar[r]^-{P_m} \ar[d]_{z^*}
& {\rE^0(\CP^\infty\times B\Sigma_m)^{L\boxtimes \rho_m^\C}}  \ar[r] \ar[d]_{z^*}
& {\rE^0(\CP^\infty\times B\Sigma_m)^{L\boxtimes \rho_m^\C}/I'} \ar[d]
\\ 
{E^0\CP^\infty} \ar[r]^-{P_m} \ar[d]_{\pi^*}
& {E^0 \CP^\infty\times B\Sigma_m} \ar[r] \ar[d]_{\pi^*}
& {E^0 \CP^\infty \times B\Sigma_m/I} \ar[d]_{\bar{\pi}^*}
\\
{E^0 S(L)} \ar[d] \ar@{.>}@/_2pc/[rr]
& {E^0S(L\boxtimes \rho_m^\C)} \ar[r] \ar[d]
& {E^0S(L\boxtimes \rho_m^\C)/I''} \ar[d]
\\
{0}
& {0}
& {0}
}
}
\end{equation}
The  left-hand and middle columns of vertical maps of
\eqref{eq:kernel-power-op-calc} arise from the
cofiber sequence coming from the pair $(D(V),S(V))$ associated to the
Thom space $X^V$; they are exact, since the Euler class of any complex
bundle of the form $L\boxtimes V\ra \CP^\infty\times X$ is a non-zero divisor.
The right-hand vertical column is obtained
by quotienting the middle column by transfer ideals, denoted $I$,
$I'$, and $I''$.  We claim that the right-hand column of
\eqref{eq:kernel-power-op-calc} is also exact,
which amounts to showing that $\pi^* I=I''$, which in turns amounts to
the observation that the diagrams
\[\xymatrix{
{E^0(\CP^\infty \times B\Sigma_i\times B\Sigma_{m-i})}
\ar[r]^-{\text{transfer}} \ar[d]_{\pi_i^*}
& {E^0\CP^\infty \times B\Sigma_m} \ar[d]^{\pi^*}
\\
{E^0S(L\boxtimes \rho_i^\C\boxtimes \rho_{m-i}^\C)} \ar[r]_-{\text{transfer}}
& {E^0S(L\boxtimes \rho_m^\C)}
}\]
commute, and that the map marked $\pi_i^*$ is also surjective.  
(Note
that if $m$ is not a $p$th power, the right-hand column of
\eqref{eq:kernel-power-op-calc} is identically
$0$.) 

The two long
horizontal composites in \eqref{eq:kernel-power-op-calc} are ring
homomorphisms, which we denote $\bP_m$.

Choose any complex orientation for $E$, let $u\in \rE^0(\CP^\infty)^L$
denote the Thom class of $L$, and let $x=z^*(u)\in E^0\CP^\infty$
denote the associated Euler class, so that $E^0S(L)\approx
E^0\CP^\infty/(x)$. 

We have that $i^*P_m(x)=i^*z^*P_m(u) =0$, and therefore there is a
ring homomorphism indicated by the dotted arrow in
\eqref{eq:kernel-power-op-calc} making the square
\[\xymatrix{
{E^0\CP^\infty} \ar[r]^-{\bP_m} \ar[d]_{\pi^*}
& {E^0 \CP^\infty \times B\Sigma_m/I}  \ar[d]_{\bar{\pi}^*}
\\
{E^0 S(L)}  \ar[r]
& {E^0S(L\boxtimes \rho_m^\C)/I''}
}\]
commute.  Since the class $P_m(u)$ is a Thom class for $L\boxtimes
\rho_m^\C$, the kernel of $\pi^*$ is a cyclic ideal generated by
$P_m(x)=z^*P_m(u)$, and thus the kernel of $\bar{\pi}^*$ is also a
cyclic ideal generated by $\bP_m(x)$.  Therefore we conclude that the
square is a pushout square in rings; since $S(L)\approx *$ and
$S(L\boxtimes \rho_m^\C)\approx \mathbb{P}(\rho_m^\C)$, we obtain the
desired result.
\end{proof}

\begin{thm}[Strickland, \cite{strickland-morava-e-theory-of-symmetric}]
The map $L[r]\ra A[r]\approx E^0B\Sigma_{p^r}/(\text{tr})$ classifying
the subgroup $K$ of $G$ described above is an isomorphism.
\end{thm}

As an immediate corollary, the maps $F^*\colon L[r]\ra A[r]$ are
isomorphisms as desired.

\subsection{Equivalence of $\ShMod$ and $\Mod{\Gamma}$, and proof of
  Theorem B}

We have equivalences of tensor categories
\[
\ShMod \approx \Comod{\mathcal{L}}\approx \Comod{\mathcal{A}}\approx
\Mod{\Gamma},
\]
which clearly induces an equivalence $\ShAlg \approx \Alg{\Gamma}$.
The object $\omega$ of  $\ShMod$ defined above corresponds to the
$\Gamma$-module $\omega$ defined above, and thus we obtain an
equivalence of tensor categories $\grShAlg\approx \grAlg{\Gamma}$.
Theorem B  follows immediately, with the statement about congruence
conditions proved using \eqref{prop:frobenius-congruence-comodule}.


\end{document}